\documentclass[leqno]{amsart}
\usepackage{amssymb}
\usepackage{mathrsfs}
\usepackage{amsmath, amsfonts, vmargin, enumerate}
\usepackage{graphics}
\usepackage{verbatim}
\usepackage{amsthm}
\usepackage{latexsym,bm}
\usepackage{euscript}
\usepackage{dsfont}

\input xypic
\xyoption {all}

 \makeatletter

\newtheorem{thm}{Theorem}[section]
\newtheorem{prop}[thm]{Proposition}
\newtheorem{cor}[thm]{Corollary}
\newtheorem{lem}[thm]{Lemma}
\newtheorem{defi}[thm]{Definition}
\newtheorem{remark}[thm]{Remark}
\newtheorem{example}[thm]{Example}
\newtheorem{pb}[thm]{Problem}

\newenvironment{rk}{\begin{remark}\rm}{\end{remark}}

\newenvironment{problem}{\begin{pb}\rm}{\end{pb}}

\numberwithin{equation}{section}

\newcommand{\real}{{\mathbb R}}

\newcommand{\ent}{{\mathbb Z}}
\newcommand{\com}{{\mathbb C}}
\newcommand{\un}{{\mathds {1}}}

\newcommand{\T}{{\mathbb T}}

\newcommand{\A}{{\mathcal A}}
\newcommand{\B}{{\mathcal B}}
\newcommand{\E}{{\mathcal E}}

\newcommand{\I}{{\mathbb I}}
\newcommand{\M}{{\mathcal M}}
\newcommand{\N}{{\mathcal N}}

\newcommand{\Q}{{\mathbb Q}}

\renewcommand{\a}{\alpha}
\renewcommand{\b}{\beta}

\newcommand{\Ga}{\Gamma}
\newcommand{\D}{\Delta}

\newcommand{\e}{\varepsilon}
\newcommand{\f}{\varphi}
\newcommand{\La}{\Lambda}
\renewcommand{\l}{\lambda}

\newcommand{\ot}{\otimes}

\newcommand{\8}{\infty}
\newcommand{\el}{\ell}

\newcommand{\la}{\langle}
\newcommand{\ra}{\rangle}
\newcommand{\wt}{\widetilde}
\newcommand{\wh}{\widehat}
\newcommand{\n}{\noindent}

\newcommand{\be}{\begin{eqnarray*}}
\newcommand{\ee}{\end{eqnarray*}}
\newcommand{\beq}{\begin{equation}}
\newcommand{\eeq}{\end{equation}}
\newcommand{\beqn}{\begin{equation*}}
\newcommand{\eeqn}{\end{equation*}}

\begin{document}

\title{Harmonic analysis on quantum tori}


\author{Zeqian Chen}
\address{Wuhan Institute of Physics and Mathematics, Chinese Academy of
Sciences, 30 West Strict, Xiao-Hong-Shan, Wuhan
430071, China}
\email{zqchen@mail.wipm.ac.cn}
\thanks{Z. Chen is partially supported by NSFC grant No. 11171338.}

\author{Quanhua Xu}
\address{School of Mathematics and Statistics, Wuhan University, Wuhan 430072, China and Laboratoire de Math{\'e}matiques, Universit{\'e} de Franche-Comt{\'e},
25030 Besan\c{c}on Cedex, France}
\email{qxu@univ-fcomte.fr}
\thanks{Q. Xu and Z. Yin are partially supported by ANR-2011-BS01-008-01.}

\author{Zhi Yin}
\address{School of Mathematics and Statistics, Wuhan University, Wuhan 430072, China and Laboratoire de Math{\'e}matiques, Universit{\'e} de Franche-Comt{\'e},
25030 Besan\c{c}on Cedex, France}
\email{hustyinzhi@163.com}

\date{}
\maketitle

\markboth{Z. Chen, Q. Xu, and Z. Yin}%
{Harmonic analysis on quantum tori}

\begin{abstract}
This paper is devoted to the study of harmonic analysis on quantum
tori. We consider several summation methods on these tori, including the square Fej\'er means, square and circular Poisson means, and Bochner-Riesz means. We first establish the maximal inequalities for these means, then obtain the corresponding pointwise convergence theorems.  In particular, we prove the noncommutative analogue of the classical Stein theorem on Bochner-Riesz means. The second part of the paper deals with Fourier multipliers on quantum tori. We prove that the completely bounded $L_p$ Fourier multipliers on a quantum torus are exactly  those on the classical torus of the same dimension. Finally, we present the Littlewood-Paley theory associated with the circular Poisson semigroup on quantum tori. We show that the Hardy spaces in this setting possess the usual properties of Hardy spaces, as one can expect. These include the quantum torus analogue of Fefferman's $\mathrm{H}_1$-BMO duality theorem and interpolation theorems. Our analysis is based on the recent developments of noncommutative martingale/ergodic inequalities and  Littlewood-Paley-Stein theory.
\end{abstract}

\bigskip

\tableofcontents


\section{Introduction}\label{Intro}


The subject of this paper follows the current line of investigation on noncommutative harmonic analysis. This topic has many interactions with other fields such as operator spaces, quantum probability, operator algebras, and of course, harmonic analysis. The aspect we are interested in is particularly related to the recent developments of noncommutative martingale/ergodic inequalities and Littlewood-Paley-Stein theory for quantum Markov semigroups.  Motivated by operator spaces and by using tools from this theory, many classical martingale and ergodic inequalities have been successfully transferred to the noncommutative setting (see, for instance, \cite{PX1997, Junge2002, JX2003, JX2007,  Ran2002, Ran2007, PR2006, Bek2008, BCPY2010, Per2009}). These inequalities of quantum probabilistic nature have, in return, applications to operator space theory (cf., e.g.  \cite{PS2002, Junge2005, JP2008, JP2010, JP2010b, Xu2006, Xu2006b}). Closely related to that, harmonic analysis on quantum semigroups has started to be developed in the last years. This first period of development of the noncommutative Littlewood-Paley-Stein theory deals with square function inequalities, $\mathrm{H}_1$-BMO duality and Riesz transforms  (cf. \cite{JLX2006, Mei2007, Mei2008, JM2010, JM2011}). One can also include in this topic the very fresh promising direction of research on the Calder\'on-Zygmund singular integral operators in the noncommutative setting (cf. \cite{Parcet2009, MP2009, JMP2011}). The concern of the present paper is directly linked to this last direction. Our objective is to develop harmonic analysis on quantum tori.

Quantum or noncommutative tori are fundamental  examples in operator algebras and probably the most accessible interesting class of objects of study in  noncommutative geometry (cf. \cite{Connes1980, Connes1994}). There exist extensive works on them (see, for instance, the survey paper by Rieffel \cite{Rie1990} for those before the 1990's). We refer to \cite{CM2011, EL2007, Var2006} for some illustrations of more recent developments on this topic.

We now recall the definition of quantum tori. Let $d \ge 2$ and $\theta = (\theta_{k j})$ be a real skew-symmetric
$d \times d$-matrix. The $d$-dimensional noncommutative
torus $\mathcal{A}_{\theta}$ is the universal $C^*$-algebra generated by $d$
unitary operators  $U_1, \ldots, U_d$ satisfying the following commutation
relation
 $$U_k U_j = e^{2 \pi \mathrm{i} \theta_{k j}} U_j U_k,\quad j,k=1,\ldots, d.$$
There exists a  faithful tracial state  $\tau$ on $\mathcal{A}_{\theta}.$  Let $\T^d_\theta$ be the von Neumann algebra in the GNS
representation of $\tau$.  $\T^d_\theta$  is called the quantum $d$-torus associated with $\theta$. Note that if $\theta=0$, then $\mathcal{A}_{\theta}=C(\mathbb{T}^d)$  and $\T^d_{\theta}=L_\8(\mathbb{T}^d)$, where $\T^d$ is the usual $d$-torus.  So a quantum $d$-torus is a deformation of the usual $d$-torus. It is thus natural to expect that $\T^d_\theta$ shares many properties with $\T^d$.  This is indeed the case for differential geometry, as shown by the works of Connes and his collaborators. However, little is done regarding analysis.  To our best knowledge, up to now, only the mean convergence theorem of quantum Fourier series by the square Fej\'er summation was proved at the $C^*$-algebra level (cf. \cite{Weaver1996, {Weaver2001}}), and on the other hand,  the quantum torus analogue of Sobolev inequalities was obtained only in the Hilbert, i.e., $L_2$ space case (cf. \cite{Spera92}).  The reason of this lack of development of analysis might be explained by numerous difficulties one may encounter when dealing with noncommutative $L_p$-spaces, since these spaces come up unavoidably if one wishes to do analysis on quantum tori. For instance, the usual way of proving pointwise convergence theorems is to pass through the corresponding maximal inequalities. But the study of maximal inequalities is one of the most delicate and difficult parts in noncommutative analysis.

\medskip

 This paper is the first one of a long project that intends to develop analysis on quantum tori and more generally on twisted crossed products by amenable groups. Our aim here is to study some important aspects of harmonic analysis on $\T^d_\theta$. The subject that we address is three-fold:
\begin{enumerate}[{\rm i)}]

\item {\it Convergence of Fourier series.} We consider several summation methods on $\T^d_\theta$, including the square Fej\'er means, square and circular Poisson means, and Bochner-Riesz means. We first establish the maximal inequalities for them  and then obtain the corresponding pointwise convergence theorems. This part heavily relies on the theory of noncommutative martingale and ergodic inequalities.

\item {\it Fourier multipliers.} The right framework for our study of Fourier multipliers is operator space theory. We  show that for $1\le p\le\8$ the completely bounded  $L_p$ Fourier multipliers  on $\T^d_\theta$  coincide with those on $\T^d$.

\item {\it Hardy and {\rm BMO} spaces.} Based on the recent development of the noncommutative Littlewood-Paley-Stein theory and the operator-valued harmonic analysis, we define Hardy and BMO spaces on $\T^d_\theta$ via the circular Poisson semigroup. We show  that the properties of Hardy spaces in the classical case remain true in the quantum setting. In particular, we get  the $\mathrm{H}_1$-BMO duality theorem.

\end{enumerate}

One main strategy for approaching these problems is  to transfer them to the corresponding ones in the case of operator-valued functions on the classical tori, and then to use existing results in the latter case or adapt classical arguments. Due to the noncommutativity of operator product, substantial difficulties arise in our arguments, like usually in noncommutative analysis. One of the subtlest parts of our arguments is the proof of the weak type $(1,1)$ maximal inequalities for the square Fej\'er and Poisson means because of their multiple-parameter nature. This is the first time that noncommutative weak type $(1,1)$ maximal inequalities are considered for mappings of  this nature. Another intricate part concerns  the  analogue  for $\T^d_\theta$ of the classical Stein theorem on Bochner-Riesz means. The proof of the corresponding maximal inequalities is quite technical too. Our study of Hardy spaces via the Littlewood-Paley theory necessitates a very careful analysis of various BMO-norms and square functions. The difficulty of this study is partly explained by the lack of an explicit handy formula of the circular Poisson kernel on $\T^d$ for $d\ge2$. 

\medskip

We end this introduction with a brief description of the organization of the paper. In Section \ref{Pre} we present some preliminaries and notation on quantum tori. This section also introduces our transference method. The simple section \ref{MeanConverg} defines the summation methods studied in the paper and deals with the mean convergence of quantum Fourier series by them. Section 4 is devoted to the maximal inequalities associated to these summation methods. Their proofs depend, via transference, on some general maximal inequalities for operator-valued functions on $\real^d$ (or $\T^d$) that are of interest for their own right.  These maximal inequalities  are then applied in  Section \ref{PointwiseConvg} to obtain the corresponding pointwise convergence theorems. Section \ref{BRMean} deals with the Bochner-Riesz means. The main theorem there is the quantum analogue of Stein's classical theorem. The difficult part is the type $(p, p)$ maximal inequality for these means. In Section \ref{CBMultiplier} we  discuss $L_p$ Fourier multipliers on $\T^d_\theta$. We show that a Fourier multiplier is completely bounded on the noncommutative $L_p$-space associated to $\T^d_\theta$ iff it is completely bounded on $L_p(\T^d)$. In this case, the two completely bounded norms are equal.  Finally, in Section \ref{H1BMOLPT}, we present the Littlewood-Paley theory on $\T^d_\theta$ and define the associated Hardy and BMO spaces using the circular Poisson semigroup, and show that they possess all expected properties of the usual Hardy spaces on $\real^d$. Our approach is to transfer this theory to the operator-valued case on $\T^d$ and  to use Mei's arguments in \cite{Mei2007} for the $\real^d$ setting. Since the geometry of $\T^d$ and the circular Poisson kernel are less handy than those of $\real^d$, we cannot directly apply Mei's results to  our case. However, considering functions on $\T^d$ as periodic functions on $\real^d$, we can still reduce most of our problems to the corresponding ones on periodic functions on $\real^d$, then adapt Mei's argument to the periodic case. A good part of this section is devoted to the study of several BMO-norms and square functions naturally appearing in this periodization procedure.


\section{Preliminaries}\label{Pre}



\subsection{Noncommutative $L_p$ spaces}


Let  $\mathcal{M}$ be a von Neumann algebra and $\mathcal{M}_+$ its positive part. Recall that
a {\it trace} on $\mathcal{M}$ is a map $\tau: \mathcal{M}_+
\rightarrow [0, \infty ]$ satisfying:
\begin{enumerate}[{\rm i)}]

\item $\tau (x + y) = \tau (x) + \tau (y)$ for
arbitrary $x , y \in \mathcal{M}_+;$

\item $\tau ( \lambda x ) = \lambda \tau (x)$ for any $\lambda
\in [0, \infty )$ and $x \in \mathcal{M}_+;$

\item $\tau (x^* x) = \tau (xx^*)$ for all $x\in \mathcal{M}.$

 \end{enumerate}
$\tau$ is said to be {\it normal} if $\sup_{\gamma} \tau
(x_{\gamma}) = \tau ( \sup_{\gamma} x_{\gamma})$ for any bounded
increasing net $(x_{\gamma} )$ in $\mathcal{M}_+,$ {\it semifinite}
if for each $x \in \mathcal{M}_+ \backslash \{0\}$ there is a
nonzero $y \in \mathcal{M}_+$ such that $y \leq x$ and $\tau (y) <
\infty,$ and {\it faithful} if for each $x \in \mathcal{M}_+
\backslash \{0\},$ $\tau (x) > 0.$ A von Neumann algebra
$\mathcal{M}$ is called {\it semifinite} if it admits a normal
semifinite faithful trace $\tau.$ We refer to \cite{Tak1979} for theory of von Neumann algebras. 
Throughout this paper, $\mathcal{M}$ will always denote a semifinite von Neumann 
algebra equipped with a normal semifinite faithful trace $\tau.$

Denote by $\mathcal{S}_+$ the set of all $x \in \mathcal{M}_+$ such
that $\tau (\mathrm{supp} (x) ) < \infty,$ where $\mathrm{supp} (x)$ is
the support of $x$ which is defined as the least projection $e$ in
$\mathcal{M}$ such that $ex = x$ or equivalently $xe = x.$ Let
$\mathcal{S}$ be the linear span of $\mathcal{S}_+.$ Then
$\mathcal{S}$ is a $\ast$-subalgebra of $\mathcal{M}$ which is
$w^*$-dense in $\mathcal{M}.$ Moreover, for each $0 < p < \infty,$
$x \in \mathcal{S}$ implies $| x |^p \in \mathcal{S}_+$ (and so
$\tau (| x |^p)< \infty$), where $|x| = (x^* x )^{1/2}$ is the
modulus of $x.$ Now, we define $\| x \|_p = \left [\tau (| x |^p)
\right ]^{1/p} $ for all $x \in \mathcal{S}.$ One can show that
$\|\cdot \|_p$ is a norm on $\mathcal{S}$ if $1 \leq p < \infty,$
and a quasi-norm (more precisely, $p$-norm) if $0< p < 1.$ The
completion of $( \mathcal{S}, \|\cdot \|_p )$ is denoted by $L_p
(\mathcal{M}, \tau )$ or simply by $L_p(\M)$. This is the noncommutative 
$L_p$-space associated with $(\mathcal{M}, \tau ).$ The elements of $L_p(\M)$ 
can be described by densely 
defined closed operators measurable with respect to $(\M,\tau)$, like in
the commutative case. 
For convenience, we set
$L_{\infty} (\mathcal{M}) = \mathcal{M}$ equipped with the
operator norm. The trace $\tau$ can be extended to a linear
functional on $\mathcal{S},$ still denoted by $\tau.$ Since $| \tau
(x) | \leq \| x \|_1$ for all $x \in \mathcal{S},$ $\tau$ further extends to
a continuous functional on $L_1 (\mathcal{M}).$

Let $0 < r, p, q \leq \infty$ be such that $1/ r = 1/p +
1/q.$ If $x \in L_p (\mathcal{M}), y \in L_q (\mathcal{M})$ then $ xy \in L_r (\mathcal{M})$
and  the following H\"older inequality holds:
 $$\| x y \|_r \leq \| x \|_p \| y\|_q.$$
In particular, if $r = 1,$ $| \tau (x y) |
\leq \| x y \|_1 \leq \| x \|_p \| y \|_q$ for arbitrary $x \in L_p
(\mathcal{M})$ and $y \in L_q (\mathcal{M}).$ This
defines a natural duality between $L_p (\mathcal{M})$ and
$L_q (\mathcal{M}): \langle x, y \rangle = \tau (x y).$ For
any $1 \leq p < \infty$ we have $L_p (\mathcal{M})^* = L_q
(\mathcal{M})$ isometrically. Thus, $L_1 (\mathcal{M})$ is the predual $\mathcal{M}_*$ of $\mathcal{M},$ and $L_p
(\mathcal{M})$ is reflexive for $1 < p < \infty.$  We refer  to \cite{PX2003} for more information on noncommutative
$L_p$-spaces.


\subsection{Quantum tori}


Let $d\ge2$ and $\theta=(\theta_{kj})$ be a real skew symmetric $d\times d$-matrix. The associated $d$-dimensional noncommutative
torus $\mathcal{A}_{\theta}$ is the universal $C^*$-algebra generated by $d$
unitary operators $U_1, \ldots, U_d$ satisfying the following commutation
relation
 \beq \label{eq:CommuRelation}
 U_k U_j = e^{2 \pi \mathrm{i} \theta_{k j}} U_j U_k,\quad j,k=1,\ldots, d.
 \eeq
We will use standard notation from multiple Fourier series. Let  $U=(U_1,\cdots, U_d)$. For $m=(m_1,\cdots,m_d)\in\ent^d$ we define
 $$U^m=U_1^{m_1}\cdots U_d^{m_d}.$$
 A polynomial in $U$ is a finite sum
  $$ x =\sum_{m \in \mathbb{Z}^d}\alpha_{m} U^{m}\quad \text{with}\quad
 \alpha_{m} \in \mathbb{C},$$
that is, $\alpha_{m} =0$ for all but
finite indices $m \in \mathbb{Z}^d.$ The involution algebra
$\mathcal{P}_{\theta}$ of such all polynomials is
dense in $\A_{\theta}.$ For any polynomial $x$ as above we define
 $$\tau (x) = \alpha_{\mathbf{0}},$$
where $\mathbf{0}=(0, \cdots, 0)$.
Then, $\tau$ extends to a  faithful  tracial state on $\A_{\theta}$.  Let  $\mathbb{T}^d_{\theta}$ be the $w^*$-closure of $\A_{\theta}$ in  the GNS representation of $\tau$. This is our $d$-dimensional quantum torus. The state $\tau$ extends to a normal faithful tracial state on $\mathbb{T}^d_{\theta}$ that will be denoted again by $\tau$. Recall that the von Neumann algebra $\mathbb{T}^d_{\theta}$ is hyperfinite.

Since $\tau$ is a state, $L_q(\T^d_\theta)\subset L_p(\T^d_\theta)$ for any $0<p<q\le\8$.  Any $x\in L_1(\T^d_\theta)$ admits a formal
Fourier series:
 $$x \sim \sum_{m \in\mathbb{Z}^d} \hat{x} ( m ) U^{m},$$
where \
 $$\hat{x}( m) = \tau((U^m)^*x),\quad m \in \mathbb{Z}^d$$
are the Fourier coefficients of $x$. $x$ is, of course, uniquely determined by its Fourier series.


\subsection{Transference}\label{Transference}


We denote the usual $d$-torus by $\mathbb{T}^d$:
 $$\mathbb{T}^d = \big \{(z_1, \ldots, z_d): \; |z_j| =1, z_j \in\mathbb{C},\; 1\le j\le d \big \}.$$
$\mathbb{T}^d$ is equipped with the usual topology and group law
multiplication, that is,
 $$z \cdot w = (z_1, \ldots, z_d) \cdot(w_1, \ldots, w_d) = (z_1 w_1, \ldots, z_d w_d). $$
For any $m \in \mathbb{Z}^d$ and $z =(z_1, \ldots, z_d) \in
\mathbb{T}^d$ let
 $$z^m =z_1^{m_1}\cdots z_d^{m_d}.$$
We will need the tensor von Neumann algebra $\mathcal{N}_{\theta} = L_{\infty} (\mathbb{T}^d) \overline{\otimes}
\mathbb{T}^d_{\theta}$, equipped with the tensor trace $\nu = \int d
m \otimes \tau,$ where $d m$ is  normalized Haar measure on
$\mathbb{T}^d.$ Note that for every $0 < p < \8,$
 $$L_p(\mathcal{N}_{\theta}, \nu ) \cong L_p (\mathbb{T}^d; L_p (\mathbb{T}^d_{\theta})).$$
The space on the right hand side is the space of Bochner $p$-integrable functions from $\T^d$ to $L_p(\mathbb{T}^d_{\theta})$.  Accordingly, let $C(\mathbb{T}^d; \A_{\theta})$ denote the $C^*$-algebra of continuous functions from $\T^d$ to $\A_{\theta}$.  For each $z
\in \mathbb{T}^d,$ define $\pi_{z}$ to be the isomorphism of
$\mathbb{T}^d_{\theta}$ determined by
  $$\pi_{z} (U^{m}) = z^{m} U^{m } = z_1^{m_1} \cdots z_d^{m_d} U_1^{m_1} \cdots U_d^{m_d}. $$
It is clear that  $\pi_{z}$ is trace preserving, so extends to an isometry on  $L_p(\mathbb{T}^d_{\theta})$ for every $0 < p <\8$. Thus we have
 $$\|\pi_{z} (x) \|_p = \|x\|_p,\quad x\in L_p(\mathbb{T}^d_{\theta}),\; 0 < p \le\8.$$

\begin{prop}\label{prop:TransC}
 For any $x \in L_p (\mathbb{T}^d_{\theta})$ the function
$\tilde{x}: z \mapsto \pi_z(x)$ is continuous from
$\mathbb{T}^d$ to $L_p(\mathbb{T}^d_{\theta})$ $($with respect to the $w^*$-topology for $p=\8)$. If $x\in\A_\theta$, it is continuous from $\mathbb{T}^d$ to $\A_\theta$.
 \end{prop}

\begin{proof}
Consider first the case  $0 < p < \8$. Let $x \in L_p (\mathbb{T}^d_{\theta}).$ Since
$\mathcal{P}_{\theta}$ is dense in $L_p
(\mathbb{T}^d_{\theta}),$ for arbitrary $\varepsilon >0$ there is
$x_0 \in \mathcal{P}_{\theta}$ such that $\| x - x_0
\|_p < \varepsilon.$ Clearly, $\pi_z (x_0)$ is a polynomial in $U$
of the same degree as $x_0.$ Thus, $z \mapsto \pi_z
(x_0)$ is continuous from $\mathbb{T}^d$ into
$L_p(\mathbb{T}^d_{\theta}).$ We then deduce the desired continuity of  $\tilde{x}$. The same argument works equally for $\A_\theta$.
The  case of $p=\8$ follows from that of $p=1$ by duality.
\end{proof}

The previous result in the case of $p=\infty$ implies, in particular, that the map $x\mapsto \tilde{x}$ establishes an isomorphism from
$\mathbb{T}^d_{\theta}$ into $\N_{\theta}.$ It is also clear that this isomorphism is trace preserving. Thus we get the following

\begin{cor}\label{prop:TransLp}
 \begin{enumerate}[{\rm i)}]

\item Let $0 < p \leq \8.$ If $x\in L_p (\mathbb{T}^d_{\theta}),$ then $\tilde{x} \in L_p (\N_{\theta})$ and $\| \tilde{x} \|_p = \| x\|_p,$ that is, $x \mapsto \tilde{x}$ is an isometric embedding from $L_p (\mathbb{T}^d_{\theta})$ into $L_p (\N_{\theta}).$ Moreover, this map is also an isomorphism from $\A_\theta$ into $C(\T^d; \A_\theta)$.

\item Let $\widetilde{\mathbb{T}^d_{\theta}} = \{ \tilde{x}: x \in \mathbb{T}^d_{\theta}\}.$ Then $\widetilde{\mathbb{T}^d_{\theta}}$ is a von Neumann subalgebra of $\mathcal{N}_{\theta}$ and the associated conditional expectation is given by
 $$\mathbb{E} (f)(z) = \pi_{z} \Big ( \int_{\mathbb{T}^d} \pi_{\bar{w}}
 \big [ f( w )\big ] d m (w) \Big ),\quad z\in\T^d, \; f \in \N_{\theta}.$$
Moreover, $\mathbb{E}$ extends to a contractive projection from $L_p(\N_{\theta})$ onto $L_p(\widetilde{\mathbb{T}^d_{\theta}})$ for $1\leq p\leq \infty.$
\item $L_p (\mathbb{T}^d_{\theta})$ is isometric to $L_p (\widetilde{\mathbb{T}^d_{\theta}})$ for every $0 < p \le \8.$

\end{enumerate}

\end{cor}

Our transference method consists in the following procedure:
 \be
 x\in L_p(\mathbb{T}^d_{\theta}) \mapsto \tilde{x}\in
 L_p(\widetilde{\mathbb{T}^d_{\theta}}) \subset L_p(\mathcal{N}_{\theta}).
 \ee
This allows us to work in
$L_p(\mathcal{N}_{\theta}).$ Then, in order to return back to
$L_p(\widetilde{\mathbb{T}^d_{\theta}}) \cong
L_p(\mathbb{T}^d_{\theta}),$ we apply the conditional expectation
$\mathbb{E}$ to elements in $L_p(\mathcal{N}_{\theta}).$


\section{Mean Convergence}\label{MeanConverg}


We begin with the mean convergence of Fourier series defined on quantum tori for an illustration of the transference method described in the previous section. This section also introduces the summation methods studied throughout the paper.  They are the following:

\smallskip

\begin{enumerate}[$\bullet$]

\item The {\it square Fej\'er mean}
 $$F_N [x] = \sum_{m \in \mathbb{Z}^d, |m |_{\infty}\leq N} \Big ( 1-\frac{|m_1|}{N+1} \Big )
 \cdots \Big( 1-\frac{|m_d|}{N+1} \Big ) \hat{x} ( m) U^{m},\quad N \ge 0.$$

\item The {\it square Poisson mean}
 $$P_r[x] = \sum_{m \in \mathbb{Z}^d } \hat{x} ( m )r^{|m|_1} U^{m}, \quad 0 \le r < 1.$$

\item The {\it circular Poisson mean}
 $$\mathbb{P}_r[x] = \sum_{m \in \mathbb{Z}^d } \hat{x} ( m ) r^{|m|_2} U^{m}, \quad 0 \le r < 1.$$

\item Let  $\Phi$ be a continuous function on $\mathbb{R}^d$ with $\Phi (0) =1.$  Define
  $$\Phi^{\varepsilon}[x]=\sum_{m \in \mathbb{Z}^d} \Phi (\varepsilon m) \hat{x} ( m ) U^m, \quad \varepsilon>0.$$
We will always impose the following condition to $\Phi$:
 \beq\label{eq:PhiFuncCondition}
 \left \{
 \begin{split}
 & ~~\Phi (s) = \hat{\varphi} (s)\quad \text{with}\; \int_{\mathbb{R}^d} \varphi (s) d s =1;\\
 & ~~|\Phi (s) | + |\varphi(s)| \le A (1 + |s|)^{-d-\delta},\;
 \forall s \in \mathbb{R}^d,
 \end{split} \right.
 \eeq
for some $A, \delta >0$ (cf. \cite[p. 253]{SW1975}). \\
 In the above, $x\in L_1(\mathbb{T}^d_{\theta})$ has its Fourier series expansion: $x\sim\sum_{m\in\ent^d}\hat x(m)U^m$, and for $m\in\ent^d$
 $$|m|_p = (\sum^d_{j=1} |m_j|^p )^{1/p}$$
with the usual modification for $p=\8$.

 \end{enumerate}

The last summation method contains two special important examples of the function $\Phi$. The first one is
 $$\Phi (s) = e^{- 2 \pi|s|}\quad \text{and} \quad \varphi (s) = c_d( 1 + |s|^2 )^{- (d+ 1)/2},\quad \forall s\in \mathbb{R}^d,$$
 where we have used the standard notation in harmonic analysis that $|s|=|s|_2$ denotes the Euclidean norm of $\real^d$.  In this case,
 $$\Phi^\e[x]= \sum_{m \in \mathbb{Z}^d} e^{- 2 \pi |m|_2\e} \hat{x} ( m ) U^m.$$
 This is the circular Poisson integral $\mathbb{P}_r [x]$ of $x$ with $r = e^{- 2 \pi \e}.$

The second example arises when $\alpha > (d-1)/2$ in the following definition
  \be \Phi (s) =
 \left \{ \begin{split}
 &(1 - |s|^2)^{\alpha} & \quad \text{if}\quad |s| < 1,\\
 &0 & \quad \text{if}\quad |s| \ge 1.
 \end{split} \right.
 \ee
It is well known that
 $$\varphi (s) = \hat{\Phi} (s) =
 \frac{\Gamma(\alpha+1)J_{\frac{d}{2}+\alpha}(2\pi|s|)}{\pi^{\alpha}|s|^{\frac{d}{2}+\alpha}},\quad
 \forall\; s \in \mathbb{R}^d \setminus \{0\}, $$
where $J_{\lambda}$ is the Bessel function of order $\lambda.$ In this case we obtain
the {\it Bochner-Riesz mean} of order $\alpha$ on the quantum
torus:
 $$B^{\alpha}_R[x] = \sum_{| m|_2 \leq R} \Big ( 1 -\frac{|m|_2^2}{R^2} \Big )^{\alpha} \hat{x} ( m ) U^m.$$

A fundamental problem is in which sense the above means of  the operator $x$ converge back to  $x$. This problem is partly
investigated in this section. Indeed, we have the following mean
convergence theorem.

\begin{prop}\label{th:MeanConver}
 Let $1 \le p < \infty$ and $x \in L_p (\mathbb{T}^d_{\theta}).$
Then $F_N[x]$ converges to $x$ in $L_p(\mathbb{T}^d_{\theta})$ as $N\to\8$. The same convergence holds for $P_r[x]$, $\mathbb{P}_r[x]$ as
$r \to 1$ and  $\Phi^{\varepsilon}[x]$ as $\varepsilon \to 0$. Moreover, for $p=\8$ these limits hold for any
 $x\in\A_{\theta}$.
\end{prop}

The proof can be done either by imitating the classical proofs (cf.
\cite{SW1975}), or using the transference argument. The
second method is more elegant and simpler. The corresponding results in $L_p
(\mathcal{N}_{\theta})$ are simple and well-known when one writes
$L_p (\mathcal{N}_{\theta}) = L_p (\mathbb{T}^d; L_p
(\mathbb{T}^d_{\theta})),$ which reduces the mean convergence in
$L_p ( \mathbb{T}^d_{\theta})$ to the corresponding one in
the vector-valued case on the usual torus $\mathbb{T}^d.$

As all these summation methods in the vector-valued case are given
by approximation identities, it is better to state and prove first a
general convergence theorem for convolution operators by an
approximation identity in $L_p (\mathbb{T}^d; X),$ where $X$ is a
Banach space. Here $L_p (\mathbb{T}^d; X)$ denotes the $L_p$-space of Bochner $p$-integrable functions from $\T^d$ to $X$.

\begin{def}\label{df:ApprIdent}
Let $\Lambda$ be a directed set. An {\it approximation identity} on
the multiplication group $\mathbb{T}^d$ (as $\lambda \to \lambda_0 $) is a
family of functions $(\varphi_{\lambda})_{\lambda \in \Lambda}$ in
$L_1 (\mathbb{T}^d)$ verifying the following three conditions:
\begin{enumerate}[{\rm i)}]

\item $\int_{\mathbb{T}^d} \varphi_{\lambda} ( z ) d m (z) =1$ for all $\lambda \in \Lambda.$

\item $\sup_{\lambda \in \Lambda} \| \varphi_{\lambda} \|_1 < \8.$

\item For any neighborhood $V$ of the identity $(1, \ldots, 1)$ of the group $\mathbb{T}^d$ we have \be \int_{\mathbb{T}^d \setminus V} | \varphi_{\lambda}| d m (z) \to 0\; \text{as}\; \lambda \to \lambda_0. \ee

\end{enumerate}
\end{def}

Recall that for $N \ge0$ an integer, the square Fej\'{e}r
kernel on $\mathbb{T}^d$ is
 \beq\label{squareFejerKer}
  F_N ( z ) = \sum_{\substack{m \in \mathbb{Z}^d,\,
 |m |_{\infty}\leq N}} \Big ( 1-\frac{|m_1|}{N+1} \Big ) \cdots
 \Big ( 1-\frac{|m_d|}{N+1} \Big ) z^{ m}.
 \eeq
For $0 \le r < 1,$ the {\it square and circular} Poisson kernels are respectively
 \beq\label{PoissonKer}
 P_r (z) = \sum_{m \in\mathbb{Z}^d} r^{| m|_1} z^{m}\quad\text{and}\quad
 \mathbb{P}_r (z) =\sum_{m \in \mathbb{Z}^d} r^{| m|_2} z^{m}.
 \eeq
It is well known that $(F_N)_{N \ge 1},$ $(P_r)_{0 \le r <1}$ and $(\mathbb{P}_r )_{0 \le r <1}$ are all
approximation identities on $\mathbb{T}^d.$ Also, if we write $\Phi^{\varepsilon} (s) = \Phi (\varepsilon s),$
then $\Phi^{\varepsilon} = \wh{\varphi_{\varepsilon}}$ with $\varphi_{\varepsilon} (s) = \frac{1}{\varepsilon^d} \varphi \big ( \frac{s}{\varepsilon} \big )$ for $s\in\real^d$.  Let
 $$K_{\varepsilon} (s) = \sum_{m \in \mathbb{Z}^d} \varphi_{\varepsilon} (s + m),\quad s\in\real^d.$$
$K_{\varepsilon}$ is  periodic, so can be viewed as a function on $\T^d$. Then by \eqref{eq:PhiFuncCondition} it can be proved that $( K_{\varepsilon}
)_{\varepsilon >0}$ is an approximation identity on $\mathbb{T}^d$
such that
 \beq\label{eq:PhiFunConv}
 ( K_{\varepsilon} \ast f ) (z)= \sum_{m \in \mathbb{Z}^d} \Phi ( \varepsilon m ) \hat f(m) z^m,
 \quad f \sim\sum_{m \in \mathbb{Z}^d}  \hat f(m) z^m
  \eeq
(see the proof of Theorem VII.2.11 in \cite{SW1975}).

Let $X$ be a Banach space and let $1 \le p \le\8.$ Suppose that
$(\varphi_{\lambda})_{\lambda \in \Lambda}$ is an approximation
identity on $\mathbb{T}^d.$ For any $f \in L_p (\mathbb{T}^d; X)$ we
define the convolution $\varphi_{\lambda}\ast f$ by
 $$(\varphi_{\lambda} \ast f ) (z) =
 \int_{\mathbb{T}^d} f ( w ) \varphi_{\lambda} (\bar{w}\cdot z) d m (w),\quad \forall\; z \in \mathbb{T}^d. $$
Then for any $f \in L_p (\mathbb{T}^d; X)$ we have $\varphi_{\lambda}\ast f \in L_p
(\mathbb{T}^d; X)$ and
 $$\|\varphi_{\lambda}\ast f \|_p \le \| f  \|_p \| \varphi_{\lambda} \|_1.$$
The following vector-valued result is well-known. The proof in the
scalar case (cf. e.g. \cite[Theorem~1.2.19]{Gra2008}) is valid as well in the vector-valued setting without any change. $ C(\mathbb{T}^d;X)$ denotes the space of continuous functions from $\T^d$ to $X$, equipped with the uniform norm.

\begin{prop}\label{le:VectorValueMeanConverConvo}
 Let $X$ be a Banach space and let $1 \le p < \8.$ Let $(\varphi_{\lambda})_{\lambda \in \Lambda}$ be an approximation
identity on $\mathbb{T}^d.$ If $f \in L_p
(\mathbb{T}^d; X)$, then
 $$ \| \varphi_{\lambda} \ast f - f \|_p \to0\; \text{as}\; \lambda \to \l_0.$$
Moreover, when $p=\infty$ the above limit holds for any $f\in C(\mathbb{T}^d;X).$
 \end{prop}

It is now clear that Proposition \ref{th:MeanConver} immediately follows from
Proposition \ref{le:VectorValueMeanConverConvo} via the transference method.


\section{Maximal inequalities}\label{Maximal inequalities}


In this section, we present the maximal inequalities of the
summation methods of  Fourier series defined previously. These
inequalities will be used for the pointwise convergence  in the next
section. We first recall the definition of the noncommutative
maximal norm introduced by Pisier \cite{Pisier1998} and Junge
\cite{Junge2002}. Let $\M$ be a von Neumann algebra equipped 
with a normal semifinite faithful trace $\tau.$ Let $1\le p\le\8$. 
We define $L_p(\M;\el_\8)$  to
be the space of all sequences $x=(x_n)_{n\ge1}$ in $L_p(\M)$ which
admit a factorization of the following form: there exist $a, b\in
L_{2p}(\M)$ and a bounded sequence $y=(y_n)$ in $L_\8(\M)$ such that
 $$x_n=ay_nb, \quad \forall\; n\ge1.$$
The norm of  $x$ in $L_p(\M;\el_\8)$ is given by
 $$\|x\|_{L_p(\M;\el_\8)}=\inf\big\{\|a\|_{2p}\,
 \sup_{n\ge1}\|y_n\|_\8\,\|b\|_{2p}\big\} ,$$
where the infimum runs over all factorizations of $x$ as above.

We will follow the convention adopted in  \cite{JX2007}  that
$\|x\|_{L_p(\M;\el_\8)}$ is denoted by
 $\big\|\sup_n^+x_n\big\|_p\ .$ We should warn the reader that
$\big\|\sup^+_nx_n\big\|_p$ is just a notation since $\sup_nx_n$
does not make any sense in the noncommutative setting. We find,
however, that $\big\|\sup^+_nx_n\big\|_p$ is more intuitive than
$\|x\|_{L_p(\M;\el_\8)}$. The introduction of this notation is
partly justified by the following remark.
\medskip

\begin{remark}\label{rk:MaxFunct}\rm
Let $x=(x_n)$ be a sequence of selfadjoint operators in $L_p(\M)$.
Then $x\in L_p(\M;\el_\8)$ iff there exists a positive element $a\in
L_p(\M)$ such that $-a\leq x_n \le a$ for all $n\ge1$. In this case we have
 $$\big \|{\sup_{n \ge1}}^{+} x_n \big \|_p=\inf\big\{\|a\|_p\;:\; a\in L_p(\M),\; -a\leq x_n\le a,\;\forall\; n\ge1\big\}.$$
\end{remark}

More generally, if $\Lambda$ is any index set, we define  $L_p (\M;
\ell_{\8}(\Lambda))$ as the space of all $x = (x_{\lambda})_{\lambda \in \Lambda}$ in $L_p (\M)$ that can be factorized as
 $$x_\l=ay_\l b\quad\mbox{with}\quad a, b\in L_{2p}(\M),\; y_\l\in L_\8(\M),\; \sup_\l\|y_\l\|_\8<\8.$$
The norm of $L_p (\M; \ell_{\8}(\Lambda))$ is defined by
 $$\big \| {\sup_{\l\in\Lambda}}^{+} x_\l\big \|_p=\inf_{x_\l=ay_\l b}\big\{\|a\|_{2p}\,
 \sup_{\l\in\Lambda}\|y_\l\|_\8\,\|b\|_{2p}\big\} .$$
It is shown in \cite{JX2007} that $x\in L_p (\M;
\ell_{\8}(\Lambda))$ iff
 $$\sup\big\{\big \| {\sup_{\l\in J}}^{+} x_\l\big \|_p\;:\; J\subset\La,\; J\textrm{ finite}\big\}<\8.$$
In this case, $\big \| {\sup_{\l\in\Lambda}}^{+} x_\l\big \|_p$ is
equal to the above supremum.

The following is the main theorem of this section.

\begin{thm}\label{th:FejerPoissonPhiMaxIneq}
 \begin{enumerate}[{\rm i)}]
 \item  Let $x\in L_1(\T^d_\theta)$. Then for  any $\a>0$ there exists a projection $e\in\T^d_\theta$ such that
 $$\sup_{N\ge0}\big\|eF_N[x]e\big\|_\8\le\a\quad\text{and}\quad \tau(e^\perp)\le C_d\,\frac{\|x\|_1}\a.$$

\item  Let $1<p\le\8$. Then
 $$\big\|{\sup_{N\ge0}}^+F_N[x]\big\|_p\le C_d\,\frac{p^2}{ (p-1)^2}\,\|x\|_p,\quad\forall\; x\in L_p(\T^d_\theta).$$
Both statements hold for the three other summation methods $P_r$,   $\mathbb{P}_r$ and $\Phi^{\varepsilon}.$ In the case of $\Phi^{\varepsilon}$, the constant $C_d$ also  depends on the two constants in \eqref{eq:PhiFuncCondition}.
\end{enumerate}
\end{thm}

In the terminology of  \cite{JX2007}, we can rephrase parts i) and ii) as that the map $x\mapsto (F_N[x])_{N\geq0}$ is of weak type $(1,1)$ and of type $(p,p)$, respectively. Before proceeding to the proof of the theorem, we point out that its part concerning the circular Poisson mean $\mathbb{P}_r$ can be easily deduced from  \cite{JX2007}. This is due to the fact that $\big(\mathbb{P}_{r}\big)_{0\le r<1}$ is a symmetric diffusion semigroup on $\mathbb{T}_{\theta}^d$. Let us show this latter statement.  Define
\be
\delta_j(U_j)=2\pi {\rm i} U_j, \;\; \delta_j (U_k)=0, \quad k\neq j
\ee
(cf. \cite{Connes1980}). These operators $\delta_j$ commute with the involution of $\mathbb{T}_{\theta}^d$
and play the role of the partial derivatives
$\frac{\partial}{\partial x_j}$ on the classical $d$-torus. Let $\triangle= \sum_{j=1}^d \delta_j^2.$ Then $\triangle$ is a negative operator on $L_2(\mathbb{T}_{\theta}^d)$ and its spectrum consists of
the numbers $-4\pi ^2 |m|_2^2, \,m\in \mathbb{Z}^d.$ For any
$\lambda>0,$ we have
 $$
\|(\lambda-\triangle)^{-1}\|\leq \sup_{z\in \sigma(-\triangle)}\frac1{|\lambda+z|}\leq \frac{1}{\lambda}.
 $$
Then by the Hille-Yosida theorem, $\triangle$ is the infinitesimal
generator of a  semigroup of contractions on  $L_2(\mathbb{T}_{\theta}^d)$. Denote this
semigroup by $(T_t)$. Then $T_t=\exp(t\triangle).$ It is easy to check that  $(T_t)$
satisfies the following properties:
\begin{enumerate}[{\rm i)}]

\item $T_t$ is a contraction on $\mathbb{T}_{\theta}^d: \|T_t x\|_{\infty}\leq
\|x\|_{\infty}$ for all $x\in \mathbb{T}_{\theta}^d;$

\item $T_t$ is positive: $T_t x\geq 0$ if $x\geq 0;$

\item $\tau\circ T_t=\tau:$ $\tau(T_t x)=\tau(x)$ for all $x\in \mathbb{T}_{\theta}^d;$

\item $T_t$ is symmetric relative to $\tau:$
$\tau(T_t(y)^*x)=\tau(y^*T_t(x))$ for all $x, y \in
L_2(\mathbb{T}_{\theta}^d).$
 \end{enumerate}
Then $(T_t)$ extends to a semigroup of contractions on $L_p(\mathbb{T}_{\theta}^d)$ for every $1\le p\le\8$. This is the heat semigroup of $\mathbb{T}_{\theta}^d$. The circular Poisson means $\mathbb{P}_r[x]$ is exactly the Poisson semigroup subordinated to $T_t,$ where
$r=e^{-2\pi t}.$ Then by  \cite{JX2007}, we get the part of Theorem \ref{th:FejerPoissonPhiMaxIneq} concerning the circular Poisson means.

The previous argument does not apply to the three other means. However, we can get the type $(p, p)$ inequality for $F_N$ and $P_r$ again from \cite{JX2007} but not with the right estimate on the constant $C_p$. Indeed,  the square Poisson mean $P_r$ is the restriction to the diagonal $(r, ..., r)$ of the following multiple parameter semigroup $P_{(r_1,..., r_d)}$:
 $$
 P_{(r_1,..., r_d)}[x] = \sum_{m \in \mathbb{Z}^d } \hat{x} ( m )r_1^{|m_1|}\cdots r_d^{|m_d|}U^{m}.
 $$
By iteration $P_{(r_1,..., r_d)}$ satisfies a maximal inequality on $L_p(\mathbb{T}_{\theta}^d)$ with a relevant constant controlled by $C^d p^{2d}/(p-1)^{2d}$. It then follows that the map $x\mapsto (P_r[x])_r$ is of type $(p, p)$ with the same constant. Since each Fej\'{e}r mean $F_N$ is majorized by $P_r$ for an appropriate $r$, we deduce that the same maximal inequality holds for $F_N$. We cannot, unfortunately, prove the weak type $(1, 1)$ maximal inequality for $F_N$ and $P_r$ in this way.

The rest of this section is essentially devoted to the proof of Theorem \ref{th:FejerPoissonPhiMaxIneq}. We will use transference and require the following two theorems  which are of interest for their own right. Recall that $\M$ denotes a von Neumann algebra  with a normal semifinite faithful trace $\tau$. $L_\8(\real^d)\overline{\ot}\M$ is equipped with the tensor trace $\nu=dx\ot\tau$, where $dx$ is Lebesgue measure on $\real^d$.

\begin{thm}\label{le:VecRadialPhiMaxIneq}
  Let $\varphi$ be an integrable function on $\real^d$ such that $|\varphi|$ is radial and radially decreasing.
 Let $\varphi_\e(s)=\frac1{\e^d}\, \varphi(\frac s\e)$ for $s\in\real^d$ and $\e>0$.
 \begin{enumerate}[{\rm i)}]
 \item  Let $f\in L_1(\real^d; L_1(\M))$. Then for  any $\a>0$ there exists a projection $e\in L_\8(\real^d)\overline{\ot}\M$ such that
 $$\sup_{\e>0}\big\|e(\f_\e*f)e\big\|_\8\le\a\quad\text{and}\quad \nu(e^\perp)\le C_d\|\f\|_1\frac{\|f\|_1}\a.$$

\item  Let $1<p\le\8$. Then
 $$\big\|{\sup_{\e>0}}^+\f_\e*f\big\|_p\le C_d\|\f\|_1\,\frac{p^2}{ (p-1)^2}\|f\|_p,\quad\forall\; f\in L_p(\real^d; L_p(\M)).$$
\end{enumerate}
  \end{thm}

\begin{proof}
 Let $f\in L_1(\real^d; L_1(\M))$. Without loss of generality, we assume that $f$ is positive. On the other hand, it is easy to reduce the problem to the case where $\f$ is positive too.  Indeed, decomposing $\f$ into its real and imaginary parts, we need only to consider each part separately. Since $f\ge0$, we have
  $${\rm Re} (\f_\e)*f\le |{\rm Re} (\f_\e)|*f\le |\f|_\e*f.$$
This gives the announced reduction. Thus in the sequel we assume that $\f\ge0$.  First take $\varphi$ to be of the form $\varphi=\sum_k \alpha_k
\un_{B_k}$ (a finite sum), where $B_k$ are balls of center 0 and $\a_k\ge0$. Then
\be
\varphi_\e\ast f(s)=\sum_k\alpha_k(\un_{B_k})_\e\ast
f(s)=\sum_k\alpha_k|B_k|M_{\e B_k}(f)(s),
\ee
where $M_{B} (f)(s) = \frac{1}{|B|} \int_{B} f(s-t) dt$ for any ball $B$ centered at $0$. We now appeal to Mei's noncommutative Hardy-Littlewood maximal weak type (1,1) inequality \cite{Mei2007}:  For any $\a>0$ there exists a projection $e\in L_\8(\real^d)\overline{\ot}\M$ such that
 $$\nu(e^\perp)\le C_d\,\frac{\|f\|_1}\a\quad\mbox{and}\quad
   \|e M_{B}(f)e\|_\8\le \a,\; \forall\; \textrm{ball } B \textrm{ centered at } 0.$$
We then deduce that
 $$\|e (\varphi_\e*f) e\|_{\infty}\le C_d\sum_k\alpha_k|B_k| \a=C_d\|\varphi\|_1 \a,\quad\forall\; \e>0.$$
For a general positive $\varphi$, choose an increasing sequence $(\varphi^{(n)})$ of functions of the previous form such that $\varphi^{(n)}$ converges
to $\varphi$ pointwise.  Then for any $\a>0,$ there exists a projection
$e_n\in L_\8(\real^d)\overline{\ot}\M$ such that
 $$\nu(e_n^\perp)\lesssim\frac{\|f\|_1}\a\quad\mbox{and}\quad
  \|e_n (\varphi^{(n)}_\e*f)e_n\|_\8\le \a,\quad\forall\; \e>0.$$
Let $a$ be a $w^{\ast}$-accumulation point of $e_n$. Note that
$$(\varphi_\e^{(n)}\ast f)^{\frac{1}{2}}e_n-(\varphi_\e\ast f)^{\frac{1}{2}}a
=\big((\varphi_\e^{(n)}\ast f)^{\frac{1}{2}}-(\varphi_\e \ast f)^{\frac{1}{2}}\big)e_n
+(\varphi_\e \ast f)^{\frac{1}{2}}(e_n-a).$$ Since
$\varphi^{(n)}\rightarrow \varphi$ increasingly, then $(\varphi_\e^{(n)}\ast f)^{\frac{1}{2}}$ strongly converges to $(\varphi_\e \ast f)^{\frac{1}{2}}.$ Hence $(\varphi_\e^{(n)}\ast f)^{\frac{1}{2}}e_n$
weakly converges to $(\varphi_\e\ast f)^{\frac{1}{2}}a.$ Then we
deduce
 $$\nu(1-a)\lesssim\frac{\|f\|_1}\a\quad\mbox{and}\quad
 \|(\varphi_\e\ast f)^{\frac{1}{2}}a\|_{\infty}\leq\liminf_n\|(\varphi_\e^{(n)}\ast f)^{\frac{1}{2}}e_n\|_{\infty}\leq \a^{\frac{1}{2}}.$$
Let $e=\un_{[\frac{1}{2},\;1]}(a)$, the spectral projection of $a$ corresponding to the interval $[\frac12,\, 1]$. Note that
$1-e=\un_{[\frac{1}{2},\;1]}(1-a).$ Then
$\frac{1}{2}(1-e)\leq 1-a,$ which implies that $\frac{1}{2}\nu
(1-e)\leq \nu (1-a).$ Moreover, letting
$g(r)=\frac{1}{r}\un_{[\frac{1}{2},\;1]}(r), \; r\in
(0,1]$, we have $e= e g(a) a$ and
 $$e (\varphi_\e* f)e=e g(a) [a (\varphi_\e* f) a] e g(a).$$
Since $\|e g(a)\|_{\infty}\leq 2,$ we deduce that
 $$\|e (\varphi_\e* f)e\|_{\infty}\leq 4\|a (\varphi_\e* f)a\|_{\infty}\le 4\a.$$
Therefore  the projection $e$ satisfies:
 $$\nu(e^\perp)\lesssim\frac{\|f\|_1}\a\quad\mbox{and}\quad
  \|e (\varphi_\e*f)e\|_\8 \le 4 \a,\quad \forall\; \e>0.$$
Thus  we get i).

Part ii) is proved by interpolation. It is clear that the map $f\mapsto (\f_\e*f)_{\e>0}$ is of type $(\8,\8)$ with constant $\|\f\|_1$. On the other hand, since we have assumed that $\f\ge0$, $\f_\e*f\ge0$ for $f\ge0$. Thus by the interpolation theorem from \cite{JX2007}, we deduce the desired $(p,p)$ type maximal inequality, i.e., part ii).
 \end{proof}

The conclusion of the previous theorem also holds for another family of functions $\f$ which satisfy an estimate of multiple-parameter nature.

\begin{thm}\label{le:VecSingularPhiMaxIneq}
 Let $\varphi$ be an integrable function on $\real^d$ that has the following
decomposition: $\varphi(s_1,\cdots,s_d)=\varphi_1(s_1)\cdots\varphi_d(s_d),$
where each $\f_k$ satisfies
 $$|\varphi_k(t)| \le \frac{A}{(1+|t|)^{1+\delta}},\quad \forall t \in \mathbb{R},$$
for some $A, \delta >0.$ Then the conclusion of Theorem~\ref{le:VecRadialPhiMaxIneq} remains true.
 \end{thm}

\begin{proof}
 This proof is much more involved than the previous one. Again, we can assume that all functions $\f_k$ are positive. It suffices to show the weak type $(1,1)$ inequality. Fix a positive $f\in L_1(\real^d; L_1(\M))$.  Let
$I_0=[-1,\; 1]$ and $I_k=\{t\in \mathbb{R}: 2^{k-1}<|t|\le2^{k}\}$
for $k=1,2,\ldots$. Also, let $\tilde{I}_k=[-2^k,\; 2^k].$ Split
$\mathbb{R}^d$ into $d!$  regions of the form $|t_{j_1}|\geq
\cdots\ge |t_{j_d}|,$ where $\{j_1,\dots,j_d\}$ is a permutation of the
set $\{1,\dots,d\}$. Then
 $$\varphi_\e \ast f(s) =\sum_{\{j_1,\dots,j_d\}}\int_{|t_{j_1}|\geq\cdots\geq|t_{j_d}|} \varphi(t)f(s-\e t)dt.$$
By symmetry, it suffices to consider one of these regions, say the one where $|y_1| \geq \cdots \geq|y_d|.$  Let
 $$F_\e(s)=\int_{|t_1|\geq\cdots\geq|t_d|} \varphi(t)f(s-\e t)dt,\quad s=(s_1, ..., s_d)\in\real^d.$$
We must show that for any $\a>0$ there exists a projection $e\in L_{\8}(\mathbb{R}^d)\overline{\otimes}\M$ such that
 \beq\label{weak F}
 \nu(e^\perp)\lesssim\frac{\|f\|_1}\a\quad\text{and}\quad \|eF_\e e\|_\8\le\a.
 \eeq
Using the assumption on $\f$ and by change of variables, we have
 \be\begin{split}
 F_\e(s) &=\sum_{k_1=0}^\8\sum_{k_2=0}^{k_1}\cdots\sum_{k_d=0}^{k_{d-1}}\int_{I_{k_1}}\int_{I_{k_2}}
 \cdots\int_{I_{k_d}}\varphi(t)f(s-\e t)dt\\
 &\lesssim\sum_{k_1=0}^\8\sum_{k_2=0}^{k_1}\cdots\sum_{k_d=0}^{k_{d-1}}2^{-k_1(1+\delta)}\cdots2^{-k_d(1+\delta)}\int_{I_{k_1}}\int_{I_{k_2}}
  \cdots\int_{I_{k_d}}f(s-\e t)dt\\
 &\lesssim\sum_{k_1=0}^\8\sum_{k_2=0}^{k_1}\cdots\sum_{k_d=0}^{k_{d-1}}2^{-k_1(1+\delta)}\cdots2^{-k_d(1+\delta)}\int_{\tilde{I}_{k_1}}\int_{\tilde I_{k_2}} \cdots\int_{\tilde{I}_{k_d}}f(s-\e t)dt\\
 &\lesssim\sum_{k_1=0}^\8\sum_{k_2=0}^{k_1}\cdots\sum_{k_d=0}^{k_{d-1}}2^{-(k_1+\cdots +k_d)\delta}
 \frac{1}{|\tilde{I}_{k_1}|^d}\int_{\tilde{I}_{k_1}^d} f(s_1-\e t_1,s_2-2^{k_2-k_1}\e t_2,\cdots, s_d-2^{k_d-k_1}\e t_d)dt.
 \end{split}\ee
Given a function $g\in L_1(\real^d; L_1(\M))$ and a cube $Q\subset\real^d$ centered at $0$ and with sides parallel to the axes put
 $$M_Q(g)(s)=\frac1{|Q|}\int_Qg(s-t)dt,\quad s\in\real^d.$$
Note that this average function appeared already in the proof of Theorem \ref{le:VecRadialPhiMaxIneq} but with balls instead of cubes.  For any fixed $k=(k_1,\cdots,k_d)$ with $k_1\geq k_2\geq \cdots\geq k_d$ let
 $$f_k(z_1,z_2,\cdots,z_d)=f(z_1,2^{k_2-k_1}z_2,\cdots,2^{k_d-k_1}z_d).$$
Then
 $$ \frac{1}{|\tilde{I}_{k_1}|^d}\int_{\tilde{I}_{k_1}^d} f(s_1-\e t_1,s_2-2^{k_2-k_1}\e t_2,\cdots, s_d-2^{k_d-k_1}\e t_d)dt
 =M_{\e\tilde{I}_{k_1}^d}(f_k)(s_1,2^{k_1-k_2}s_2,\cdots,2^{k_1-k_d}s_d).$$
 Thus
 \beq\label{Ft}
 F_\e (s)\lesssim\sum_{k_1=0}^\8\sum_{k_2=0}^{k_1}\cdots\sum_{k_d=0}^{k_{d-1}}2^{-(k_1+ \cdots +k_d)\delta}
 M_{\e\tilde{I}_{k_1}^d}(f_k)(s_1,2^{k_1-k_2}s_2,\cdots,2^{k_1-k_d}s_d).
 \eeq
Now we use again Mei's noncommutative Hardy-Littlewood maximal weak type $(1,1)$ inequality which remains true with balls replaced by cubes.
For any $\a_k>0,$ there exits a projection $e_k$ in $L_\8(\real^d)\overline{\ot}\M$ such that
 \beq\label{weak fk}
 \nu(e_k^\perp)\le C_d\,\frac{\|f_k\|_1}{\a_k}\quad \text{and} \quad \|e_kM_{\e\tilde{I}_{k_1}^d}(f_k)e_k\|_\8\le \a_k,\quad\forall\; \e>0.
 \eeq
Let $T$ be the mapping
 $$(s_1,s_2,\cdots,s_d)\mapsto(s_1,2^{k_1-k_2}s_2,\cdots,2^{k_1-k_d}s_d).$$
$T$ is a homeomorphism of  $\mathbb{R}^d$, so induces an isomorphism of $L_{\8}(\mathbb{R}^d)\overline{\otimes}\M$, still denoted by $T$. Then for any $g\in L_{\8}(\mathbb{R}^d)\overline{\otimes}\M,$ we have
 $$\int \tau (T(g)(s))ds=\int \tau (g\circ T(s))ds=2^{k_2-k_1}\cdots 2^{k_d-k_1}\int \tau (g(s))ds.$$
Let $\tilde{e}_k=T(e_k).$ Then $\tilde{e}_k$ is a projection and
\beq\label{ek}
\nu(\tilde{e}_k^\perp)=2^{k_2-k_1}\cdots2^{k_d-k_1}\nu(e_k^\perp).
\eeq
 On the other hand,
 $$M_{\e\tilde{I}_{k_1}^d}(f_k)(s_1,2^{k_1-k_2}s_2,\cdots,2^{k_1-k_d}s_d)=T\big(M_{\e\tilde{I}_{k_1}^d}(f_k)\big)(s_1, s_2, \cdots, s_d)$$
and
 $$T\big(e_kM_{\e\tilde{I}_{k_1}^d}(f_k)e_k\big)
 =\tilde e_kM_{\e\tilde{I}_{k_1}^d}(f_k)(\cdot,2^{k_1-k_2}\cdot,\cdots,2^{k_1-k_d}\cdot)\tilde e_k.$$
Therefore, by \eqref{weak fk}
 \beq\label{max fk}
 \big\|\tilde e_kM_{\e\tilde{I}_{k_1}^d}(f_k)(\cdot,2^{k_1-k_2}\cdot,\cdots,2^{k_1-k_d}\cdot)\tilde e_k\big\|_\8
 =\big\|e_kM_{\e\tilde{I}_{k_1}^d}(f_k)e_k\big\|_{\8}\le \a_k,\quad\forall\; \e>0.
 \eeq
Let $\a>0$. For each $k$ with $k_1\geq k_2\geq \cdots\geq k_d$ we choose
 $$\a_k = \a\,2^{k_1\delta/(2d)}2^{k_2\delta(1-1/(2d))}\cdots2^{k_d\delta(1-1/(2d))}.$$
Then
 \beq\label{ak}
 2^{-(k_1+\cdots+k_d)\delta}\a_k=\a\,2^{-k_1\delta/2}\,2^{-n_2/(2d)}\cdots2^{-n_d/(2d)},
 \eeq
where $n_2=k_1-k_2,\cdots, n_d=k_1-k_d.$ Note that all $n_j$ are nonnegative integers.  Finally, let $e=\bigwedge_k\tilde{e}_k.$  Then $e$ is a projection in $L_\8(\real^d)\overline{\ot}\M$, and by \eqref{ek}, \eqref{weak fk}, the definition of $f_k$ and the choice of $\a_k$, we have
 \be\begin{split}
 \nu(e^\perp)
 &\le\sum_{k_1=0}^\8\sum_{k_2=0}^{k_1}\cdots\sum_{k_d=0}^{k_{d-1}}\nu(\tilde{e}^{\perp}_k)\\
 &=\sum_{k_1=0}^\8\sum_{k_2=0}^{k_1}\cdots\sum_{k_d=0}^{k_{d-1}}2^{k_2-k_1}\cdots2^{k_d-k_1}\nu(e_k^\perp)\\
 &\le C_d \sum_{k_1=0}^\8\sum_{k_2=0}^{k_1}\cdots\sum_{k_d=0}^{k_{d-1}}2^{k_2-k_1}\cdots2^{k_d-k_1}\frac{\|f_k\|_1}{\a_k}\\
 &\le C_d \|f\|_1\sum_{k_1=0}^\8\sum_{k_2=0}^{\8}\cdots\sum_{k_d=0}^{\8}\frac{1}{\a_k}\lesssim \frac{\|f\|_1}\a.
 \end{split} \ee
On the other hand, for any $\e>0,$ by \eqref{Ft}, \eqref{max fk} and \eqref{ak}
 \be\begin{split}
 \|eF_\e e\|_{\infty}
 &\le\sum_{k_1=0}^\8\sum_{k_2=0}^{k_1}\cdots\sum_{k_d=0}^{k_{d-1}}2^{-(k_1+\cdots+k_d)\delta}
 \big\|\tilde e_kM_{\e\tilde{I}_{k_1}^d}(f_k)(\cdot,2^{k_1-k_2}\cdot,\cdots,2^{k_1-k_d}\cdot)\tilde e_k\big\|_\8\\
 &\le \sum_{k_1=0}^\8\sum_{k_2=0}^{k_1}\cdots\sum_{k_d=0}^{k_{d-1}}2^{-(k_1+\cdots+k_d)\delta}\a_k\\
 &\le \a\sum_{k_1\ge0}\sum_{n_2\geq0}\cdots\sum_{n_d\geq0}
 2^{-k_1\delta/2}\,2^{-n_2/(2d)}\cdots2^{-n_d/(2d)}\lesssim\a.
 \end{split} \ee
Thus we get the desired estimate \eqref{weak F}, so finish the proof of the theorem.
 \end{proof}

We also require the following lemma for the proof of Theorem \ref{th:FejerPoissonPhiMaxIneq}.

\begin{lem}\label{weak-sub}
 Let $\N$ be a $w^*$-closed involutive subalgebra of $\M$ that is the image of a normal conditional expectation $\E$. Let $(x_n)$ be a sequence of  positive operators
 in $L_1(\N)$. Assume that for any $\a>0$  there exists a projection $\tilde e\in\M$ such that
 $$\sup_n\|\tilde ex_n\tilde e\|_\8\le\a \quad\text{and}\quad \tau(\tilde e^\perp)\le\frac{C}\a.$$
Then there exists a projection $e\in\N$ such that
 $$\sup_n\|ex_ne\|_\8\le4\a \quad\text{and}\quad \tau( e^\perp)\le\frac{2C}\a.$$
\end{lem}

\begin{proof} Let $a=\E(\tilde e)$. Then $a\in\N$ and
 $$\|ax_n^{1/2}\|_\8=\|\E(\tilde ex_n^{1/2})\|_\8\le \a^{1/2}.$$
As in the proof of Theorem \ref{le:VecRadialPhiMaxIneq}, we then see that  $e=\un_{[1/2,\, 1]}(a)$ is the desired projection in $\N$.
 \end{proof}

\noindent\emph{Proof of Theorem \ref{th:FejerPoissonPhiMaxIneq}.} We will identify the $d$-torus $\T^d$ with the cube $\mathbb I^d=[0,\,1]^d\subset\real^d$ (with $\mathbb I=[0,\,1]$) via
$(e^{2 \pi \mathrm{i} s_1},\cdots , e^{2 \pi \mathrm{i} s_d})\;\leftrightarrow\; (s_1, \cdots, s_d)$.  Accordingly, $\N_\theta=L_\8(\T^d)\overline{\ot}\T^d_\theta$ is viewed as a subalgebra of $\M_\theta=L_\8(\real^d)\overline{\ot}\T^d_\theta$; the associated conditional expectation is just the multiplication by the indication function $\un_{\mathbb I^d}$ of $\mathbb I^d$. Thus  $\wt{\T^d_\theta}$  becomes a subalgebra of $\M_\theta$ too. The corresponding conditional expectation is $\un_{\mathbb I^d}\cdot\mathbb E$, where $\mathbb E$ is the conditional expectation from $\N_\theta$ to $\wt{\T^d_\theta}$ given by Corollary~\ref{prop:TransLp}.

Now let us show the weak type $(1,1)$ inequality for the Fej\'er means.  Recall that $F_N$ is the Fej\'er kernel  on $\T^d$ given by \eqref{squareFejerKer} and that
 $$F_N(s_1, \cdots, s_d)=G_N(s_1)\cdots G_N(s_d),$$
where $G_N$ is the $1$-dimensional Fej\'er kernel. It is a well-known  elementary fact that
 $$G_N(s)\le\frac{\pi^2}2\,\frac{N+1}{1+(N+1)^2|s|^2}.$$
Thus
 $$F_N(s_1, \cdots, s_d)\lesssim\frac1{\e^d}\, \eta(\frac{s_1}\e)\cdots \eta(\frac{s_d}\e)=\eta_\e(s_1)\cdots \eta_\e(s_d),$$
 where $\eta (s) = (1 + |s|^2)^{-1}$ and  $\e=(N +1)^{-1}$. Let $x\in L_1(\T^d_\theta)$. Writing $x$ as a linear combination of four positive elements, we can assume $x\ge0$. Using transference, we have that $\tilde x\in L_1(\wt{\T^d_\theta})\subset L_1(\N_\theta)$ and
 \be \begin{split}
 \wt{F_N[x]}( s_1, \cdots, s_d)
 &=F_N*\tilde x(s_1, \cdots, s_d)\\
 & = \int_{\mathbb I^d} F_N (s_1-t_1, \cdots, s_d-t_d)\, \tilde x(t_1, \cdots, t_d) d t\\
 & =\int_{\mathbb{R}^d} F_N (s_1-t_1, \cdots, s_d-t_d) \un_{\mathbb I^d}(t_1, \cdots, t_d)\,\tilde x(t_1, \cdots, t_d) d t.
 \end{split}\ee
Therefore, we are in a situation of applying Theorem~\ref{le:VecSingularPhiMaxIneq}, so for any $\a>0$ there exists a projection $\tilde e\in\M_\theta$ such that
 $$\sup_N\|\tilde e \wt{F_N[x]}\tilde e)\|_\8\le\a\quad\text{and}\quad
 \nu(\tilde e^\perp)\lesssim\frac{ \|\un_{\mathbb I^d}\tilde x\|_{L_1(\M_\theta)}}\a=\frac{ \|x\|_{1}}\a.$$
Since $x\ge0$, $\wt{F_N[x]}\ge0$ for every $N$. Thus by Lemma~\ref{weak-sub}, we get the desired weak type $(1,1)$ inequality for $F_N$. Similarly, we show the type $(p,p)$ inequality. The same argument works equally for the square Poisson means $P_r$.

It remains  to show the part of the theorem concerning $\Phi^\e$ (which contains the circular Poisson mean $\mathbb P_r$ as a special case). We will use the convolution formula \eqref{eq:PhiFunConv}. Note that for maximal inequalities on  $\Phi^\e$ we do not need all conditions on $\Phi$ and $\f$ in \eqref{eq:PhiFuncCondition}. What we really need here is the last growth assumption on $\f$ there:
 $$|\f(s)|\le\frac{A}{(1+|s|)^{d+\delta}},\quad s\in\real^d.$$
 Then like in the proof of Theorem~\ref{le:VecRadialPhiMaxIneq} we can assume that $\f$ is nonnegative. In this case the kernel $K_\e$ is nonnegative too.  Moreover, replacing $\f$ by the function on the right hand side above, we can further suppose that $\f$ satisfies the assumption of Theorem~\ref{le:VecRadialPhiMaxIneq}.
 Now let $x\in L_1(\T^d_\theta)$. Without loss of generality, assume  again $x\ge0$. By  \eqref{eq:PhiFunConv}, for $s=(s_1, \cdots, s_d)\in \I^d$ we have
 \be\begin{split}
 \wt{\Phi^\e[x]}( s)
 &=\int_{\mathbb I^d} K_\e (s-t)\,\tilde x(t) d t\\
 &=\sum_{m \in \mathbb{Z}^d} \int_{\mathbb I^d} \varphi_{\varepsilon} (s -t + m) \,\tilde x (t) dt\\
 &= \int_{\mathbb I^d} \varphi_{\varepsilon} (s -t) \,\tilde x (t) dt
  + \sum_{m\neq\mathbf{0}} \int_{\mathbb I^d} \varphi_{\varepsilon} (s -t + m) \,\tilde x (t) dt\; .
 \end{split}\ee
The first term on the right can be dealt with in the same way as before for $F_N$:
 $$ \int_{\mathbb I^d} \varphi_{\varepsilon} (s -t) \,\tilde x (t) dt=
  \int_{\mathbb R^d} \varphi_{\varepsilon} (s -t) \un_{\mathbb I^d}(t)\,\tilde x (t) dt.$$
Then by Theorem~\ref{le:VecRadialPhiMaxIneq} for any $\a>0$ there exists a projection $\tilde e_1\in \M_\theta$ such that
 $$ \nu(\tilde e_1^\perp)\lesssim\frac{ \|x\|_{1}}\a \quad\text{and}\quad
  \big\|\tilde e_1 \big[\int_{\mathbb I^d} \varphi_{\varepsilon} (\cdot -t) \,\tilde x (t) dt\big]\tilde e_1\big\|_\8\le\a,\quad \forall\;\e>0.$$
On the other hand, for $s,t \in \mathbb I^d$ and $m\neq\mathbf0$ we have
 $$\varphi_{\varepsilon} (s -t + m) \lesssim\frac1{\e^d}\,(1+\frac{|m|}\e)^{- d - \delta}.$$
 Note that
 $$\sum_{m\not=0}\frac1{\e^d}\,(1+\frac{|m|}\e)^{- d - \delta}
 \approx\frac1{\e^d}\sum_{1\le|m|\le\e}+\e^\delta\sum_{\e<|m|}\frac1{|m|^{d +\delta}}\lesssim 1.       $$
 Hence (recalling that $x\ge0$),
 $$\sum_{m \not= 0} \int_{\mathbb I^d}\varphi_{\varepsilon} (s -t + m) \,\tilde x (t) dt
 \lesssim \sum_{m\not=0}\frac1{\e^d}\,(1+\frac{|m|}\e)^{- d - \delta} \int_{\mathbb I^d}\tilde x (t) dt
 \lesssim  \int_{\mathbb I^d} \tilde x (t) dt.$$
The last integral is an operator in $L_1(\wt{\T^d_\theta})$ and its $L_1$-norm is less than or equal to that of $x$. Thus there exists a projection $\tilde e_2\in \wt{\T^d_\theta}$ such that
 $$ \nu(\tilde e_2^\perp)\lesssim\frac{ \|x\|_{1}}\a \quad\text{and}\quad
  \big\|\tilde e_2 \big[ \int_{\mathbb I^d} \tilde x (t) dt\big]\tilde e_2\big\|_\8\le\a.$$
Let $\tilde e=\tilde e_1\vee \tilde e_2$. Then $\tilde e$ is a projection in $\M_\theta$, and combining the preceding two parts we get
 $$ \nu(\tilde e^\perp)\lesssim\frac{ \|x\|_{1}}\a \quad\text{and}\quad
  \big\|\tilde e\,\wt{\Phi^\e[x]}\,\tilde e\big\|_\8\le\a,\quad \forall\;\e>0.$$
We then deduce the weak type $(1,1)$ inequality for $\Phi^\e$ thanks to  Lemma~\ref{weak-sub}.  The type $(p,p)$ inequality is proved similarly. Therefore, the proof of Theorem \ref{th:FejerPoissonPhiMaxIneq} is complete.
 \hfill $\Box$


\section{Pointwise convergence}\label{PointwiseConvg}


In this section we apply the maximal inequalities proved in the
previous section to study the pointwise convergence of Fourier
series on quantum tori. To this end we first need an appropriate
analogue for the noncommutative setting of the usual almost
everywhere convergence. This is the almost uniform convergence
introduced by Lance \cite{Lance1976}.

Let $(x_{\lambda})_{\lambda \in \Lambda}$ be a family of elements in
$L_p(\M).$ Recall that
$(x_{\lambda})_{\lambda\in\Lambda}$ is said to converge almost
uniformly to $x,$ abbreviated as $x_{\lambda} \xrightarrow{a.u} x,$
if for every $\epsilon>0$ there exists a projection $e \in
\M$ such that
 \be \tau(1 - e) < \epsilon \quad\text{and} \quad
 \lim_{\lambda}\|( x_{\lambda} - x )e\|_{\8} =0.\ee
Also, $( x_{\lambda})_{\lambda \in \Lambda}$ is said to
converge bilaterally almost uniformly to $x ,$ abbreviated as
$x_{\lambda} \xrightarrow{b.a.u} x,$ if the limit above is replaced by
 $$ \lim_{\lambda} \|e ( x_{\lambda} - x )e\|_{\infty} = 0.$$
In the commutative case, both convergences are equivalent to the usual
almost everywhere convergence thanks to Egorov's theorem. However,
they are different in the noncommutative setting.

\begin{thm}\label{th:PWConver}
 Let $1\le p\leq\infty$ and $x \in L_p (\mathbb{T}^d_{\theta})$.
Then $F_N[x] \xrightarrow{b.a.u.} x$ as $N \to\infty$. Moreover, for $2\le p \leq \8$ the   b.a.u. convergence can
be strengthened to a.u. convergence.

Similar statements hold for the two Poisson means $P_r$, $\mathbb P_r$ as $r\to\8$ as well as for the mean $\Phi^\e$ as $\e\to0$.
 \end{thm}

\begin{proof}
 Let $x\in L_1(\mathbb{T}_{\theta}^d)$ and $\epsilon>0.$  Let $(\varepsilon_m)$ and $(\delta_m)$ be two sequences of small positive numbers. Then for each $m\geq 1$ choose  $y_m\in\A_\theta$ such that $\|x-y_m\|_1\le\delta_m$. Let $z_m=x-y_m$, so $x=y_m+z_m$.  Applying Theorem \ref{th:FejerPoissonPhiMaxIneq} to each $z_m$, we find a projection $e_m$ such that
 $$\sup_N \|e_m F_N[z_m] e_m\|_\8 \leq \varepsilon_m \quad \text{and} \quad \tau(e_m^\bot)\le C\|z_m\|_1\varepsilon_m^{-1}\le C \delta_m\varepsilon_m^{-1}.$$
The first inequality implies that
 $$ \|e_m z_m e_m\|_\8 \leq \varepsilon_m.$$
 Let $e=\bigwedge_m e_m.$ Then
  $$\tau(e^{\bot})\leq C\sum_m \delta_m \varepsilon_m^{-1}<\epsilon$$
 provided $\varepsilon_m$ and $\delta_m$ are appropriately chosen. On the other hand,
 \be\begin{split}
  \|e(F_N[x]-x)e\|_\8
 &\le  \|e(F_N[y_m]-y_m)e\|_\8 +\|eF_N[z_m]e\|_\8+\|ez_me\|_\8\\
  &\le  \|F_N[y_m]-y_m\|_\8 +2\e_m.
 \end{split}\ee
By Proposition \ref{th:MeanConver},
 $$\lim_{N\to\8} \|F_N[y_m]-y_m\|_\8=0$$
for $y_m\in\A_\theta$. It then follows that
 $$\limsup_{N\to\8}\|e(F_N[x]-x)e\|_\8\le2\e_m.$$
Whence $\lim_{N\to\8}\|e(F_N[x]-x)e\|_\8=0$. Therefore,  $F_N[x]$ converges to $x$ b.a.u. The b.a.u. convergence statements  for the other summation methods are proved exactly in the same way.

Let us turn to the a.u. convergence. Let $x\in L_2(\mathbb{T}_{\theta}^d)$  and $\epsilon>0.$ We can assume $x$ selfadjoint.  As in the preceding argument, let $x=y_m+z_m$ with $y_m\in\A_\theta$ and $\|z_m\|_2\le\delta_m$. Both $y_m$ and $z_m$ can be chosen selfadjoint. Now applying Theorem~\ref{th:FejerPoissonPhiMaxIneq} to $y_m^2$, we find a projection $e_m$ such that
$$\sup_N \|e_m F_N[z_m^2] e_m\| _\8\leq \varepsilon_m \quad \text{and} \quad
\tau(e_m^\bot)\le C\varepsilon_m^{-1} \tau(z_m^2)\le C \varepsilon_m^{-1} \delta_m^2.$$
Since the map $z\mapsto F_N[z]$ is positive, by Kadison's Cauchy-Schwarz inequality \cite{Kadison1952}, we have
 $$\big(F_N[z_m]\big)^2\le F_N[z_m^2].$$
Thus
 \beq\label{side one weak}
 \|F_N[z_m]e_m\|_\8^2\leq \|e_m F_N[z_m^2] e_m\|_\8\le \varepsilon_m.
 \eeq
Let $e=\bigwedge_m e_m$. Then $\tau(e^{\bot})\le\epsilon$ for appropriate $\e_m$ and $\delta_m$ and
 $\lim_N\|(F_N[x]-x)e\|_\8=0.$
Therefore,  $F_N[x] \xrightarrow{a.u.} x$. The proof of the corresponding statements for $P_r$ and $\mathbb P_r$ is the same.

However, a minor extra argument is required for the mean $\Phi^\e$ because the map $z\mapsto \Phi^\e[z]$ is not positive in general. So we cannot apply directly  Kadison's  inequality to this map. But what is really missing is the one-sided weak type $(1, 1)$ maximal inequality \eqref{side one weak} for $\Phi^\e$ instead of $F_N$. In order to show this latter inequality, we can assume, as in the proof of Theorem \ref{th:FejerPoissonPhiMaxIneq},  that $\f$ is nonnegative. Then  the kernel $K_\e$ in \eqref{eq:PhiFunConv} is nonnegative too. Thus the map $z\mapsto K_\e*\tilde z$ is positive, so we can apply Kadison's inequality to this map. Then as before for $F_N$, we get the desired inequality  \eqref{side one weak} with  $F_N$ replaced by $\Phi^\e$, and then deduce that   $\Phi^\e[x] \xrightarrow{a.u.} x$ as $\e\to0$.   Therefore, the theorem is completely proved.
  \end{proof}


\section{Bochner-Riesz means}\label{BRMean}


As pointed out in section~\ref{MeanConverg}, when $\alpha > (d-1)/2,$ the function $\Phi$ and
$\varphi$ associated with the Bochner-Riesz mean satisfy
\eqref{eq:PhiFuncCondition}. Therefore, by Proposition~\ref{th:MeanConver}, Theorems~\ref{th:FejerPoissonPhiMaxIneq} and \ref{th:PWConver}, we get the following

\begin{prop}\label{prop:BocherRieszMeanAbove}
Let $\alpha>(d-1)/2$ and $x\in L_p(\mathbb{T}^d_{\theta})$ with $1\le p\le\8$. Then
\begin{enumerate}[{\rm i)}]

\item $\displaystyle\lim_{R \to \8} B^{\alpha}_R[x] = x$ in $L_p (\mathbb{T}^d_{\theta})$ $($relative to the $w^*$-topology for $p=\8)$.

\item $\displaystyle\big \|{ \sup_{R>0}}^{+} B^{\alpha}_R[x] \big \|_p \lesssim\|x\|_p$ for $p>1$.

\item  $\displaystyle B^{\alpha}_R[x]\xrightarrow{b.a.u}x$ as $R \to \8.$

\end{enumerate}

\end{prop}

If $\alpha$ is below the critical index $(d-1)/2,$ the above results usually fail even in the
scalar case, see for example  \cite[ VII.4]{SW1975}. However, we
have the following theorem, i.e., Theorem \ref{th:BocherRieszMeanBelow}, which is the noncommutative analogue of
Stein's theorem \cite{Stein1958} (see also \cite[VII.5]{SW1975}).

\begin{thm}\label{th:BocherRieszMeanBelow}
Let $1<p<\infty$ and $\alpha>(d-1)|\frac{1}{2}-\frac{1}{p}|.$ Then for any $x\in L_p(\mathbb{T}^d_{\theta})$

\begin{enumerate}[{\rm i)}]

\item $\displaystyle \big \|{\sup_{R>0}}^{+} B^{\alpha}_R[x] \big \|_p \lesssim\|x\|_p$ with the relevant constant depends only $p$, $d$ and $\a$.

\item $\displaystyle\lim_{R \to \8} B^{\alpha}_R[x] = x$ in $L_p (\mathbb{T}^d_{\theta}).$

\item $\displaystyle B^{\alpha}_R[x]\xrightarrow{b.a.u}x$ as $R \to \8.$

\end{enumerate}
 \end{thm}

\begin{proof}
 The  hard part of the theorem is the maximal inequality i). Assuming this part, it is easy to show the two others. Indeed, i) implies that for any $R>0$
 $$\big\| B^{\alpha}_R [x]\big\|_p \le \big\|{\sup_{r>0}}^+ B^{\alpha}_r [x] \big\|_p\lesssim\|x\|_p,\quad \forall\; x \in L_p (\T_\theta^d).$$
Whence
 $$\sup_{R>0} \big\| B^{\alpha}_R \big\|_{L_p \to L_p} < \8.$$
Together with the density of polynomials in $L_p(\mathbb{T}^d_{\theta})$, this implies the mean convergence in ii).
The pointwise convergence iii) can be proved as Theorem \ref{th:PWConver}. The only thing to note is the fact that the type $(p,p)$ maximal inequality in i) implies the corresponding weak type $(p,p)$ inequality. The details are left to the reader.

\smallskip

The remainder of this section is devoted to the proof of i). We will follow the patten set up by Stein in the classical setting. The proof is quite technical and complicated, but essentially everything is based on two main ideas: estimate maximal function and square function by duality and interpolation.

We will frequently use the duality between $L_{p'}(\mathbb{T}^d_{\theta};\ell_1)$ and $L_{p}(\mathbb{T}^d_{\theta};\ell_\8)$ ($p'$ being the conjugate index of $p$). For the convenience of the reader we recall this duality.   $L_{p'}(\mathbb{T}^d_{\theta};\ell_1)$ is defined to be the space of all sequences $y=(y_n)$ in $L_{p'}(\mathbb{T}^d_{\theta})$ which can be decomposed as
 $$y_n=\sum_{k\geq 1}u_{kn}^{\ast}v_{kn}, \quad \forall n \geq 1$$
for two families $(u_{kn})_{k,n\geq1}$ and $(v_{kn})_{k,n\geq 1}$ in $L_{2p'}((\mathbb{T}^d_{\theta})$ such that
 $$\sum_{k,n\geq 1}u^{\ast}_{kn}u_{kn}\in L_{p'}(\mathbb{T}^d_{\theta})\;\; \text{and}\;\; \sum_{k,n\geq 1}v^{\ast}_{kn}v_{kn}\in L_{p'}(\mathbb{T}^d_{\theta}).$$
$L_{p'}(\mathbb{T}^d_{\theta};\ell_1)$ is equipped with the norm
$$\|y\|_{L_{p'}(\mathbb{T}^d_{\theta};\ell_1)}=\inf \|\sum_{k,n\geq 1}u^{\ast}_{kn}u_{kn}\|_{p'}^{1/2}\|\sum_{k,n\geq 1}v^{\ast}_{kn}v_{kn}\|_{p'}^{1/2},$$
where the infimum runs over all decompositions of $y$ as above.  It is easy to see that if $y_n\ge0$ for all $n$, then $(y_n)\in L_{p'}(\mathbb{T}^d_{\theta};\ell_1)$ iff $\sum_ny_n\in L_{p'}(\mathbb{T}^d_{\theta})$. In this case, we have
 $$\|y\|_{L_{p'}(\mathbb{T}^d_{\theta};\ell_1)}=\big\|\sum_ny_n\big\|_{p'}.$$
Let $1\leq {p'} <\infty$. Then the dual space of $L_{p'}(\mathbb{T}^d_{\theta};\ell_1)$ is $L_{p}(\mathbb{T}^d_{\theta};\ell_{\infty})$. The duality bracket is given by
 $$\la x,\;y\ra=\sum_n\tau(x_ny_n),\quad x=(x_n)\in L_{p}(\mathbb{T}^d_{\theta};\ell_\8),\;y=(y_n)\in L_{p'}(\mathbb{T}^d_{\theta};\ell_1).$$
We refer to \cite{Junge2002} and  \cite{JX2007} for more information.

\medskip

For clarity we divide the proof of i) into three steps.

\medskip\noindent {\it Step 1}. If $\alpha \in \mathbb{C}$ and $\mathrm{Re}(\alpha) >
\frac{d-1}{2},$ then for $1<p\leq\infty,$
 $$\big \|{\sup_{R>0}}^{+} B^{\alpha}_R [x ] \big \|_p \lesssim\|x\|_p,\quad \forall\; x \in  L_p(\mathbb{T}^d_{\theta}). $$
To this end, choose $\delta >0 $ and $\beta \in \mathbb{C}$ such
that $\mathrm{Re}( \alpha )> \delta> \frac{d-1}{2}$ and $\alpha =
\delta + \beta.$ We have the following identity
 \beq\label{recurrence BR}
 B^{\alpha}_R =C_{\beta, \delta} R^{-2\alpha} \int^R_0 (R^2 - t^2)^{\beta -1} t^{2\delta +1} B^{\delta}_t dt,
 \eeq
where $C_{\beta,\delta} = 2 \Gamma (\beta + \delta +1)/ [\Gamma(\delta +1) \Gamma(\beta)].$ Let $(R_n)$ be a sequence in $(0,\8)$ and $(y_n)$ an
element in the unit ball of $L_{p'} (\mathbb{T}^d_{\theta}; \el_1)$. Then, for any $x \in L_p(\mathbb{T}^d_{\theta})$ we have
 \be\begin{split}
 \Big | \tau \Big (  \sum_n B^{\alpha}_{R_n} [x] y_n \Big ) \Big |
 &=  |C_{\beta, \delta}| \Big | \sum_n R^{-2\alpha}_n \int^{R_n}_0 (R^2_n - t^2)^{\beta -1} t^{2 \delta +1} \tau \big ( B^{\delta}_t [x] y_n \big ) d t \Big |\\
 &\le|C_{\beta, \delta}|\int^1_0 |(1 - t^2)^{\beta -1} t^{2 \delta +1}| \Big |  \tau \Big ( \sum_n B^{\delta}_{t R_n} [x] y_n \Big ) \Big | d t \\
 &\le  |C_{\beta, \delta}| \int^1_0 |(1 - t^2)^{\beta -1} t^{2 \delta +1}| d t \big \| \mathrm{sup}^+_{R>0} B^{\delta}_R [x] \big \|_p\\
 & \lesssim \| x \|_p,
 \end{split}\ee
where we have used Proposition \ref{prop:BocherRieszMeanAbove} ii)
in the last inequality and the fact that
 $$ \int^1_0 |(1 -t^2)^{\beta -1} t^{2 \delta +1}| d t
 = \int^1_0 (1 -t^2)^{\mathrm{Re}(\beta) -1} t^{2 \delta +1} d t<\8$$
since $\mathrm{Re}(\beta) = \mathrm{Re} (\alpha) -\delta> 0$ and $\delta >0.$ By duality we then deduce the desired
maximal inequality.

\medskip

\noindent{\it Step 2}. If $\alpha>0,$ then
 \beq\label{eq:RieszMeanMaxInequL2}
 \big \| {\sup_{R>0}}^{+}B^{\alpha}_R [x]\big\|_2 \lesssim \|x\|_2,\quad \forall\; x \in L_2(\mathbb{T}^d_{\theta}).
 \eeq
We first consider the case of $\alpha >1/2.$ Choose $\beta > 1$ such that $\alpha = \beta +
\delta$ with $\delta > - 1/2.$ By \eqref{recurrence BR}
 \be\begin{split}
 B^{\beta + \delta}_R
 &=  - C_{\beta, \delta} R^{-2(\beta+ \delta)} \int^R_0 \Big ( \int^t_0 B^{\delta}_r d r \Big )
  \big [ (R^2 - t^2)^{\beta -1} t^{2 \delta +1} \big ]' d t\\
 &= C_{\beta, \delta} \int^1_0 \varphi (t) M^{\delta}_{R t} d t,
 \end{split}\ee
where 
 $$M^{\delta}_t = \frac{1}{t} \int^t_0 B^{\delta}_r d r\quad\text{and}\quad 
 \varphi (t)= 2 (\beta -1) (1 - t^2)^{\beta -2} t^{2 \delta + 3} - (2\delta +1) (1 - t^2)^{\beta -1} t^{2 \delta+1}. $$
Note that $\int^1_0 | \varphi (t) | d t < \8.$ We will  use the following fact that for any $(x_n) \in L_2 (\T_{\theta}^d; \el_{\8})$ one has
 $$ \big \|{\sup_n}^+ x_n \big \|_2 \approx \sup \Big \{ \Big | \sum_n \tau(x_n y_n ) \Big |\;:\;
 y_n \in L^+_2 (\T_{\theta}^d),\; \Big \| \sum_ny_n \Big \|_2 \le 1 \Big \}$$
with universal equivalence constants (see \cite{Junge2002, JX2007}).
In what follows, we fix  $x\in L_2(\mathbb{T}_{\theta}^d)$ and always assume that $(R_n)$ is a sequence in
$(0, \8)$ and $(y_n)$ a sequence of positive elements in $L_2(\T_{\theta}^d)$ with $\| \sum_n y_n \|_2 \le 1.$
Since
 \be\begin{split}
 \Big | \tau \Big (  \sum_n B^{\alpha}_{R_n} [x] y_n \Big ) \Big |
 & = |C_{\beta, \delta}| \Big | \tau \Big ( \sum_n \Big ( \int^1_0 \varphi (t) M^{\delta}_{R_n t} (x) d t \Big ) y_n \Big ) \Big |\\
 & \le |C_{\beta, \delta}| \int^1_0  | \varphi (t)|   \Big | \tau \Big ( \sum_n M^{\delta}_{R_n t} (x) y_n \Big ) \Big | dt\\
 & \lesssim \big \| \mathrm{sup}^+_{R >0} M^{\delta}_R(x) \big \|_2 \int^1_0 | \varphi (t) | d t,
 \end{split}\ee
where we have used duality in the last inequality. We then deduce that
 $$\big \| {\sup_{R>0}}^{+} B^{\alpha}_R[x ]\big\|_2 \lesssim\big \| {\sup_{R >0}}^+ M^{\delta}_R(x)\big\|_2 .$$
Now we must show that
 \beq\label{eq:RieszMeanMaxIntInequL2}
 \big \| {\sup_{R >0}}^+ M^{\delta}_R(x)\big\|_2 \lesssim \|x\|_2\quad  \text{if} \quad \delta> - 1/2.
 \eeq
To this end, we again use duality. We have
 \be\begin{split}
 \Big |\tau \Big ( \sum_n M^{\delta}_{R_n} (x) y_n \Big ) \Big |
 &\le \Big |\tau \Big ( \sum_n M^{\delta + 1}_{R_n} (x) y_n \Big ) \Big |
 + \Big |\tau \Big ( \sum_n \big [ M^{\delta+1}_{R_n} (x) - M^{\delta}_{R_n} (x) \big ] y_n \Big ) \Big |\\
 &\le \big \|{\sup_{R>0}}^{+} M^{\delta + 1}_R (x ) \big \|_2
 + \Big |\tau \Big ( \sum_n G^{\delta}_{R_n} (x) y_n \Big ) \Big |,
 \end{split}\ee
where $G^{\delta}_R (x) = M^{\delta+1}_R(x) - M^{\delta}_R(x).$
Using the following elementary inequality
 $$|\tau (ab)|^2 \le \tau(|a| b ) \tau (| a^*| b),\quad \forall\;a, b \in\T^d_{\theta}\;\text{with}\; b \ge 0,$$
we have
 $$ \Big |\tau \Big ( \sum_n G^{\delta}_{R_n} (x) y_n \Big ) \Big |^2
 \le\tau \Big ( \sum_n \big | G^{\delta}_{R_n} (x) \big | y_n \Big )
  \tau \Big ( \sum_n \big | G^{\delta}_{R_n} (x)^* \big | y_n \Big ).$$
Note that
 \be\begin{split}
 \big | G^{\delta}_R (x)\big |
 &= \Big | \frac{1}{R} \int^R_0 \big [ B^{\delta + 1}_r [x] - B^{\delta}_r [x] \big ] d r \Big |\\
 & \le \Big ( \int^R_0 \big | B^{\delta + 1}_r [x] - B^{\delta}_r [x] \big |^2 \frac{d r}{R} \Big )^{1/2}
 \le G^{\delta} (x),
 \end{split}\ee
where
 $$G^{\delta} (x) = \Big ( \int^{\8}_0 \big | B^{\delta+ 1}_r [x] - B^{\delta}_r [x] \big |^2 \frac{ d r}{r}\Big )^{1/2}.$$
It then follows that
 $$ \tau \big ( \sum_n \big | G^{\delta}_{R_n} (x)\big | y_n \big )
 \le \tau \big ( G^{\delta} (x) \sum_n y_n \big )\le \| G^{\delta} (x) \|_2  \big \| \sum_n y_n \big \|_2 \le \| G^{\delta} (x) \|_2.$$
Similarly,
 $$\tau \big(\sum_n \big| G^{\delta}_{R_n} (x)^*\big| y_n \big ) \le \| G^{\delta}_* (x) \|_2,$$
where
 $$G^{\delta}_* (x)=\Big( \int^{\8}_0 \big |\big(B^{\delta + 1}_r [x] - B^{\delta}_r [x]\big)^*\big |^2
  \frac{ d r}{r}\Big )^{1/2}.$$
Combining the preceding inequalities, we obtain
 $$\| {\sup_{R>0}}^{+} M^{\delta}_R (x) \big \|_2
 \le \big \|{\sup_{R>0}}^{+} M^{\delta + 1}_R (x) \big \|_2 + \|G^{\delta} (x) \|_2^{1/2} \| G^{\delta}_* (x) \|_2^{1/2}.$$
We now claim that
 $$\max\big(\| G^{\delta} (x) \|_2,\; \| G^{\delta}_*(x) \|_2 \big)\lesssim \| x \|_2,\quad \text{if}\quad \delta > -1/2.$$
Indeed, by Parseval's identity we have
 \be\begin{split}
 \| G^{\delta} (x) \|_2^2
 & =\int^{\8}_0  \tau \big(\big|B^{\delta + 1}_r[x]  - B^{\delta}_r[x] \big |^2\big) \frac{ d r}{r}\\
 & = \int^{\8}_0 \sum_{|m|_2 \le R} \Big | \Big ( 1 - \frac{|m|^2_2}{r^2} \Big )^{\delta +1} -
 \Big ( 1 - \frac{|m|^2_2}{r^2} \Big )^{\delta} \Big |^2|\hat x(m)|^2\frac{d r}{ r}\\
 & = \sum_{m \not= 0}|\hat x(m)|^2 \int^{\8}_{|m|_2} \frac{|m|^4_2}{r^4} \Big ( 1- \frac{|m|^2_2}{r^2} \Big )^{\delta} \frac{d r}{ r}\\
 &\lesssim\|x\|^2_2
 \end{split}\ee
because the integral
 $$\int^{\8}_{|m|_2} \frac{|m|^4_2}{r^4}\Big ( 1- \frac{|m|^2_2}{r^2} \Big )^{\delta} \frac{d r}{ r} =
 \int^{\8}_1 r^{-5} (1 - r^{-2})^{2 \delta} d r < \8$$
if $\delta > -1/2.$
In the same way, we have
 $$\| G^{\delta}_* (x) \|_2\lesssim \|x\|_2.$$
Hence our claim is proved. Consequently,
 $$\big \|{\sup_{R>0}}^{+} M^{\delta}_R(x ) \big \|_2 \lesssim\big \| {\sup_{R>0}}^{+} M^{\delta + 1}_R(x) \big \|_2 +\|x\|_2.$$
Then by iteration,  for any positive integer $k$ we have
 $$\big \| {\sup_{R>0}}^{+}M^{\delta}_R (x) \big \|_2 \lesssim \big\|{\sup_{R>0}}^{+}M^{\delta + k}_R (x) \big \|_2 + \| x \|_2. $$
Now, if we choose $k $ such that $\delta + k > (d-1)/2$, then using {\it Step 1}, we have
 $$\big \|{\sup_{R>0}}^+ M^{\delta + k}_R [x] \big \|_2\leq
 \big \| {\sup_{R>0}}^+ B^{\delta + k}_R [x] \big\|_2 \lesssim\| x \|_2.$$
Therefore, we deduce  \eqref{eq:RieszMeanMaxIntInequL2},  and hence
\eqref{eq:RieszMeanMaxInequL2} provided $\alpha > 1/2.$

We now deal with the general case of $\alpha > 0.$ Choose
$\beta > 1/2$ and $\delta > -1/2$ so that $\alpha = \beta + \delta.$
Then by \eqref{recurrence BR}
 \be\begin{split}
 B^{\beta + \delta}_R  - \frac{C_{\beta,\delta}}{C_{\beta, \delta +1}} B^{\beta + \delta +1}_R
 &=C_{\beta,\delta} R^{-2(\beta+\delta)}\Big [ \int^R_0 (R^2 - t^2)^{\beta -1} t^{2 \delta +1} B^{\delta}_t d t\\
 &\hskip .5cm - R^{-2}\int^R_0 (R^2 - t^2)^{\beta -1} t^{2 (\delta+1) +1} B^{\delta +1}_t d t\Big ]\\
 &= C_{\beta,\delta} R^{-2(\beta+\delta)}\Big [ \int^R_0 (R^2 - t^2)^{\beta -1} t^{2 \delta +1} (B^{\delta}_t- B^{\delta+1}_t )d t\\
 &\hskip .5cm + \int^R_0 (R^2 - t^2)^{\beta -1} t^{2 \delta + 1} (1 - R^{-2}t^2) B^{\delta +1}_t d t\Big ]\\
 \triangleq & {\rm I}_R + {\rm II}_R.
 \end{split}\ee
We first estimate ${\rm I}_R$. By the argument already used above
 $$ \big | \tau \big ( \sum_n {\rm I}_{R_n} (x) y_n \big )\big |^2
 \le \tau \big ( \sum_n |{\rm I}_{R_n} (x)| y_n \big ) \tau \big (\sum_n |{\rm I}_{R_n} (x)^*| y_n \big).$$
However,
 \be\begin{split}
 |{\rm I}_R (x) |
 &=|C_{\beta,\delta}| R^{-2(\beta+\delta)} \Big |\int^R_0 (R^2 - t^2)^{\beta -1} t^{2 \delta +1}
 \big(B^{\delta + 1}_t [x] - B^{\delta}_t [x] \big)d t \Big |\\
 &\le |C_{\beta,\delta}| R^{-2(\beta+\delta)} \Big ( \int^R_0 \big | (R^2 - t^2)^{\beta -1} t^{2 \delta +1}\big |^2 d t \Big )^{1/2}\\
 & \; \times R^{1/2} R^{-1/2} \Big ( \int^R_0 \big| B^{\delta + 1}_t [x] - B^{\delta}_t [x] \big|^2 d t\Big )^{1/2}\\
 &\lesssim G^{\delta} (x)
 \end{split}\ee
because the integral
 $$R^{1- 4(\beta+\delta)} \int^R_0 \big|(R^2 - t^2)^{\beta -1} t^{2 \delta +1}\big |^2 d t
 = \int^1_0 | (1 -t^2)^{\beta -1} t^{2 \delta +1} |^2 d t < \8$$
when $\beta >1/2.$ Similarly,
 $$| {\rm I}_R (x)^{\ast} | \lesssim G^{\delta}_* (x).$$
Hence, we deduce
 $$\big\| {\sup_{R>0}}^+ {\rm I}_R (x)\big\|_2 \lesssim \| G^{\delta} (x)\|^{1/2}_2 \|G^{\delta}_* (x)\|^{1/2}_2 \lesssim \|x\|_2. $$
Next, we estimate the second term ${\rm II}_R.$ Since
 \be\begin{split}
 {\rm II}_R
 & = C_{\beta,\delta} R^{-2(\beta+\delta)} \int^R_0 (R^2 - t^2)^{\beta -1} t^{2 \delta + 1} (1 - R^{-2}t^2) B^{\delta +1}_t d t\\
 & = C_{\beta,\delta} R^{-2(\beta+\delta)-2} \int^R_0 (R^2 -t^2)^{\beta} t^{2 \delta + 1} B^{\delta +1}_t d t
 \end{split}\ee
and $\beta > 1/2,$  ${\rm II}_R$ can be dealt with as $B^{\alpha}_R$ in the
case of $\alpha > 1/2$. So we conclude that
 $$\big\|{\sup_{R>0}}^+ B_R [x ]\big\|_2 \lesssim \|x\|_2.$$
Therefore, we have finally arrived at
 \be\begin{split}
 \big\|{\sup_{R>0}}^+B^{\beta + \delta}_R (x)\big\|_2
 &\le \frac{|C_{\beta,\delta}|}{|C_{\beta, \delta +1}|}
 \big\|{\sup_{R>0}}^+ B^{\beta + \delta + 1}_R [x]\big\|_2 \\
 & \hskip .5cm+ \big\|{\sup_{R>0}}^+ {\rm I}_R (x)\big\|_2 + \big\|{\sup_{R>0}}^+ {\rm II}_R [x ]\big\|_2\\
 &\lesssim \| x\|_2.
 \end{split}\ee
This completes the proof of {\it Step 2}.

\medskip\noindent{\it Step 3}. When $p$ is near $1$ or $\8,$ the announced result is
in fact already contained in {\it Step 1}. Moreover, {\it Step 2}
gives  the desired inequality in the special case of $p=2.$ The
general case can be deduced from these special ones by applying
Stein's complex interpolation. To this end, we need first a
strengthening of \eqref{eq:RieszMeanMaxInequL2} which allows the order
$\alpha$ to be complex, that is,
\beq\label{eq:RieszMeanMaxInequL2Complex}
 \big \|{\sup_{R>0}}^{+} B^{\alpha}_R [x]\big\|_2 \lesssim\|x\|_2,\quad\a\in\com,\;\mathrm{Re}(\alpha) >0.
 \eeq
This can be reduced to the case of $\a>0$ by using the argument in {\it Step 1}. We omit  the details.

Let $x \in L_p (\mathbb{T}^d_{\theta})$ with $\| x
\|_{p} < 1$ and $y =
(y_n)$ be a finite sequence in $L_{p'} (\T^d_{\theta})$  with $\| y \|_{L_{p'} (\T^d_{\theta};\el_1)} <1.$ Assume first that $p < 2.$ For any fixed $\alpha >
(d-1)(1/p - 1/2)$ we can always choose $p_1 > 1, \alpha_0 > 0$ and
$\alpha_1 > (d-1)/2$ such that \
 $$\alpha = (1-t) \alpha_0 + t\alpha_1\quad \text{and}\quad
 \frac{1}{p} = \frac{1-t}{2} +\frac{t}{p_1}$$
for some $0<t<1.$ Define
 $$ f(z) = u|x|^{\frac{p(1-z)}{2}+ \frac{p z}{p_1}},\quad z \in \mathbb{C},$$
where $x= u |x|$ is the polar decomposition of $x$.  On the other hand, by Proposition 2.5 of
\cite{JX2007}, there is a function $g =(g_n)_n$ continuous on the
strip $\{ z \in \mathbb{C}:\; 0 \le \mathrm{Re} (z) \le 1 \}$ and
analytic in the interior such that $g(t) = y$ and
 $$ \sup_{s \in \mathbb{R}} \max \big \{ \big \| g( \mathrm{i} s ) \big \|_{L_2(\T^d_{\theta};\el_1)},
 \big \| g( 1 + \mathrm{i} s ) \big\|_{L_{p_1'} (\T^d_{\theta}; \el_1)} \big \} < 1. $$
Fix a sequence $(R_n)\subset(0,\,\8)$ and $\delta>0$. We define
 $$F(z)= \exp\big(\delta(z^2 - t^2)\big )\sum_n \tau \big (B^{(1-z) \alpha_0 + z \alpha_1}_{R_n} [f(z)]g_n (z) \big ). $$
$F$ is a function analytic in the open
strip $\{z \in \mathbb{C}\;:\; 0 < \mathrm{Re} (z) < 1 \}.$ By
\eqref{eq:RieszMeanMaxInequL2Complex}, for any $s\in\real$ we have
 \be\begin{split}
 |F(\mathrm{i} s ) |
 &\le\exp\big ( -\delta( s^2 +t^2)\big) \big \| \big ( B^{\alpha_0 + \mathrm{i} s (\alpha_1-\alpha_0)}_{R_n}
 (f(\mathrm{i} s)) \big )_n \big \|_{L_2(\T^d_{\theta}; \el_{\8})}\big \| g (\mathrm{i} s) \big\|_{L_2 (\T^d_{\theta};\el_1)}\\
 &\lesssim \| f(\mathrm{i} s) \|_2\lesssim1.
 \end{split}\ee
Similarly, by {\it Step 1} we have
 $$|F(1 + \mathrm{i} s ) | \lesssim 1. $$
Therefore, by the maximum principle we get $| F(t)|  \lesssim 1$ i.e., \
 $$ \big|\tau \big ( \sum_n B^{\alpha}_{R_n} [x] y_n \big ) \big | \lesssim 1$$
if $\|x \|_{L_p(\N_{\theta})} <1.$ Then by duality  and homogeneity, we
deduce that
 $$ \big\|{\sup_{R>0}}^+ B^{\alpha}_R [x] \big\|_p\lesssim \|x\|_p, \quad \forall\; x \in  L_p (\mathbb{T}^d_{\theta}).$$
The argument for the case of $p>2$ is similar once we begin by setting $p_1 =\8.$
Thus the proof of Theorem \ref{th:BocherRieszMeanBelow} is complete.\end{proof}

\begin{rk}
The previous proof  gives a slightly more general result by allowing $\alpha$ to be complex. Namely, Theorem \ref{th:BocherRieszMeanBelow} remains true under the assumption that $\mathrm{Re} (\alpha)>(d-1)|\frac{1}{2}-\frac{1}{p}|$ with $\alpha \in\mathbb{C}$ and $1< p < \8.$
\end{rk}

\begin{rk}
 Let $\M$ be a semifinite von Neumann algebra. Then Theorem \ref{th:BocherRieszMeanBelow} admits the following analogue for the algebra $\T^d\overline{\ot}\M$ with the same proof:
Let $1<p \le \infty$ and $\mathrm{Re} (\alpha)>(d-1)|\frac{1}{2}-\frac{1}{p}|.$
Then
  $$ \big \| {\sup_{R>0}}^{+} B^{\alpha}_R [f ] \big \|_p \lesssim \| f \|_p,\quad
 \forall\; f \in L_p (\mathbb{T}^d; L_p (\M)).$$
Moreover, $B^{\alpha}_R [f ]$ converges b.a.u. to $f$ as $R\to\8$. Here
 $$
  B^{\alpha}_R[f] = \sum_{| m|_2 \leq R} \Big ( 1 -\frac{|m|_2^2}{R^2} \Big )^{\alpha} \hat{f}(m) z^m
  $$
 for $f \in L_p (\mathbb{T}^d; L_p (\M))$ with Fourier series expansion
 $$f\sim \sum_{m\in\ent^d}\hat f(m) z^m.$$
\end{rk}


\section{Fourier multipliers}\label{CBMultiplier}


It is  our intention in this section to study Fourier multipliers
in the quantum $d$-torus $\T^d_\theta$. We will compare (completely) bounded $L_p$ Fourier multipliers with those in the usual $d$-torus $\T^d$. The right framework for this investigation is the category of operator spaces.

We now recall some standard operator space notions and refer the reader to \cite{ER2000} and \cite{Pisier2003} for more information. A (concrete) operator space is a closed subspace $E$
of $\mathcal{B} (H)$ for some Hilbert space $H.$
Then $E$ inherits the matricial structure of $\mathcal{B}(
H)$ via the embedding $\mathbb{M}_n (E)
\subset \mathbb{M}_n (\mathcal{B} (H)).$ More precisely,
let $\mathbb{M}_n (E)$ denote the space of $n \times n$
matrices with entries in $E,$ equipped with the norm
induced by $\mathcal{B} (\el^n_2 (H))$.  An abstract matricial norm
characterization of operator spaces was given by Ruan. The morphisms in the category of operator spaces are
completely bounded maps. Let $H, K$ be two Hilbert
spaces. Suppose that $E \subset \mathcal{B} (H)$ and
$F \subset \mathcal{B} (K)$ are two operator
spaces. A map $u:\; E \rightarrow F$ is called
completely bounded (in short c.b.) if
 $$\sup_n\| \mathrm{id}_{\mathbb{M}_n}\ot u \|_{\mathbb{M}_n (E) \to \mathbb{M}_n(F)}<\8,$$
and the c.b. norm $\| u \|_{\mathrm{cb}}$
is defined to be the above supremum. We denote by $\mathrm{CB} (E,
F)$ the space of all c.b. maps from $E$ to
$F$,  equipped with the norm $\|\,\|_{\mathrm{cb}}.$ This is a Banach space.

For an operator space $E$ there exists a natural matricial
structure on the Banach dual $E^*$ of $E$ so that
$E^*$ becomes an operator space too. The norm of $\mathbb{M}_n
( E^* )$ is that of $\mathrm{CB} (E,
\mathbb{M}_n)$ ($\mathbb{M}_n=\mathbb{M}_n(\com)$). This is usually called the standard dual of
$E.$ We will simply say the dual of $\mathbb{E}$ since only standard duals are used in the sequel.

We will need the natural operator space structure on noncommutative $L_p$-spaces introduced by Pisier. Let $\M$ be a (semifinite) von Neumann algebra on a Hilbert space $H$. Then the embedding $\M\subset \mathcal{B} (H)$ gives to $\M$ an operator space structure.  To equip $L_1(\M)$ with an operator space structure, we view $L_1(\M)$ as the predual of the opposite algebra $\M^{\rm op}$ instead of $\M$ itself. In this way,  $L_1(\M)$ becomes a subspace of the dual operator space of $\M^{\rm op}$. This is the natural operator space structure of  $L_1(\M)$. Then for any $1<p<\8$ the operator space structure of $L_p(\M)$ is defined via the complex interpolation formula $L_p(\M)=\big(L_\8(\M),\, L_1(\M)\big)_{1/p}$.  We refer the reader to \cite{Pisier1998, Pisier2003} for more details.

We will use the following fundamental property of c.b. maps between two noncommutative $L_p$-spaces due to Pisier \cite{Pisier1998}. Let $\N$ be another (semifinite) von Nuemann algebra. Then a map $u: L_p(\M)\to L_p(\N)$ is c.b. iff ${\rm id}_{S_p}\ot u: L_p(\mathcal{B}(\el_2)\overline{\ot}\M)\to L_p(\mathcal{B}(\el_2)\overline{\ot}\N)$ is bounded. In this case, 
 $$\|u\|_{\rm cb}=\big\|{\rm id}_{S_p}\ot u: L_p(\mathcal{B}(\el_2)\overline{\ot}\M)\to L_p(\mathcal{B}(\el_2)\overline{\ot}\N)\big\|.$$
Here $S_p$ denotes the Schatten $p$-class, namely, the noncommutative $L_p$-space associated to $\mathcal{B}(\el_2)$ equipped with the usual trace. The readers who are not very familiar with operator space theory can take this property as the definition of c.b. maps between noncommutative $L_p$-spaces.

Now we turn to Fourier multipliers on quantum tori. Let $\phi = (\phi_m)_{m\in \mathbb{Z}^d}\subset\com$. We define $T_{\phi}$ by
 $$\widehat{T_{\phi} x} (m) = \phi_{m}\hat{x} (m), \; \forall m \in\mathbb{Z}^d,$$
for any polynomial $x \in \mathcal{P}_{\theta}.$ We call $\phi$ a bounded $L_p$ multiplier
(resp. c.b. $L_p$ multiplier) on the quantum torus $\mathbb{T}^d_{\theta}$ if $T_{\phi}$ extends to a bounded (resp. c.b.)
map on $L_p (\mathbb{T}^d_{\theta})$. The space of all $L_p$ multipliers (resp.
c.b. $L_p$ multipliers) on $\mathbb{T}^d_{\theta}$ is denoted
by $\mathrm{M} (L_p (\mathbb{T}^d_{\theta}))$ (resp. $\mathrm{M}_{\mathrm{cb}} (L_p
(\mathbb{T}^d_{\theta}))$), equipped with the natural norm (resp. c.b. norm). When $\theta=0$, we recover the Fourier multipliers on the usual $d$-torus $\T^d$. The corresponding multiplier spaces are denoted by $\mathrm{M} (L_p (\mathbb{T}^d))$ and  $\mathrm{M}_{\mathrm{cb}} (L_p(\mathbb{T}^d))$, respectively.

The following remark summarizes some  easily checked basic properties of quantum Fourier multipliers. We only state them for c.b. case, although all of them are equally valid for bounded multipliers.

\begin{rk}\label{prop:MuliplierElement}
 Let $1 \le p, p' \le \8$ with $\frac{1}{p}+\frac{1}{p'}=1$.

\begin{enumerate}[{\rm i)}]

\item $\mathrm{M}_{\mathrm{cb}} (L_p (\mathbb{T}^d_{\theta}))$ is a Banach algebra under pointwise multiplication.

\item $\mathrm{M}_{\mathrm{cb}} (L_p (\mathbb{T}^d_{\theta})) = \mathrm{M}_{\mathrm{cb}} (L_{p'} (\mathbb{T}^d_{\theta})).$

\item $\mathrm{M}_{\mathrm{cb}} (L_q (\mathbb{T}^d_{\theta}))\subset\mathrm{M}_{\mathrm{cb}} (L_p (\mathbb{T}^d_{\theta}))$, a contractive inclusion for $2\le p\le q\le\8$.

\item $\mathrm{M}_{\mathrm{cb}} (L_2 (\mathbb{T}^d_{\theta})) = \mathrm{M} (L_2 (\mathbb{T}^d_{\theta})) = \ell_{\infty}(\mathbb{Z}^d)$ with equal norms.    \end{enumerate}
 \end{rk}

It is well-known that in the classical case Fourier multipliers are closely related to Schur multipliers. We will exploit such a relation in the quantum case too. To this end we first recall the definition of Schur multipliers. Let $\Lambda$ be an index set. The elements of $\B(\ell_2(\Lambda))$ are represented by infinite matrices in the canonical basis of $\ell_2(\Lambda)$. A complex function $\psi=(\psi_{st})$  on $\Lambda\times\Lambda$ (or matrix indexed by $\Lambda$)  is called a bounded Schur multiplier on
$\B(\ell_2(\Lambda))$ if for every operator $a =(a_{st})\in\B(\ell_2(\Lambda))$, the matrix $(\psi_{st}a_{st})$ represents a bounded operator
on $\ell_2(\Lambda)$. We then denote
$M_{\psi}a=(\psi_{st}a_{st}).$ In this case, $M_{\psi}$ is necessarily bounded on $\B(\ell_2(\Lambda))$. More generally, for $1\le p\le\8$, if  $M_\psi$ induces a bounded map on the Schatten $p$-class $S_p(\ell_2(\Lambda))$ based on $\ell_2(\Lambda)$, we call $\psi$ a bounded Schur multiplier on  $S_p(\ell_2(\Lambda))$. Similarly, we  define the completely boundedness of $M_{\psi}$.

Fourier and Schur multipliers are linked together via  Toeplitz matrices. As usual, we represent $\mathbb{T}^d_{\theta}$ as a von Neumann algebra on
$L_2(\mathbb{T}^d_{\theta})$ by left multiplication. For every $x\in \mathbb{T}^d_{\theta},$ let $[x]$ denote the
representation matrix of $x$ on $\ell_2 (\mathbb{Z}^d)$ in the orthonormal basis $(U^m)_{m\in\ent^d}$. Namely,
 $$[x]=\big( \langle xU^n,\;U^m\rangle \big)_{m,n\in \mathbb{Z}^d}.$$
Let $\tilde{\theta}$ be the following $d\times d$-matrix deduced from the skew symmetric matrix $\theta$:
 \[\tilde{\theta}=-2\pi\begin{pmatrix}
 0 & \theta_{12} & \theta_{13} &\dots & \theta_{1d}\\
 0 & 0           & \theta_{23} &\dots & \theta_{2d}\\
 \hdotsfor{5}\\
 0 & 0           & 0           &\dots & \theta_{d-1,d}\\
 0 & 0           & 0           &\dots & 0
 \end{pmatrix}.\]
Then by the commutation relation \eqref{eq:CommuRelation}, we have
 $$xU^n=\sum_k\hat{x}(k)U^kU^n
 =\sum_k\hat{x}(k)U_1^{k_1}\cdots U_d^{k_d}U_1^{n_1}\cdots U_d^{n_d}
 =\sum_k\hat{x}(k)e^{\mathrm{i}n\tilde{\theta} k^t}U^{k+n},$$
where $n=(n_1,\dots,n_d)$, $k^t$ is the transpose of $k=(k_1,\dots,k_d)$ and $n \tilde{\theta} k^t$ denotes the
matrix product. Thus
 \begin{equation}\label{Toeplitz}
 [x]=\Big(\hat{x}(m-n)e^{\mathrm{i} n \tilde{\theta} (m-n)^t}\Big )_{m,n\in\mathbb{Z}^d}.
 \end{equation}
If $\theta=0$,   $[x]$ is a Toeplitz matrix. In the general case,  $[x]$ is a twisted Toeplitz
matrix.

For $\phi = (\phi_m)_{m \in \mathbb{Z}^d}\in\el_{\8}(\mathbb{Z}^d),$ we have
 \beq\label{F-S}
 \big[T_{\phi}x\big]=\big(\phi_{m-n}\hat{x}(m-n)e^{\mathrm{i} n\tilde{\theta}(m-n)^t}\big)_{m,n\in\mathbb{Z}^d}
 =M_{\tilde{\phi}}([x]),
 \eeq
where $\tilde{\phi}_{mn}=\phi_{m-n}.$ This is the link between the Fourier and Schur
multipliers associated to $\phi$. This link remains valid for operators $x$ in $\B(\el_2)\overline{\ot}\mathbb{T}^d_{\theta}$. In this case, the entries of the twisted Toeplitz matrix $[x]$ are operators in $\B(\el_2)$.

To illustrate the usefulness of the relationship above, let us show the following simple result.

\begin{prop}\label{prop:cbL8Multiplier}
We have
 $$\mathrm{M}_{\mathrm{cb}} (\mathbb{T}^d_{\theta})
 = \mathrm{M}_{\mathrm{cb}} (L_{\8} (\mathbb{T}^d))= \mathrm{M} (L_\8(\mathbb{T}^d))\quad \text{with equal norms}.$$
\end{prop}

\begin{proof}
 The argument below is standard. Let  $\Gamma_{\infty}$ denote the subspace of  $\B(\ell_2(\mathbb{Z}^d))$  consisting of all twisted Toeplitz matrices of the form \eqref{Toeplitz}.
 By the preceding discussion, for any $x\in \mathbb{T}^d_{\theta}$ we have
$$\|T_{\phi}(x)\|_{\infty}=\|T_{\phi}(x)\|_{\B(L_2(\mathbb{T}^d_{\theta}))}=\|[T_{\phi}(x)]\|_{\B(\ell_2(\mathbb{Z}^d))}
=\|M_{\tilde{\phi}}[x]\|_{\B(\ell_2(\mathbb{Z}^d))}.$$
Consequently,
 $$T_{\phi}\; \text{is bounded} \; \text{on}\; \mathbb{T}^d_{\theta}\; \Longleftrightarrow\;
 M_{\tilde{\phi}}\big|_{\Gamma_{\infty}}: \Gamma_{\infty}\to \Gamma_{\infty}\; \text{is bounded}.$$
Moreover, in this case,
 $$ \|T_{\phi}\|=\big\|M_{\tilde{\phi}}\big|_{\Gamma_{\infty}}\big\|. $$
Considering the vector-valued case where $x\in\B(\el_2)\overline{\ot} \mathbb{T}^d_{\theta}$, we get the c.b. analogue of the above equivalence:
  $$T_{\phi}\; \text{is c.b.} \; \text{on}\; \mathbb{T}^d_{\theta}\; \Longleftrightarrow\;
 M_{\tilde{\phi}}\big|_{\Gamma_{\infty}}\; \text{is c.b.} \; \text{on}\; \Gamma_{\infty}
 \quad\text{and}\quad \|T_{\phi}\|_{\rm cb}=\big\|M_{\tilde{\phi}}\big|_{\Gamma_{\infty}}\big\|_{\rm cb}. $$
Thus, if $M_{\tilde{\phi}}$ is c.b. on $\B(\ell_2(\mathbb{Z}^d)),$ then
$M_{\tilde{\phi}}\big|_{\Gamma_{\infty}}$ is c.b. on $\Gamma_{\infty}$,  so is
$T_{\phi}$ on $\mathbb{T}^d_{\theta}.$

Conversely,  suppose $\phi\in \mathrm{M}_{\mathrm{cb}} (
\mathbb{T}^d_{\theta}).$  Let
$V= \mathrm{diag} (\cdots,U^n,\cdots)$ be the diagonal matrix with diagonal
entries $(U^n)_{n\in \mathbb{Z}^d}$.  $V$ is a unitary operator in
$\B(\ell_2(\mathbb{Z}^d))\overline{\otimes}\mathbb{T}^d_{\theta}$. For any $a=(a_{mn})_{m,n\in
\mathbb{Z}^d}\in\B (\ell_2(\mathbb{Z}^d)),$  let  $x=V(a\ot 1_{\mathbb{T}^d_{\theta}})V^*\in \B(\ell_2(\mathbb{Z}^d))\overline{\otimes}\mathbb{T}^d_{\theta}$, where $1_{\mathbb{T}^d_{\theta}}$ denotes the unit of $\T^d_\theta$. Then
 $$x=(U^ma_{mn}U^{-n})_{m,n \in\mathbb{Z}^d}
 =\sum_{m,n}a_{mn}e_{mn}\otimes U^mU^{-n}
 =\sum_{m,n}a_{mn}e_{mn}\otimes e^{- \mathrm{i} n\tilde{\theta}m^{t}}U^{m-n},$$
where $(e_{mn})$ are the canonical matrix units of $\B(\ell_2(\mathbb{Z}^d)).$ Since $V$ is unitary, we have
 $$\|x\|_{\B(\ell_2(\mathbb{Z}^d))\overline{\otimes}\mathbb{T}^d_{\theta}}=\|a\|_{\B(\ell_2(\mathbb{Z}^d))}.$$
On the other hand,
 $$(\mathrm{id}_{\B(\ell_2(\mathbb{Z}^d))}\otimes T_{\phi} ) (x)
 =\sum_{m,n}\phi_{m-n}a_{mn}e_{mn}\otimes
 e^{-\mathrm{i} n\tilde{\theta}m^{t}}U^{m-n}=V (M_{\tilde{\phi}}(a) \ot 1_{\mathbb{T}^d_{\theta}})V^{*}.$$
It then follows that
  \be\begin{split}
 \|M_{\tilde{\phi}}(a)\|_{\B(\ell_2(\mathbb{Z}^d))}
 &=\|(\mathrm{id}_{\B(\ell_2(\mathbb{Z}^d))}\otimes T_{\phi} ) (x)\|_{\B(\ell_2(\mathbb{Z}^d))\overline{\otimes}\mathbb{T}^d_{\theta}}\\
 &\leq \|T_{\phi}\|_{\mathrm{cb}}\|x\|_{\B(\ell_2(\mathbb{Z}^d))\overline{\otimes}\mathbb{T}^d_{\theta}}
 =\|T_{\phi}\|_{\mathrm{cb}}\|a\|_{\B(\ell_2(\mathbb{Z}^d))}.
 \end{split}\ee
Therefore, $\tilde{\phi}$ is a bounded Schur multiplier on $\B(\ell_2(\mathbb{Z}^d)).$
Considering matrices $a=(a_{mn})_{m,n\in\mathbb{Z}^d}$ with entries in $\B(\el_2)$, i.e.,
$a=(a_{mn})_{m,n\in\mathbb{Z}^d}\in\B(\el_2)\overline{\ot}\B (\ell_2(\mathbb{Z}^d)),$ we show in the same way that  $M_{\tilde{\phi}}$ is c.b. on
$\B(\ell_2(\mathbb{Z}^d))$, so $\tilde{\phi}$ is a c.b. Schur multiplier on $\B(\ell_2(\mathbb{Z}^d))$ and $\|M_{\tilde{\phi}}\|_{\rm cb}\le \|T_{\phi}\|_{\mathrm{cb}}$.

In summary, we have proved that
 $$T_{\phi}\;\text{is c.b. on}\;\;\mathbb{T}^d_{\theta}\; \Longleftrightarrow \;
 M_{\tilde{\phi}}\;\text{is c.b. on}\;\; \B(\ell_2(\mathbb{Z}^d)).$$
Applying this result to the commutative case ($\theta=0$), we get that
 $$T_{\phi}\; \text{is c.b. on}\;\; L_{\infty}(\mathbb{T}^d)\; \Longleftrightarrow\;
 M_{\tilde{\phi}}\;\text{is  c.b. on}\;\; \B(\ell_2(\mathbb{Z}^d)).$$
Therefore,
 $$\mathrm{M}_{\rm cb}(\T^d_\theta)=\mathrm{M}_{\rm cb}(L_\8(\T^d))\quad\text{with equal norms}.$$
However, it is well known that a Fourier multiplier $\phi$ is bounded on $L_\8(\T^d)$ iff it is the Fourier transform of a bounded Borel measure $\mu$ on $\T^d$. In this case, $T_\phi$ is the convolution operator by $\mu$ and its norm is equal to $\|\mu\|$. Then it is easy to check that $T_\phi$ c.b. on $L_\8(\T^d)$. Thus
 \beq\label{cb=b}
 \mathrm{M}_{\rm cb}(L_\8(\T^d))=\mathrm{M}(L_\8(\T^d))\quad\text{with equal norms}.
 \eeq
Combining the preceding results, we deduce the announced assertion. \end{proof}

The main result of this section is the following theorem, which extends the first equality in the previous proposition to all $1\le p\le\8$. We point out that the inclusion
$\mathrm{M}_{\mathrm{cb}} (L_{p} (\mathbb{T}^d))\subset\mathrm{M}_{\mathrm{cb}} (L_p (\mathbb{T}^d_{\theta}))$ was proved independently by  Junge, Mei and Parcet \cite{JMP2011}.

\begin{thm}\label{th:cbLpMultiplier}
 Let $1 < p < \8$. Then  $\mathrm{M}_{\mathrm{cb}} (L_p (\mathbb{T}^d_{\theta}))
 =\mathrm{M}_{\mathrm{cb}} (L_{p} (\mathbb{T}^d))$ with equal norms.
\end{thm}

\begin{proof}
  The inclusion $\mathrm{M}_{\mathrm{cb}} (L_{p} (\mathbb{T}^d))\subset\mathrm{M}_{\mathrm{cb}} (L_p (\mathbb{T}^d_{\theta}))$ can be easily proved by transference. Indeed, let $\phi\in \mathrm{M}_{\mathrm{cb}} (L_{p} (\mathbb{T}^d))$, and let $x\in L_p(\B(\el_2)\overline{\ot}\mathbb{T}^d_{\theta})$ be a polynomial in $U$:
    $$x=\sum_{m\in\ent^d} \hat x(m)\ot U^m,$$
where only a finite number of coefficients $\hat x(m)$  are nonzero operators in $S_p$. Let
 $$\tilde x(z)=\sum_{m\in\ent^d} \hat x(m)\ot U^mz^m,\quad z\in\T^d.$$
Then $\tilde x\in L_p(\T^d; L_p(\B(\el_2)\overline{\ot}\mathbb{T}^d_{\theta}))$ and
 $$T_\phi(\tilde x)=\wt{T_\phi(x)},$$
where the first $T_\phi$ is viewed as a multiplier on $\T^d$ and the second on $\T^d_\theta$. Recall that $\mathbb{T}^d_{\theta}$ is hyperfinite, so the algebra $\B(\el_2)\overline{\ot}\mathbb{T}^d_{\theta}$ can be approximated by matrix algebras. Therefore, the complete boundedness of $\T_\phi$ on $L_p(\T^d)$ implies
 $$\big\| \wt{T_\phi(x)}\big\|_{L_p(\T^d; L_p(\B(\el_2)\overline{\ot}\mathbb{T}^d_{\theta}))}
 \le\|\phi\|_{\mathrm{M}_{\mathrm{cb}} (L_{p} (\mathbb{T}^d))}\|\tilde x\|_{L_p(\T^d; L_p(\B(\el_2)\overline{\ot}\mathbb{T}^d_{\theta}))}.$$
However,  by Corollary \ref{prop:TransLp}
 $$\big\| \wt{T_\phi(x)}\big\|_{L_p(\T^d; L_p(\B(\el_2)\overline{\ot}\mathbb{T}^d_{\theta}))}
 =\|T_\phi(x)\|_{L_p(\B(\el_2)\overline{\ot}\mathbb{T}^d_{\theta})}$$
 and
 $$\|\tilde x\|_{L_p(\T^d; L_p(\B(\el_2)\overline{\ot}\mathbb{T}^d_{\theta}))}
 =\|x\|_{L_p(\B(\el_2)\overline{\ot}\mathbb{T}^d_{\theta})}.$$
Thus
 $$\|T_\phi(x)\|_p\le \|\phi\|_{\mathrm{M}_{\mathrm{cb}} (L_{p} (\mathbb{T}^d))}\|x\|_p.$$
Whence $T_\phi$ is c.b., so $\phi\in\mathrm{M}_{\mathrm{cb}} (L_p (\mathbb{T}^d_{\theta}))$  and
$\|\phi\|_{\mathrm{M}_{\mathrm{cb}} (L_p (\mathbb{T}^d_{\theta}))}\le \|\phi\|_{\mathrm{M}_{\mathrm{cb} }(L_p (\mathbb{T}^d))}$.

For the converse inclusion, note that the argument in the second part of the proof of Proposition~\ref{prop:cbL8Multiplier} works  equally at the level of $L_p$-spaces. Thus we get  that
 $$T_{\phi}\;\text{is c.b. on}\;\;L_p(\mathbb{T}^d_{\theta})\; \Longrightarrow \;
 M_{\tilde{\phi}}\;\text{is c.b. on}\;\; S_p(\ell_2(\mathbb{Z}^d)).$$
Then using Neuwirth and Ricard's transference theorem \cite{NR2011}, we deduce that $T_\phi$ is c.b. on $L_p(\T^d)$, so $\mathrm{M}_{\mathrm{cb}} (L_p (\mathbb{T}^d_{\theta}))\subset \mathrm{M}_{\mathrm{cb} }(L_p (\mathbb{T}^d))$ contractively.

However, for reason of completeness, we include a self-contained proof in the spirit of the proof of Proposition~\ref{prop:cbL8Multiplier} by adapting Neuwirth and Ricard's argument  to the present setting of twisted Toeplitz matrices. Moreover, this proof does not need the first part above.
 Let
 $$Z_N=\{-N,\dots,-1,0,1,\dots,N\}^d\subset\mathbb{Z}^d.$$
$(Z_N)$ is a F{\o}lner sequence of $\mathbb{Z}^d,$ that is,
 $$\lim_{N\rightarrow\infty}\frac{|Z_N\triangle(Z_N+n)|}{|Z_N|}=0,\quad \forall n\in \mathbb{Z}^d.$$
Define two maps $A_N$ and $B_N$ as follows:
 $$A_N:\; \mathbb{T}^d_{\theta}\to\B(\ell_2^{|Z_N|}) \quad \text{with}\quad x\mapsto P_N([x]),$$
where $P_N: \; \B(\ell_2(\mathbb{Z}^d))\rightarrow \B(\ell_2^{|Z_N|})$ with $(a_{mn})\mapsto
(a_{mn})_{m,n\in Z_N}$. And
 $$B_N:\; \B(\ell_2^{|Z_N|})\to \mathbb{T}^d_{\theta}\quad \text{with}\quad
  e_{mn}\mapsto \frac{1}{|Z_N|}e^{-\mathrm{i} n \tilde{\theta} (m-n)^t}U^{m-n}.$$
Here $\B(\ell_2^{|Z_N|})$ is endowed with the normalized trace. It is easy to check that both $A_N, B_N$ are  unital, completely positive
and trace preserving. Consequently, $A_N$  extends to a complete contraction from
$L_p(\mathbb{T}^d_{\theta})$ into $L_p(\B(\ell_2^{|Z_N|})),$ while $B_N$ a complete contraction from $L_p(\B(\ell_2^{|Z_N|}))$
into $L_p(\mathbb{T}^d_{\theta}).$

We now claim that
 $\lim_{N\rightarrow\infty}B_N\circ A_N(x)=x$ in
$L_p(\mathbb{T}^d_{\theta})$ for any $x\in
L_p(\mathbb{T}^d_{\theta}).$ It suffices to consider a monomial $x=U^k.$ Then
 $$A_N(U^{k})=\big(e^{\mathrm{i} n \tilde{\theta} (m-n)^t} \big)_{m,n\in Z_N,m-n=k},$$
which implies
 \be\begin{split}
 B_N\circ A_N(U^k)
 &=\frac{1}{|Z_N|}\sum_{m,n\in Z_N,m-n=k}U^{m-n}
 e^{-\mathrm{i} n \tilde{\theta}(m-n)^t} e^{i n \tilde{\theta}(m-n)^t}\\
 &=\frac{\big|Z_N\cap (Z_N+k) \big|}{|Z_N|}\,U^k.
 \end{split}\ee
Then by the F${\o}$lner property of
$Z_N$, we deduce that  $\lim_N B_N\circ A_N (U^k)=U^k$ in
$L_p(\mathbb{T}^d_{\theta}).$ So  the claim is proved.

Now assume that the Schur multiplier $M_{\tilde{\phi}}$ is $\mathrm{c.b.}$ on
$S_p(\ell_2(\mathbb{Z}^d))$. We want to prove that $T_{\phi}$ is
$\mathrm{c.b.}$ on $L_p(\mathbb{T}^d_{\theta}).$  For any $x\in L_p(\B(\ell_2)\overline{\otimes}\mathbb{T}^d_{\theta}),$
 $$\|{\rm id}\ot T_{\phi}(x)\|_{L_p(\B(\ell_2)\overline{\otimes}\mathbb{T}^d_{\theta})}
 =\lim_{N}\|\big({\rm id}\ot B_N\big)\circ\big({\rm id}\ot A_N\big)
 \big({\rm id}\ot T_{\phi}(x)\big)\|_{L_p(\B(\ell_2)\overline{\otimes}\mathbb{T}^d_{\theta})}.$$
Using \eqref{F-S}, we see that  ${\rm id}\ot A_N({\rm id}\ot T_{\phi}(x))={\rm id}\ot M_{\tilde{\phi}}({\rm id}\ot A_N(x)).$ Thus
 \be\begin{split}
 \|{\rm id}\ot T_\phi(x)\|_{L_p(\B(\ell_2)\overline{\otimes}\mathbb T^d_\theta)}
 &\le \limsup_N\|{\rm id}\ot M_{\tilde\phi}({\rm id}\ot A_N(x))\|_{L_p(\B(\ell_2)\overline{\otimes}\B(\ell_2^{|Z_N|}))}\\
 &\leq \limsup_N \|M_{\tilde\phi}\|_{\mathrm{cb}}\|{\rm id}\ot A_N(x)\|_{L_p(\B(\ell_2)\overline{\otimes}\B(\ell_2^{|Z_N|}))}\\
 & \leq \|M_{\tilde\phi}\|_{\mathrm{cb}}\|x\|_{L_p(\B(\ell_2)\overline{\otimes}\mathbb{T}^d_\theta)}.
 \end{split}\ee
This implies that $T_{\phi}$ is c.b. on $L_p(\mathbb{T}^d_{\theta})$ and $\|T_\phi\|_{\rm cb}\le \|M_{\tilde\phi}\|_{\mathrm{cb}}$, as desired.

In summary, we have proved that
 $$T_{\phi}\;\text{is c.b. on}\;\;L_p(\mathbb{T}^d_{\theta}) \;\Longleftrightarrow\;
 M_{\tilde{\phi}}\;\text{is c.b. on}\;\; S_p(\ell_2(\mathbb{Z}^d)).$$
Applying this result to the case of $\theta=0$, we get that
 $$T_{\phi}\; \text{is c.b. on}\;\; L_p(\mathbb{T}^d) \;\Longleftrightarrow\;
 M_{\tilde{\phi}}\;\text{is  c.b. on}\;\; S_p(\ell_2(\mathbb{Z}^d)).$$
Therefore,
 $$\mathrm{M}_{\mathrm{cb}} (L_p (\mathbb{T}^d_{\theta}))
 =\mathrm{M}_{\mathrm{cb}} (L_{p} (\mathbb{T}^d))\quad \text{with equal norms}.$$
Thus the theorem is proved.
 \end{proof}

\begin{rk}
 The preceding proof shows that $\phi$ is a c.b. Fourier multiplier on $L_p(\T^d_\theta)$ iff $\tilde\phi$ is a c.b. Schur multiplier on $S_p(\el_2(\ent^d))$. This is the extension of Neuwirth and Ricard's transference result to twisted Toeplitz matrices. We will pursue this subject elsewhere for more general groups.
 \end{rk}

\begin{rk}
 It would be interesting to study thin sets on $\T^d_\theta$, for instance, $\Lambda(p)$-sets and Sidon sets. At the level of complete boundedness, Theorem~\ref{th:cbLpMultiplier}  shows that the $\Lambda(p)_{\rm cb}$-sets on $\T^d_\theta$ are exactly those on $\T^d$. We refer to Harcharras' thesis \cite{Har1999} for related results.
 \end{rk}

Theorem~\ref{th:cbLpMultiplier} suggests the following problem:

\begin{problem}
 Let $2<p\le\8$. Does one have
 $$\mathrm{M}(L_p(\mathbb{T}_{\theta}^d))=\mathrm{M}(L_p(\mathbb{T}^d))\,?$$
 \end{problem}

We conjecture that the answer would be negative. Indeed, it is negative in the case of $p=\8$ if one allows the number of generators to be infinite, as shown by the following remark that is communicated to us by Eric Ricard.

\begin{rk}
 Let $\theta=(\theta_{kj})$ be the infinite skew matrix such that $\theta_{kj}=1/2$ for all $k<j$. Let $\T^\8_\theta$ be the associated quantum torus. Now the generators of $\T^\8_\theta$ is a sequence $U=(U_1, U_2, \cdots)$ of anticommuting unitary operators:
  $$U_kU_j=-U_jU_k,\quad\forall\; k\neq j.$$
Let $\phi$ be the indicator function of the subset $\Lambda=\{e_k\;:\;k\ge1\}$ of $\ent^\8$, where $e_k$ is the element of  $\ent^\8$ whose coordinates all vanish except the one on the k-th position which is equal to $1$. Then $\phi\in\mathrm{M}(L_\8(\mathbb{T}_{\theta}^\8))$ but $\phi\not\in\mathrm{M}(L_\8(\mathbb{T}^\8))$.
 \end{rk}

Let us check this remark. Let $\a=(\a_k)\subset\com$ be a finite sequence and set
 $$x=\sum_k\a_k U_k.$$
Then by the anticommuting relation we have
 $$x^*x+xx^*=2\sum_k|\a_k|^2+\sum_{j\neq k}\bar\a_j\a_k(U_j^*U_k+U_kU_j^*)=2\sum_k|\a_k|^2.$$
It then follows that
 $$\|x\|_\8\le\sqrt2\,\|\a\|_2.$$
On the other hand, it is clear that
 $$\|x\|_\8\ge\|x\|_2\ge \|\a\|_2.$$
We then deduce that for any $\a=(\a_k)\subset\com$ the series $\sum_k\a_kU_k$ converges in $\T^\8_\theta$  iff $\a\in\el_2$. In this case, we have
 $$\|\a\|_2\le\big\|\sum_k\a_kU_k\big\|_\8\le \sqrt2\,\|\a\|_2.$$
This clearly implies that  $\phi$ is a bounded $L_\8$ multiplier on $\mathbb{T}_{\theta}^\8$. However, $\phi$ is not a bounded $L_\8$ multiplier on $\mathbb{T}^\8$. Otherwise, the closed subspace of $L_\8(\T^\8)$ generated by the generators $(z_1, z_2, \cdots)$ would be complemented in $L_\8(\T^\8)$. But this subspace is isometric to $\el_1$. It is well known that $\el_1$ cannot be isomorphic to a complemented subspace of an $L_\8$-space. This contradiction yields that $\phi\not\in\mathrm{M}(L_\8(\mathbb{T}^\8))$. This example also shows that
 $$\mathrm{M}_{\rm cb}(L_\8(\mathbb{T}_{\theta}^\8))\subsetneqq\mathrm{M}(L_\8(\mathbb{T}_{\theta}^\8)),$$
in contrast with equality \eqref{cb=b} in the commutative case.

\smallskip

We end this section by showing the equality $\mathrm{M}_{\mathrm{cb}} (L_p (\mathbb{T}^d_{\theta}))
 =\mathrm{M}_{\mathrm{cb}} (L_{p} (\mathbb{T}^d))$ in Theorem~\ref{th:cbLpMultiplier} holds completely isometrically. To this end we first need to equip these spaces with an operator space structure. Recall that for two operator spaces $E$ and $F$ the space $\mathrm{CB} (E,
F)$ has a natural operator space structure by setting $\mathbb M_n(\mathrm{CB} (E,
F))=\mathrm{CB} (E,\mathbb M_n(F))$. Then $\mathrm{M}_{\mathrm{cb}} (L_p (\mathbb{T}^d_{\theta}))$ inherits the operator space structure of $\mathrm{CB} (L_p (\mathbb{T}^d_{\theta}),\, L_p (\mathbb{T}^d_{\theta}))$. Let $\mathrm{TM}_{\mathrm{cb}}(S_p(\el_2(\ent^d)))$ be the subspace of all c.b. Schur multipliers $\psi$ on $S_p(\el_2(\ent^d))$ which are of the Toeplitz form, i.e., $\psi_{mn}=\phi_{m-n}$ for some $\phi$. $\mathrm{TM}_{\mathrm{cb}}(S_p(\el_2(\ent^d)))$ is also an operator space via $\mathrm{TM}_{\mathrm{cb}}(S_p(\el_2(\ent^d)))\subset  \mathrm{CB} (S_p(\el_2(\ent^d)), S_p(\el_2(\ent^d)))$.

\begin{prop}
  Let $1\le p \le\8$. Then
  $$\mathrm{M}_{\mathrm{cb}} (L_p (\mathbb{T}^d_{\theta}))
  =\mathrm{M}_{\mathrm{cb}} (L_{p} (\mathbb{T}^d))\cong \mathrm{TM}_{\mathrm{cb}}(S_p(\el_2(\ent^d)))$$
completely isometrically, where the last identification is realized by $\phi\in\mathrm{M}_{\mathrm{cb}} (L_{p} (\mathbb{T}^d))\leftrightarrow\tilde\phi\in \mathrm{TM}_{\mathrm{cb}}(S_p(\el_2(\ent^d)))$ with $\tilde\phi_{mn}=\phi_{m-n}$.
\end{prop}

\begin{proof}
 We require the following elementary fact: Let $\M$ be a von Neumann algebra and $u$ a unitary operator in $\mathbb M_n(\M)$. Then for any $x\in \mathbb M_n(L_p(\M))$
  $$\|uxu^*\|_{\mathbb M_n(L_p(\M))}=\|x\|_{\mathbb M_n(L_p(\M))}.$$
 Indeed, this is obvious for $p=\8$. Then by duality, it is also true for $p=1$. Finally, by interpolation, we deduce this equality for any $1<p<\8$. Armed with this fact, we can modify the proof of Theorem~\ref{th:cbLpMultiplier} to get the announced assertion. The details are left to the reader.
 \end{proof}


\section{Hardy spaces}\label{H1BMOLPT}


There exist several ways to define Hardy spaces on quantum tori. The resulting spaces may be different. The approach that we adopt in this section is based on the Littlewood-Paley theory and real variable method in Fourier analysis. Our Hardy spaces are defined by square functions in terms of the circular Poisson semigroup $\mathbb P_r$. This allows us to use the recent developments of operator-valued harmonic analysis and noncommutative Littlewood-Paley-Stein theory.

For any  $x\in\T^d_{\theta}$  define
 $$
 G_c(x) =  \Big ( \int_0^1 \Big | \frac{d}{dr}\mathbb{P}_r[x] \Big |^2(1-r)dr \Big )^{1/2}.
 $$
For $1\leq p <\infty$ let
 $$\|x\|_{\mathrm{H}^c_p}=|\hat x(\mathbf{0})| +\|G_c(x)\|_{L_p(\mathbb{T}^d_{\theta})}.$$
This is a norm on $\T^d_{\theta}$ (cf. e.g. \cite{JLX2006}). We define the column Hardy space
$\mathrm{H}^c_p(\mathbb{T}^d_{\theta})$ as the completion of
$\T^d_{\theta}$ with respect to this norm.  The row Hardy space  $\mathrm{H}^r_p(\mathbb{T}^d_{\theta})$ is defined to be the space of all $x$ such that $x^*\in \mathrm{H}^c_p(\mathbb{T}^d_{\theta})$ equipped with the natural norm. The mixture Hardy spaces are defined as follows: If $1\leq p<2$,
 $$\mathrm{H}_p(\mathbb{T}^d_{\theta}) = \mathrm{H}_p^c(\mathbb{T}^d_{\theta})+ \mathrm{H}_p^r(\mathbb{T}^d_{\theta})$$
equipped with the sum norm
 $$\|x\|_{\mathrm{H}_p}=\inf\{\|a\|_{\mathrm{H}^c_p} + \|b\|_{\mathrm{H}^r_p}\;:\;
  x=a+b, a \in \mathrm{H}_p^c(\mathbb{T}^d_{\theta}), b \in \mathrm{H}_p^r(\mathbb{T}^d_{\theta})\},$$
and if $2\leq p<\infty,$
 $$\mathrm{H}_p(\mathbb{T}^d_{\theta}) = \mathrm{H}_p^c(\mathbb{T}^d_{\theta}) \cap\mathrm{H}_p^r(\mathbb{T}^d_{\theta})$$
equipped with the intersection norm
 $$\|x\|_{\mathrm{H}_p} =\max\big\{\|x\|_{\mathrm{H}^c_p},\; \|x\|_{\mathrm{H}^r_p}\big\}.$$
We will also study the BMO spaces over $\mathbb{T}_{\theta}^d$. Set
 $$
 \mathrm{BMO}^c(\mathbb{T}_{\theta}^d) =  \big\{x\in L_2(\mathbb{T}_{\theta}^d)\;:\; \sup_r
  \big\|\mathbb{P}_r \big[ |x-\mathbb{P}_r[x]|^2\big] \big\|_\8<\infty \big\}
 $$
equipped with the norm
 $$\|x\|_{\mathrm{BMO}^c}=\max\big\{|\hat x(\mathbf{0})|,\;\sup_r  \big\|\mathbb{P}_r \big[ |x-\mathbb{P}_r[x]|^2 \big] \big\|_\8^{1/2}\big\}.$$
 $\mathrm{BMO}^r(\mathbb{T}_{\theta}^d)$ is defined as the space of all $x$
such that $x^*\in \mathrm{BMO}^c(\mathbb{T}_{\theta}^d)$ with the norm
$\|x\|_{\mathrm{BMO}^r}=\|x^*\|_{\mathrm{BMO}^c}.$ The mixture
$\mathrm{BMO} (\mathbb{T}_{\theta}^d)$ is  the intersection of these two
spaces:
 $$
 \mathrm{BMO} (\mathbb{T}_{\theta}^d) = \mathrm{BMO}^c(\mathbb{T}_{\theta}^d) \cap
 \mathrm{BMO}^r(\mathbb{T}_{\theta}^d)$$
with the intersection norm.

The above definitions are motivated by Hardy spaces of noncommutative martingales (\cite{JX2003, PX1997}) and  of quantum Markov semigroups (\cite{JLX2006, JM2011, Mei2007}). The main results of this section are summarized in the following statement which shows that the Hardy spaces on $\T^d_\theta$ possess  the properties of the usual Hardy spaces, as  expected. 

\begin{thm}\label{H1-BMO}
  \begin{enumerate}[\rm i)]

\item Let $1<p<\8$. Then $\mathrm{H}_p(\T^d_{\theta})=L_p(\T^d_{\theta})$ with equivalent norms.
   
\item The dual space of $\mathrm{H}_1^c(\T^d_{\theta})$ is equal to $\mathrm{BMO}^c(\T^d_{\theta})$ with equivalent norms via the duality bracket
  $$\la x,\; y\ra=\tau(xy^*),\quad x\in L_2(\T^d_\theta), \;y\in \mathrm{BMO}^c(\T^d_{\theta}).$$
The same assertion holds for the row and mixture spaces too.

\item  Let $1<p<\8$. Then
 $$(\mathrm{BMO}^c(\mathbb{T}_{\theta}^d),\; \mathrm{H}^c_1(\mathbb{T}^d_{\theta}))_{1/p}
 =\mathrm{H}^c_p(\mathbb{T}^d_{\theta})
 =(\mathrm{BMO}^c(\mathbb{T}_{\theta}^d),\; \mathrm{H}^c_1(\mathbb{T}^d_{\theta}))_{1/p, p}$$
with equivalent norms, where $(\,\cdot\,,\; \cdot\,)_{1/p}$ and $(\,\cdot\,,\; \cdot\,)_{1/p, p}$  denote respectively the complex and real interpolation functors.

\item  Let $1<p<\8$ and $X_0\in\{\mathrm{BMO}(\T^d_{\theta}),\; L_\8(\T^d_\theta)\}$, $X_1\in\{\mathrm{H}_1(\T^d_{\theta}),\; L_1(\T^d_\theta)\}$. Then
  $$(X_0,\; X_1)_{1/p}=L_p(\T^d_\theta)=(X_0,\; X_1)_{1/p, p}$$
 with equivalent norms. 
 \end{enumerate}
  \end{thm}

Some parts of this theorem can be deduced from existing results in literature. This is the case of i) and the complex interpolation equality $(\mathrm{BMO}(\mathbb{T}_{\theta}^d),\; L_1(\mathbb{T}^d_{\theta}))_{1/p}
 =L_p(\mathbb{T}^d_{\theta})$ in  iv). Let us explain these two points. 

According to the discussion following Theorem \ref{th:FejerPoissonPhiMaxIneq}, the circular Poisson semigroup $(\mathbb P_r)_{0\le r<1}$ on $\T^d_\theta$  is a noncommutative symmetric diffusion semigroup in the sense of \cite{JX2007}. We claim that  $(\mathbb P_r)_{0\le r<1}$ admits a Markov  dilation  (as well as a Rota dilation) in the sense of \cite{JLX2006}. Indeed, considering  the von Neumann subalgebra $\widetilde{\mathbb{T}_{\theta}^d}$ of $\N_{\theta}=L_\8(\T^d)\overline{\ot}\T^d_\theta$, which is the image of $\mathbb{T}_{\theta}^d$ under the map $x\mapsto \tilde{x}$, we see that  the circular Poisson semigroup on the usual torus $\T^d$  extends  to a semigroup by tensoring with ${\rm id}_{\T^d_\theta}$. By a slight abuse of notation, we will also use $(\mathbb P_r)_{0\le r<1}$  to denote the circular Poisson semigroup on the usual torus $\T^d$. It is clear that $\mathbb{P}_r\ot{\rm id}_{\T^d_\theta}[\tilde x]=\wt{\mathbb P_r[x]}$ for any $x\in\T^d_\theta$. Since every symmetric diffusion semigroup on a commutative von Neumann algebra can be dilated to a Markov unitary group as well as an inverse martingale, $(\mathbb{P}_r\ot{\rm id}_{\T^d_\theta})_{0\le r<1}$ admits a Markov/Rota dilation, so does its restriction to $\widetilde{\mathbb{T}_{\theta}^d}$. Our claim then follows.  Therefore, the semigroup $(\mathbb P_r)_{0\le r<1}$ on $\T^d_\theta$  satisfies the assumption of \cite{JLX2006} which insures the existence of an associated $H_\8$-functional calculus. Thus by \cite[Theorem~7.6]{JLX2006}, we get i).  On the other hand, the interpolation theorem of \cite{JM2011} yields $(\mathrm{BMO}(\mathbb{T}_{\theta}^d),\; L_1(\mathbb{T}^d_{\theta}))_{1/p}
 =L_p(\mathbb{T}^d_{\theta})$. We also point out that the duality result in part ii) could be deduced from a work in progress of Avsec and Mei \cite {AM2011}.

\medskip

To  prove the remaining parts  of Theorem~\ref{H1-BMO}  we will use transference to reduce the  problem to the corresponding one on $\N_{\theta}$ and then use Mei's results \cite{Mei2007}. An advantage of this proof  is that it also provides an alternative (more elementary) approach to the two parts already considered in the previous paragraph. Recall that the framework of \cite{Mei2007} is the Euclidean space $\real^d$, and the Hardy spaces there are defined by using the Poisson semigroup on $\real^d$. The geometry of $\real^d$ is simpler than $\T^d$. But what really renders matters more handy in $\real^d$ is the explicit compact formula of the Poisson kernel (or its growth estimates). The situation for $\T^d$ is harder. Although it is claimed in \cite{Mei2007} as remarks that all results there hold equally with essentially the same proofs in the $d$-torus setting, this claim is clearly true for $\T$ thanks to the explicit simple formula of the Poisson kernel of $\T$. However, it would not be so transparent whenever $d\ge2$. As a byproduct of our proof below of Theorem~\ref{H1-BMO}, we remedy this situation, which constitutes another advantage of our approach via transference. Finally,  it seems that even in the scalar case there does not exist published references  on Hardy space theory on $\T^d$ for $d\ge2$ via the Littlewood-Paley theory, although this theory is certainly known as folklore to many specialists. Our approach provides, in particular,  a complete  picture of the scalar-valued Hardy space theory on $\T^d$, exactly parallel to that on $\real^d$.

\medskip\n{\bf Convention.} For notational simplicity we will denote all circular Poisson semigroups considered in the sequel by $(\mathbb P_r)_{0\le r<1}$. Thus  $\mathbb{P}_r\ot{\rm id}_{\T^d_\theta}$ will be simply denoted by $\mathbb{P}_r$. This slight abuse of notation should not cause any confusion in concrete contexts. For instance, for $x\in\T^d_\theta$,  $\mathbb{P}_r[x]\in\T^d_\theta$ while for $f\in\N_\theta$, $\mathbb{P}_r[f]=\mathbb{P}_r\ot{\rm id}_{\T^d_\theta}[f]\in\N_\theta$.  On the other hand, $\mathbb{P}_r$ will also stands for the circular Poisson kernel on $\T^d$ given by \eqref{PoissonKer}. Thus for $f\in L_1(\N_\theta)$ we have (recalling that $dm$ denotes Haar measure on $\T^d$)
 $$\mathbb P_r[f](z)=\mathbb P_r*f(z)=\int_{\T^d}\mathbb P_r(z\cdot\overline w)f(w)dm(w),\quad z\in\T^d.$$

We will study several BMO norms as well as $\mathrm{H}_p^c$ norms. The notational system for these norms (or spaces) might look heavy; but everything should be clear in concrete contexts.   We start our analysis  with BMO spaces on $\T^d$ with values in a von Neumann algebra $\M$. For simplicity we will assume that $\M$ is equipped  with a normal faithful tracial state $\tau$ ($\M$ will be $\T^d_\theta$ in the proof of Theorem~\ref{H1-BMO}) .  We start with the BMO space. Let
 $$\mathrm{BMO}^c(\T^d;\M)=\{f\in L_2(\T^d;L_2(\M))\;: \;
 \sup_r\big\|\mathbb{P}_r\big[ |f-\mathbb{P}_r[f]|^2 \big]\big\|_{\infty}<\infty\},$$
equipped with the norm
 $$\|f\|_{\mathrm{BMO}^c}=\max\big\{\|\hat f(\mathbf 0)\|_\8,\;
 \sup_r\big\|\mathbb{P}_r\big[ |f-\mathbb{P}_r[f]|^2 \big]\big\|_{\infty}^{1/2}\big\}.$$
Here the first $L_\8$-norm is the one of $\M$ and the second that of $L_\8(\T^d)\overline\ot\M$.

We require the following lemma  which characterizes $\mathrm{BMO}^c(\T^d;\M)$ by the noncommutative analogue of the usual Garsia norm. This lemma is a special case of \cite[Theorem 2.9]{JM2011}. But we prefer to present the following elementary proof which was communicated to us by Tao Mei.

\begin{lem}\label{garsia lem}
 For any $f\in L_2(\T^d;L_2(\M))$ we have
 \beq\label{2bmo}
  \sup_r\big\|\mathbb{P}_r\big[ |f-\mathbb{P}_r[f]|^2 \big]\big\|_{\infty}\approx
  \sup_r\big\|\mathbb{P}_r[|f|^2]-|\mathbb{P}_r[f]|^2\big\|_{\infty}
  \eeq
 with universal equivalence constants.
  \end{lem}
  
\begin{proof}
 First note that 
 $$\mathbb{P}_r[|f|^2](z)-|\mathbb{P}_r[f]|^2 (z)
 =\mathbb{P}_r\big[ |f-\mathbb{P}_r[f](z)|^2 \big](z),\quad\forall\; z\in\T^d.$$
Thus
 \beq\label{garsia}
 \sup_{0\le r<1}\big\|\mathbb{P}_r[|f|^2]-|\mathbb{P}_r[f]|^2 \big\|_{\infty}
 =\sup_{0\le r<1}\sup_{z\in\T^d}\big\|\mathbb{P}_r\big[ |f-\mathbb{P}_r[f](z)|^2 \big](z)\big\|_{\M}.
 \eeq
The right hand side is exactly the analogue of the usual Garsia norm (cf. \cite[Corollary~VI.2.4]{Gar2007}). For any fixed $r$ and $z$ we have
 $$
 \big\|\mathbb{P}_r\big[ |f-\mathbb{P}_r[f]|^2 \big](z)\big\|_{\M}^{1/2}
 \le \big\|\mathbb{P}_r\big[ |f-\mathbb{P}_r[f](z)|^2 \big](z)\big\|_{\M}^{1/2}
 + \big\|\mathbb{P}_r\big[\big|\mathbb{P}_r[f-\mathbb{P}_r[f](z)]\big|^2 \big](z)\big\|_{\M}^{1/2}.$$
 By Kadison's Cauchy-Schwarz inequality, 
  $$\mathbb{P}_r\big[\big|\mathbb{P}_r[f-\mathbb{P}_r[f](z)]\big|^2 \big]
  \le \mathbb{P}_{r^2}\big[\big|f-\mathbb{P}_r[f](z)\big|^2\big].$$
On the other hand, since $\mathbb{P}_r$ is subordinated to the heat semigroup on $\T^d$, by the subordination formula, one has
 $\mathbb{P}_{r^2}[g]\le 2\mathbb{P}_{r}[g]$ for positive $g\in L_1(\T^d;L_1(\M))$. Alternatively, this inequality can be easily checked by \eqref{poisson} below. 
Then we deduce that
 $$\sup_{r}\sup_{z}\big\|\mathbb{P}_r\big[ |f-\mathbb{P}_r[f]|^2 \big](z)\big\|_{\M}^{1/2}
 \le(1+\sqrt2\,)\sup_{r}\sup_{z}\big\|\mathbb{P}_r\big[ |f-\mathbb{P}_r[f](z)|^2 \big](z)\big\|_{\M}^{1/2}.$$
This is the upper estimate of \eqref{2bmo}.

The converse inequality is harder. Fix $f\in L_2(\T^d;L_2(\M))$. By triangle inequality, we have
 \be\begin{split}
 \big\|\mathbb{P}_r[|f|^2]-|\mathbb{P}_r[f]|^2\big\|_{\infty}^{1/2}
 \le &\big\|\mathbb{P}_r[|f-\mathbb{P}_r[f]|^2]-
 |\mathbb{P}_r[f-\mathbb{P}_r[f]]|^2\big\|_{\infty}^{1/2}\\
 &+\big\|\mathbb{P}_r[|\mathbb{P}_r[f]|^2]-
 |\mathbb{P}_r[\mathbb{P}_r[f]]|^2\big\|_{\infty}^{1/2}\\
 \le &\big\|\mathbb{P}_r[|f-\mathbb{P}_r[f]|^2]\big\|_{\infty}^{1/2}\\
 &+\big\|\mathbb{P}_r[|\mathbb{P}_r[f]|^2]-
 |\mathbb{P}_r[\mathbb{P}_r[f]]|^2\big\|_{\infty}^{1/2}.
 \end{split}\ee
Assuming for the moment the following inequality
 \beq\label{conv}
 2\mathbb{P}_r[|\mathbb{P}_r[f]|^2]
 \le \mathbb{P}_{r^2}[|f|^2]+|\mathbb{P}_{r^2}[f]|^2,
 \eeq
we get
 $$2\big(\mathbb{P}_r[|\mathbb{P}_r[f]|^2]-
 |\mathbb{P}_{r^2}[f]|^2\big)
 \le \mathbb{P}_{r^2}[|f|^2]-|\mathbb{P}_{r^2}[f]|^2.$$
Combining the preceding inequalities, we then deduce that
 $$\sup_r\big\|\mathbb{P}_r[|f|^2]-|\mathbb{P}_r[f]|^2\big\|_{\infty}^{1/2}
 \le \sup_r\big\|\mathbb{P}_r\big[ |f-\mathbb{P}_r[f]|^2 \big]\big\|_{\infty}^{1/2} +
 \frac1{\sqrt2}\sup_r\big\|\mathbb{P}_r[|f|^2]-|\mathbb{P}_r[f]|^2\big\|_{\infty}^{1/2}.$$
Whence the lower  estimate of \eqref{2bmo} with $2+\sqrt2$ as constant.

It remains to prove \eqref{conv}. To this end, it is more convenient to work with $\mathbb Q_\e=\mathbb{P}_r$ for $r=e^{-2\pi\e}$.  Then we must show
 \beq\label{convbis}
 \mathbb Q_\e[|\mathbb Q_\e[f]|^2]-|\mathbb Q_{2\e}[f]|^2\le 
 \mathbb Q_{2\e}[|f|^2] -\mathbb Q_\e[|\mathbb Q_\e[f]|^2],\quad\forall\;\e>0.
 \eeq
Let us write
 $$\mathbb Q_\e[|\mathbb Q_\e[f]|^2]-|\mathbb Q_{2\e}[f]|^2
 =-\int_0^\e\frac{d}{dt}\mathbb Q_{\e-t}[|\mathbb Q_{\e+t}[f]|^2]dt.$$
Let $A$ be the negative generator of   $\mathbb Q_\e$: $\mathbb Q_\e=e^{-\e A}$. Then 
 \be\begin{split}
 \frac{d}{dt}\mathbb Q_{\e-t}[|\mathbb Q_{\e+t}[f]|^2]
 =&A\mathbb Q_{\e-t}[|\mathbb Q_{\e+t}[f]|^2] \\
 &-\mathbb Q_{\e-t}\big[(A\mathbb Q_{\e+t}[f]^*)(\mathbb Q_{\e+t}[f])
 +(\mathbb Q_{\e+t}[f]^*)(A\mathbb Q_{\e+t}[f])\big].
 \end{split}\ee
For $s>0$ let
 $$F_s(g)=-A\mathbb Q_{s}[|\mathbb Q_{s}[g]|^2] \\
 +\mathbb Q_{s}\big[(A\mathbb Q_{s}[g]^*)(\mathbb Q_{s}[g])
 +(\mathbb Q_{s}[g]^*)(A\mathbb Q_{s}[g])\big].$$
Then for $g=\mathbb Q_{\e+t}[f]$ we have
 \beq\label{int}
 \mathbb Q_\e[|\mathbb Q_\e[f]|^2]-|\mathbb Q_{2\e}[f]|^2
 =\lim_{s\to0}\int_0^\e\mathbb Q_{\e-t}[F_s(g)]dt.
 \eeq
It is easy to check that $\lim_{s\to\8}F_s(g)=0$ (one can use, for instance, \eqref{poisson} below). Then
 \beq\label{Fg-h}
 F_s(g)=-\int_s^\8\frac{d}{du}F_u(g)du.
 \eeq
Elementary calculations lead to
 \be\begin{split}
  \frac{d}{du}F_u(g)
  &=A^2\mathbb Q_{u}[|\mathbb Q_{u}[g]|^2] 
 -\mathbb Q_{u}\big[(A^2\mathbb Q_{u}[g]^*)(\mathbb Q_{u}[g])
 +(\mathbb Q_{u}[g]^*)(A^2\mathbb Q_{u}[g])\big]-2\mathbb Q_{u}[|A\mathbb Q_{u}[g]|^2]\\
 &=\mathbb Q_{u}\big[A^2|\mathbb Q_{u}[g]|^2 
 -(A^2\mathbb Q_{u}[g]^*)(\mathbb Q_{u}[g])
 -(\mathbb Q_{u}[g]^*)(A^2\mathbb Q_{u}[g])-2|A\mathbb Q_{u}[g]|^2\big].
 \end{split}\ee
Note that
 $$A=2\pi\sqrt{-\D}\,,$$
where $\D$ is the Laplacian of $\T^d$:
 $$\D=\sum_{k=1}^d\frac{\partial^2}{\partial z_k^2}\,.$$
So $A^2=-4\pi^2\D$ and
 \be
 A^2|\mathbb Q_{u}[g]|^2=(A^2\mathbb Q_{u}[g]^*)(\mathbb Q_{u}[g])
 +(\mathbb Q_{u}[g]^*)(A^2\mathbb Q_{u}[g])-8\pi^2\sum_{k=1}^d\Big|\frac{\partial}{\partial z_k}\mathbb Q_{u}[g]\Big|^2.
 \ee
Therefore,
 \be 
 \frac{d}{du}F_u(g)=
 -8\pi^2\sum_{k=1}^d\mathbb Q_{u}\big[\big|\frac{\partial}{\partial z_k}\mathbb Q_{u}[g]\big|^2\big]
 -2\mathbb Q_{u}\big[|A\mathbb Q_{u}[g]|^2\big].
 \ee 
Recall that $g=\mathbb Q_{\e+t}[f]$. By Kadison's Cauchy-Schwarz inequality and using the above equality twice, we obtain
 $$
 -\frac{d}{d u}F_u(g)\le 
 \mathbb Q_{\e}\Big[8\pi^2\sum_{k=1}^d\mathbb Q_{u}\Big[\Big|\frac{\partial}{\partial z_k}\mathbb Q_{u}[h]\Big|^2\Big]
 +2\mathbb Q_{u}\big[|A\mathbb Q_{u}[h]|^2\big]\Big]\le -\mathbb Q_{\e}\Big[\frac{d}{d u}F_u(h)\Big],$$
where $h=\mathbb Q_{t}[f]$. Thus by \eqref{Fg-h},
 $$F_s(g)\le \mathbb Q_{\e}[F_s(h)].$$
Hence by \eqref{int} and inverting the procedure leading to \eqref{int}, we obtain
 \be\begin{split}
 \mathbb Q_\e[|\mathbb Q_\e[f]|^2]-|\mathbb Q_{2\e}[f]|^2
 &\le\lim_{s\to0}\int_0^\e\mathbb Q_{2\e-t}[F_s(h)]dt\\
 &=-\int_0^\e\frac{\partial}{\partial t}\mathbb Q_{2\e-t}[|\mathbb Q_{t}[f]|^2]dt
 =\mathbb Q_{2\e}[|f|^2]-\mathbb Q_\e[\mathbb Q_\e[f]|^2].
 \end{split}\ee
This yields \eqref{convbis}, and \eqref{conv} too. Thus the lemma is proved.
 \end{proof}

Although this is not really necessary, it is more convenient to work with the cube $\mathbb{I}^d=[0,\, 1]^d$ instead of $\T^d$. Another reason is that the case of $\I^d$ is closer to that of $\real^d$.  So we will identify $\T^d$ with $\I^d$, as in the proof of Theorem \ref{th:FejerPoissonPhiMaxIneq}. The addition in $\mathbb{I}^d$ is modulo $1$ coordinatewise, which corresponds to the multiplication in $\T^d$ under the identification $(e^{2 \pi \mathrm{i} s_1},\cdots , e^{2 \pi \mathrm{i} s_d})\;\leftrightarrow\; (s_1, \cdots, s_d)$.  Accordingly, functions on $\T^d$ and $\mathbb{I}^d$ are identified too. Thus  $L_p(\T^d;L_p(\M))= L_p(\mathbb{I}^d; L_p(\M))$. 

We will use the following Poisson summation formula (see \cite[Corollary~VII.2.6]{SW1975}):
 \beq\label{poisson}
 \mathbb P_r(z)=\sum_{m\in\ent^d}\f_\e(s+m)\quad \text{with}\quad z=(e^{2\pi{\rm i}s_1},\cdots, e^{2\pi{\rm i}s_d})\quad\text{and}\quad r=e^{-2\pi\e},
 \eeq
 where $\f_\e$ is the Poisson kernel on $\real^d$:
 $$\f_\e(s)=c_d\,\frac{\e}{(\e^2+|s|^2)^{(d+1)/2}}\,,\quad s=(s_1,\cdots, s_d)\in\real^d.$$
In the sequel, we will always assume that $z$ and $s$, $r$ and $\e$ are related as in \eqref{poisson}. Let
 \beq\label{kernel Q}
 \mathbb Q_\e(s)=\sum_{m\in\ent^d}\f_\e(s+m),\quad s\in\mathbb{I}^d.
 \eeq
This notation is consistent with that introduced during the proof of Lemma~\ref{garsia lem} since 
   \beq\label{2poissons}
   \mathbb P_r[f](z)=\mathbb Q_\e[f](s)=\mathbb Q_\e*f(s)=\int_{\mathbb{I}^d }\mathbb Q_\e(s-t)f(t)dt.
   \eeq
An interval of $\mathbb{I}$ is either a subinterval of $\mathbb{I}$ or a union  $[b,\, 1]\cup[0,\, a]$ with $0<a<b<1$. The latter union is the interval $[b-1,\, a]$ by the addition modulo $1$ of $\I$. So the intervals of $\mathbb{I}$ correspond exactly to the arcs of $\T$. A cube of $\mathbb{I}^d$ is a product of $d$ intervals.  For $f\in L_1(\mathbb{I}^d; L_1(\M))$ and a cube $Q\subset \mathbb{I}^d$ let
 $$f_Q=\frac1{|Q|}\int_Q fds,$$
 where $|Q|$ denotes the volume of $Q$. Then we define $\mathrm{BMO}^c(\mathbb{I}^d; \M)$ as the space of all $f\in L_2(\mathbb{I}^d; L_2(\M))$ such that
 $$\sup_{Q\subset \mathbb{I}^d  \text{cube}} \Big\|\frac1{|Q|}\int_Q\big|f-f_Q\big|^2ds\Big\|_\8<\8,$$
equipped with the norm
 $$\|f\|_{\mathrm{BMO}^c}=\max\big\{\big\|f_{\mathbb{I}^d}\big\|_\8,\quad
 \sup_{Q\subset \mathbb{I}^d  \text{cube}} \Big\|\frac1{|Q|}\int_Q\big|f-f_Q\big|^2ds\Big\|_\8^{1/2}\big\}.$$
Here $\|\,\|_\8$ denotes, of course, the norm of $\M$.

\begin{lem}\label{BMO cube}
 $\mathrm{BMO}^c(\T^d; \M)=\mathrm{BMO}^c(\mathbb{I}^d; \M)$ with equivalent norms.
   \end{lem}

\begin{proof}
 Fix $f\in L_2(\T^d; L_2(\M))$. Without loss of generality, assume that $\hat f(\mathbf 0)=f_{\mathbb{I}^d}=0$. By Lemma~\ref{garsia lem} and \eqref{2poissons}, we need to show 
 \beq\label{classical bmo}
 \sup_{\e>0}\sup_{s\in\I^d}\big\|\mathbb{Q}_\e\big[ |f-\mathbb{Q}_\e[f](s)|^2 \big](s)\big\|_{\8}
 \approx  \sup_{Q\subset \mathbb{I}^d  \text{cube}} \Big\|\frac1{|Q|}\int_Q\big|f-f_Q\big|^2dt\Big\|_\8.
 \eeq
Let $Q$ be a cube of $\I^d$. Let $s$ and $\e$ be the center and half of the side length of $Q$, respectively. It is clear that
 $$\frac1{|Q|}\,\un_{Q}(t)\le C_d\,\f_\e(s-t)\le C_d\,\Q_\e(s-t).$$
Thus
$$\frac1{|Q|}\int_Q\big|f(t)-\mathbb{Q}_\e[f](s)\big|^2dt\le C_d\, \mathbb{Q}_\e\big[ |f-\mathbb{Q}_\e[f](s)|^2 \big](s).$$
Then
 $$\frac1{|Q|}\int_Q\big|f-f_Q\big|^2dt
 \le 4\frac1{|Q|}\int_Q\big|f-\mathbb{Q}_\e[f](s)\big|^2dt
 \le 4 C_d\, \mathbb{Q}_\e\big[ |f-\mathbb{Q}_\e[f](s)|^2 \big](s).$$
This yields one inequality of \eqref{classical bmo}. 

To show the converse inequality  fix $s\in\I^d$ and $\e>0$. Consider first the case $\e\ge1/2$.  
 Then $\mathbb{Q}_\e(t)\approx 1$ for any $t\in\I^d.$
It follows that 
 $$\mathbb{Q}_\e\big[|f-\mathbb{Q}_\e[f](s)|^2 \big](s)
 \approx \int_{\I^d}|f-\mathbb{Q}_\e[f](s)|^2\lesssim \int_{\I^d}|f|^2.$$
Whence
 $$\big\|\mathbb{Q}_\e\big[ |f-\mathbb{Q}_\e[f](s)|^2 \big](s)\big\|_{\8}
 \lesssim \Big\|\int_{\I^d}|f|^2\Big\|_\8\le\|f\|^2_{\mathrm{BMO}^c(\mathbb{I}^d; \M)}\,.$$
Now assume $\e<1/2$. By the proof of Theorem \ref{th:FejerPoissonPhiMaxIneq},  for any $t\in\I^d$
 $$\sum_{m\neq\mathbf{0}}\f_\e(t+m)\lesssim \e\lesssim\f_\e(t)\,.$$
Consequently, 
 $$\mathbb{Q}_\e\big[ |f-\mathbb{Q}_\e[f](s)|^2 \big](s)\lesssim \int_{\I^d}\f_\e(s-t)|f(t)-\mathbb{Q}_\e[f](s)|^2dt.$$
Let $Q=\{t\in\I^d\;:\; |t-s|\le\e\}$ and  $Q_k=\{t\in\I^d\;:\; |t-s|\le 2^{k+1}\e\}$.
Then
 \be\begin{split}
 \int_{\I^d}\f_\e(s-t)|f(t)-f_Q|^2dt
 &=\int_{Q}\f_\e(s-t)|f(t)-f_Q|^2dt\\
 &\hskip .2cm+\sum_{k\ge0}\int_{2^k\e<|t-s|\le 2^{k+1}\e}\f_\e(s-t)|f(t)-f_Q|^2dt\\
 & \lesssim  \frac1{|Q|}\int_{Q}|f(t)-f_Q|^2+ \sum_{k\ge0} \frac1{2^{k}|Q_k|}\int_{Q_k}|f(t)-f_Q|^2dt.
 \end{split}\ee
The above sums on $k$ are in fact finite sums. By triangle inequality (with $Q_{-1}=Q$), 
 $$\Big\|\frac1{|Q_k|}\int_{Q_k}|f-f_Q|^2\Big\|_\8^{1/2}\le 
 \Big\|\frac1{|Q_k|}\int_{Q_k}|f-f_{Q_k}|^2\Big\|_\8^{1/2}+\sum_{j=0}^k\|f_{Q_j}-f_{Q_{j-1}}\|_\8.$$
However,
 \be\begin{split}
 \|f_{Q_j}-f_{Q_{j-1}}\|_\8^2
 &\le \Big\|\frac1{|Q_{j-1}|}\int_{Q_{j-1}}|f-f_{Q_j}|^2\Big\|_\8\\
 &\le 2^d\Big\|\frac{1}{|Q_{j}|}\int_{Q_{j}}|f-f_{Q_j}|^2\Big\|_\8
 \le 2^d \|f\|^2_{\mathrm{BMO}^c(\mathbb{I}^d; \M)}.
 \end{split}\ee
Combining the preceding inequalities, we obtain
 $$\big\|\mathbb{Q}_\e\big[ |f-f_Q|^2 \big](s)\big\|_\8
 \lesssim \sum_{k\ge0}\frac{k+1}{2^k}\,\|f\|^2_{\mathrm{BMO}^c(\mathbb{I}^d; \M)}
 \lesssim\|f\|^2_{\mathrm{BMO}^c(\mathbb{I}^d; \M)}.$$ 
Finally, 
 $$\big\|\mathbb{Q}_\e\big[ |f-\mathbb{Q}_\e[f](s)|^2 \big](s)\big\|_\8^{1/2}
 \le 2\big\|\mathbb{Q}_\e\big[ |f-f_Q|^2 \big](s)\big\|_\8^{1/2}
 \lesssim\|f\|_{\mathrm{BMO}^c(\mathbb{I}^d; \M)}.$$
This implies the missing inequality of \eqref{classical bmo}. 
 \end{proof}

\begin{rk}
 The previous proof shows implicitly that the supremum on $\e$ in \eqref{classical bmo} can be restricted to $0<\e<1$. In fact, only small values of $\e$ are important for this supremum. Accordingly, only values of $r$ close to $1$ matter in the two suprema in \eqref{2bmo}. This property can be also verified by the argument in the proof of Lemma~\ref{Hardy cube} below.
 \end{rk}

Functions on $\T^d$ are $1$-periodic functions on $\real^d$, or equivalently, functions on $\I^d$ can be extended to  $1$-periodic functions to $\real^d$. For a function $f$ on $\T^d$ (or $\I^d$) $\tilde f$ will denote the corresponding $1$-periodic function on $\real^d$. Then \eqref{kernel Q} implies that $\Q_\e[f]$ is equal to the Poisson integral of $\tilde f$ on $\real^d$ that will be denoted by $\f_\e[\tilde f]$. Let us record this useful fact here for later reference:
 \beq\label{periodic}
 \Q_\e[f]=\f_\e[\tilde f]=\f_\e*\tilde f\quad\text{on}\quad\I^d.
 \eeq
Recall that $\mathrm{BMO}^c(\real^d; \M)$ is defined as the space of all locally square integrable functions $\psi$ from $\real^d$ to $L_2(\M)$ such that 
 $$\|\psi\|_{\mathrm{BMO}^c}=\max\big\{\big\|\psi_{\mathbb{I}^d}\big\|_\8,\quad
 \sup_{Q\subset \real^d  \text{cube}} \Big\|\frac1{|Q|}\int_Q\big|\psi-\psi_Q\big|^2ds\Big\|_\8^{1/2}\big\}.$$
The following lemma shows that the map $f\mapsto\tilde f$ establishes an isomorphic embedding of $\mathrm{BMO}^c(\T^d; \M)$ into $\mathrm{BMO}^c(\real^d; \M)$.

\begin{lem}\label{periodic BMO}
For any $f\in \mathrm{BMO}^c(\T^d; \M)$ we have
 $$\|f\|_{\mathrm{BMO}^c(\T^d; \M)}\approx \|\tilde f\|_{\mathrm{BMO}^c(\real^d; \M)}$$
with equivalence constants depending only on $d$.
 \end{lem}
 
\begin{proof} 
 Let $f\in L_2(\T^d; L_2(\M))$ with $f_{\I^d}=0$. By \eqref{classical bmo} and \eqref{periodic}, we have
 $$\|f\|^2_{\mathrm{BMO}^c(\T^d; \M)}\approx 
 \sup_{\e>0}\sup_{s\in\real^d}\big\|\f_\e\big[|\tilde f-\f_\e[\tilde f](s)|^2 \big](s)\big\|_{\8}.$$
Then the proof of Lemma~\ref{BMO cube} shows that the right hand side above is equivalent to $\|\tilde f\|^2_{\mathrm{BMO}^c(\real^d; \M)}$. Alternately, one can directly prove that the supremum on the right hand side in \eqref{classical bmo} is equivalent to  $\|\tilde f\|^2_{\mathrm{BMO}^c(\real^d; \M)}$. Namely,
 $$ \sup_{Q\subset \mathbb{I}^d  \text{cube}} \Big\|\frac1{|Q|}\int_Q\big|\tilde f-\tilde f_Q\big|^2ds\Big\|_\8
 \approx  \sup_{Q\subset \real^d  \text{cube}} \Big\|\frac1{|Q|}\int_Q\big|\tilde f-\tilde f_Q\big|^2ds\Big\|_\8.$$
Indeed, let $Q$ be a cube in $\real^d$. If $|Q|\le 1$, then by the definition of cubes in $\I^d$ and the periodicity of $\tilde f$, $Q$ can be considered as a cube in $\I^d$. So assume $|Q|>1$. Take another cube $R$ such that $Q\subset R$, $|R|\le 2^d|Q|$ and the side length of $R$ is an integer $k$. Then $R$ is a union of $k^d$ cubes of side length $1$. Thus by the periodicity of $\tilde f$
 \be
 \frac1{|Q|}\int_Q\big|\tilde f-\tilde f_Q\big|^2ds
 \le 4\frac1{|Q|}\int_Q|\tilde f|^2ds
 \le\frac{2^{d+2}}{|R|}\int_R|\tilde f|^2ds=2^{d+2}\int_{\I^d}|\tilde f|^2ds.
 \ee
Therefore, we get the desired equivalence.
 \end{proof}

Now we turn to  the discussion of Hardy spaces. Let  $1\leq p <\infty$.  For $f\in L_\8(\T^d)\overline\ot\M$ define
 \beq\label{torus g}
 G_c(f) (z)=  \Big ( \int_0^1 \Big | \frac{d}{dr}\mathbb{P}_r[f](z)\Big |^2(1-r)dr \Big )^{1/2}\,,\quad z\in\T^d
 \eeq
 and
 $$\|f\|_{\mathrm{H}^c_p}=\|\hat f(\mathbf{0})\|_{p} +\|G_c(f)\|_{p}.$$
Here the first $L_p$-norm is the one of $L_p(\M)$ and the second that of  $L_p(\T^d; L_p(\M))$.  Completing $L_\8(\T^d)\overline\ot\M$ under the norm $\|\,\|_{\mathrm{H}^c_p}$, we get $\mathrm{H}^c_p(\T^d;\M)$.
Like in the $\mathrm{BMO}$ case, we wish to reduce these Hardy spaces to those on $\I^d$. Using the kernel $\Q_\e$ in \eqref{kernel Q},  for $f\in L_\8(\I^d)\overline\ot \M$ let
 \beq\label{tg}
 \wt G_c(f) (s)=  \Big ( \int_0^\8 \Big | \frac{d}{d\e}\mathbb{Q}_\e[f](s) \Big |^2\e d\e \Big )^{1/2}\,,\quad s\in\I^d.
 \eeq
Let $\tilde f$ be the periodic extension of $f$ to $\real^d$. Let $\wt G_c(\tilde f)$ be  the $g$-function  of $\tilde f$ defined by the Poisson kernel $\f_\e$:
 \beq\label{tg periodic}
 \wt G_c(\tilde f) (s)=  \Big ( \int_0^\8 \Big | \frac{d}{d\e}\f_\e[\tilde f](s) \Big |^2\e d\e \Big )^{1/2}\,,\quad s\in\real^d.
 \eeq
Thanks to \eqref{periodic}, we have
 \beq\label{tg=tg periodic}
 \wt G_c(f) =\wt G_c(\tilde f) \quad\text{on }\;\I^d.
 \eeq
Thus $\wt G_c(\tilde f) $ is  the periodic extension to $\real^d$ of $\wt G_c(f) $. Let
 $$\|f\|_{\mathrm{H}^c_p}=\|f_{\I^d}\|_{p} +\|\wt G_c(f)\|_{p}.$$
Here the first $L_p$-norm is the one of $L_p(\M)$ and the second that of  $L_p(\I^d; L_p(\M))$. Define $\mathrm{H}^c_p(\I^d; \M)$ to be the  completion of $(L_\8(\I^d)\overline\ot \M,\|\,\|_{\mathrm{H}^c_p})$.

\begin{lem}\label{Hardy cube}
 Let $1\le p<\8$. Then $\mathrm{H}_p^c(\T^d;\M)=\mathrm{H}^c_p(\I^d; \M)$ with equivalent norms.   
  \end{lem}

\begin{proof}  
 We first show that in the definition of the Littlewood-Paley function $G_c(f)$ in \eqref{torus g} only values of $r$ close to $1$ matter. More precisely, for any $0<r_0<1$ setting
   $$
 G_{c,r_0}(f) (z)=  \Big ( \int_{r_0}^1 \Big | \frac{d}{dr}\mathbb{P}_r[f](z)\Big |^2(1-r)dr \Big )^{1/2},
 $$
we have
  $$\|G_c(f)\|_{p}\approx \|G_{c,r_0}(f)\|_{p},$$
 where the equivalence constants depend only on $d$ and $r_0$. To this end take any $0\le r_0<r_1<1$ and let
 $$
 G'_{c}(f) (z)=  \Big ( \int_{r_0}^{r_1} \Big | \frac{d}{dr}\mathbb{P}_r[f](z)\Big |^2(1-r)dr \Big )^{1/2}.
 $$
 We claim that
 \beq\label{small values}
 \|G'_c(f)\|_{p}\approx \sup_{n\in\ent^d,n\neq\mathbf{0}}\|\hat f(n)\|_{p}.
 \eeq
Writing the Fourier series expansion of $\mathbb{P}_r[f]$:
 $$\mathbb{P}_r[f](z)=\sum_{n\in\ent^d}r^{|n|_2}\hat f(n)z^n,$$
 we have
 $$ \frac{d}{dr}\mathbb{P}_r[f](z)=\sum_{n\in\ent^d, n\neq\mathbf{0}}|n|_2r^{|n|_2-1}\hat f(n)z^n.$$ 
We then easily get the upper estimate of \eqref{small values}. To show the lower one, for $n\in\ent^d$, $n\neq\mathbf{0}$ we have
 $$|n|_2r^{|n|_2-1}\hat f(n)z^n=\int_{\T^d} \frac{d}{dr}\mathbb{P}_r[f](z\cdot w)w^{-n}dm(w).$$
Let $H=L_2((r_0,\,r_1); (1-r)dr)$. It is clear that for any $z\in\T^d$ we have
 $$\big(\int_{r_0}^{r_1}\big||n|_2r^{|n|_2-1}\hat f(n)z^n\big|^2(1-r)dr\big)^{1/2}\approx |\hat f(n)|.$$
Then by the triangle inequality in the column $L_p$-space $L_p(L_\8(\T^d)\overline\ot\M; H^c)$, we deduce 
 $$\|\hat f(n)\|_{L_p(\M)}\lesssim  
 \int_{\T^d} \Big( \int_{\T^d}\|G'_c(f)(z\cdot w)\|^p_{L_p(\M)}dm(z)\Big)^{1/p}dm(w)=\|G'_c(f)\|_{L_p(\T^d;L_p(\M))}.$$
 Thus the claim is proved. Using \eqref{small values} twice, we get
  \be\begin{split}
  \|G_c(f)\|_{p}
  &\le \Big\| \Big ( \int_{0}^{r_0} \Big | \frac{d}{dr}\mathbb{P}_r[f]\Big |^2(1-r)dr \Big )^{1/2}\Big\|_{p}
  +\|G_{c,r_0}(f)\|_{p}\\
  &\lesssim \sup_{n\in\ent^d,n\neq\mathbf{0}}\|\hat f(n)\|_{p} +\|G_{c,r_0}(f)\|_{p}
  \lesssim \|G_{c,r_0}(f)\|_{p}.
  \end{split}\ee
  Similarly, we show that only small values of $\e$ matter in \eqref{tg} and \eqref{tg periodic}. Namely for $0<\e_0<\8$ letting
  $$\wt G_{c,\e_0}(f) (s)=  
  \Big ( \int_0^{\e_0} \Big | \frac{d}{d\e}\mathbb{Q}_\e[f](s) \Big |^2\e d\e \Big )^{1/2},\quad s\in\I^d,$$
 we have
   $$\|\wt G_{c}(f) \|_{p}\approx \|\wt G_{c,\e_0}(f) \|_{p}.$$
 Now it is easy to finish the proof of the lemma. Indeed, using the change of variables $r=e^{-2\pi\e}$, we get 
 \be\begin{split}
 G_{c,r_0}(f)(z) 
 &=  \frac1{2\pi}\Big ( \int_0^{\e_0} \Big | \frac{d}{d\e}\Q_\e[f] (s)\Big |^2e^{2\pi\e}(1-e^{-2\pi\e})d\e \Big )^{1/2}\\
 &\approx \Big ( \int_0^{\e_0} \Big | \frac{d}{d\e}\Q_\e[f] (s)\Big |^2\e d\e \Big )^{1/2}=\wt G_{c, \e_0}[f](s).
 \end{split}\ee
Together with the previous equivalences, this implies the desired assertion.
\end{proof}

We will also need the Lusin area integral function. For $\a>1$ and  $z\in\T^d$, let $\D_\a(z)$ be the Stoltz domain with vertex $z$ and aperture $\a$ (recalling that $|w|$ denotes the Euclidean norm):
 $$\D_\a(z)=\{w\in\com^d\;:\; |z-w|\le\a(1-|w|)\}.$$
 For $f\in L_\8(\T^d)\overline\ot\M$ define
 \beq\label{torus S}
 S^\a_c(f) (z)=  \Big (\int_{\D_\a(z)} \Big | \frac{d}{dr}\mathbb{P}_r[f](rw)\Big |^2\frac{dm(w)dr}{(1-r)^{d-1} }\Big )^{1/2}\,,
 \quad z\in\T^d,
 \eeq
where the integral is taken  over $\D_\a(z)$ with respect to $rw\in\D_\a(z)$ with $0\le r<1$ and $w\in\T^d$ (recalling that $dm$ is Haar measure of $\T^d$).

Like for the $g$-function, we will transfer $S^\a_c(f) $ to the usual area integral function on $\real^d$.  For $\b>0$ and  $s\in\real^d$ let 
 $$\Ga_\b(s)=\{(t,\e)\in\real^d\times\real_+\;:\; |t-s|\le\b\e\}.$$
 Let  $f\in L_\8(\I^d)\overline\ot \M$ and $\tilde f$ be its periodic extension to $\real^d$. Define
 \begin{eqnarray}\label{tS}
 \begin{array}{ccl}\begin{split}
 \wt S^\b_c(f) (s)
 &= \begin{displaystyle}
  \Big ( \int_{\Ga_\b(s)} \Big | \frac{d}{d\e}\mathbb{Q}_\e[f](t) \Big |^2 \frac{dtd\e}{\e^{d-1}} \Big )^{1/2}
  \end{displaystyle}\\
 &=\begin{displaystyle}
   \Big ( \int_{\Ga_\b(s)} \Big | \frac{d}{d\e}\f_\e[\tilde f](t) \Big |^2 \frac{dtd\e}{\e^{d-1}} \Big )^{1/2}
 =\wt S^\b_c(\tilde f) (s) \end{displaystyle}.
  \end{split}\end{array}
 \end{eqnarray}

The following is the analogue of Lemma~\ref{Hardy cube} for the Lusin square functions.

\begin{lem}\label{Hardy S}
 Let $\a>1$ and $\b>0$. Let $f\in L_\8(\T^d)\overline\ot\M$. Then
 $$\big\|S^\a_c(f) \big\|_{L_p(\T^d;L_p(\M))}\approx 
 \big\|\wt S^\b_c(f) \big\|_{L_p(\I^d;L_p(\M))}$$
with equivalence constants depending only on $d$ and $\a, \b$. Moreover, the norms above are independent of $\a$ and $\b$ up to equivalence.
 \end{lem}
 
 \begin{proof}
  This proof is similar to that of Lemma~\ref{Hardy cube}. For $0<r_0<1$  we introduce the truncated Stoltz domain:
  $$\D_{\a, r_0}(z)=\{w\in\com^d\;:\; |z-w|\le\a(1-|w|),\, r_0<|w|<1\}.$$
 Also for  $\e_0>0$ set
  $$\Ga_{\b,\e_0}(s)=\{(t,\e)\in\real^d\times\real_+\;:\; |t-s|\le\b\e,\,\e<\e_0\}.$$
 The corresponding truncated square functions are
  $$S^\a_{c, r_0}(f) (z)=  
  \Big (\int_{\D_{\a,r_0}(z)} \Big | \frac{d}{dr}\mathbb{P}_r[f](rw)\Big |^2\frac{dm(w)dr}{(1-r)^{d-1} }\Big )^{1/2}$$
 and 
  $$\wt S^\b_{c,\e_0}(f) (s)=\Big ( \int_{\Ga_{\b,\e_0}(s)} \Big | \frac{d}{d\e}\f_\e[\tilde f](t) \Big |^2 \frac{dtd\e}{\e^{d-1}} \Big )^{1/2}.$$
Then by the reasoning in the proof of  Lemma~\ref{Hardy cube}, we have
  $$\big\|S^\a_c(f) \big\|_{p}\approx 
 \big\|S^\a_{c, r_0}(f) \big\|_{p}$$
and a similar equivalence for $\wt S^\b_c(f)$. On the other hand, it is easy to see that for any $\a>1$ and $0<r_0<1$ there exist $\b_1, \b_2>0$ and $\e_1, \e_2>0$ such that under the change of variables $r=e^{-2\pi\e}$ and $w=e^{-2\pi{\rm i}t}$
 $$\Ga_{\b_1,\e_1}(s)\subset \D_{\a, r_0}(z)\subset\Ga_{\b_2,\e_2}(s),\quad\forall\; z=e^{-2\pi{\rm i}s}\in\T^d.$$
Conversely, every truncated cone $\Ga_{\b,\e_0}(s)$ is located between two truncated Stoltz domains. Then the argument at the end of the proof of  Lemma~\ref{Hardy cube} implies 
 $$\wt S^{\b_1}_{c,\e_1}(f) (s)\lesssim S^\a_{c, r_0}(f) (z)\lesssim \wt S^{\b_2}_{c,\e_2}(f) (s);$$
 whence
 $$\big\|\wt S^{\b_1}_{c,\e_1}(f)\big\|_{L_p(\I^d;L_p(\M))}
 \lesssim \big\|S^\a_{c, r_0}(f)\big\|_{L_p(\T^d;L_p(\M))}\lesssim 
 \big\|\wt S^{\b_2}_{c,\e_2}(f)\big\|_{L_p(\I^d;L_p(\M))}.$$
However, standard arguments in harmonic analysis  show that
 $$\big\|\wt S^{\b_1}_{c}(f)\big\|_{L_p(\I^d;L_p(\M))}\approx \big\|\wt S^{\b_2}_{c}(f)\big\|_{L_p(\I^d;L_p(\M))},$$
where the equivalence constants depend on $d$ and $\b_1,\b_2$ (cf. e.g., \cite{CMS}).  Therefore, we deduce the first equivalence assertion of the lemma. The second part then follows too.
 \end{proof}

Now we can show that the results of \cite{Mei2007} remain valid for $\T^d$ too. We state only those  relevant to Theorem~\ref{H1-BMO}. In the following statement, the row and mixture Hardy/BMO spaces  are defined in the usual way, and
 $ S_{c}(f)=S^2_{c}(f)$, $\wt S_{c}(f)=\wt S^1_{c}(f).$

\begin{thm}\label{H1-BMO bis}
  \begin{enumerate}[\rm i)]

\item The dual space of $\mathrm{H}_1^c(\T^d; \M)$  coincides isomorphically with $\mathrm{BMO}^c(\T^d; \M)$ with the natural duality bracket. The same assertion holds for the row and mixture spaces.

\item Let $1\le p<\8$. Then for any $f\in L_\8(\T^d)\overline\ot\M$
 $$\big\|G_c(f) \big\|_{p}\approx \big\|S_c(f) \big\|_{p}$$
with relevant constants depending only on $d$. Consequently, the two square functions $G_c$ and $S_c$ define a same Hardy space.
   
\item  Let $1<p<\8$. Then $\mathrm{H}_p(\T^d; \M)=L_p(\T^d; L_p(\M))$ with equivalent norms.

\item  Let $1<p<\8$. Then
 $$(\mathrm{BMO}^c(\T^d; \M),\; \mathrm{H}^c_1(\T^d; \M))_{1/p}
 =\mathrm{H}^c_p(\T^d; \M)
 =(\mathrm{BMO}^c(\T^d; \M),\; \mathrm{H}^c_1(\T^d; \M))_{1/p, p}\,.$$

\item  Let $X_0\in\{\mathrm{BMO}(\T^d; \M),\; L_\8(\T^d; L_p(\M))\}$, $X_1\in\{\mathrm{H}_1(\T^d; \M),\; L_1(\T^d; \M)\}$. Then for any $1<p<\8$
  $$(X_0,\; X_1)_{1/p}=L_p(\T^d; \M)=(X_0,\; X_1)_{1/p, p}\,.$$
 \end{enumerate}
  \end{thm}

\begin{proof}
 By the identification $\T^d\cong\I^d$ and Lemmas~\ref{BMO cube}, \ref{Hardy cube}  and \ref{Hardy S}, it suffices to prove this theorem with $\I^d$ instead of $\T^d$. The geometry of $\I^d$ is closer to that of $\real^d$. However, what makes our arguments parallel to those of \cite{Mei2007}  is the use of periodic functions. This periodization puts the arguments of \cite{Mei2007} directly at our disposal.  For any function $f$ on $\I^d$ with periodic extension $\tilde f$ to $\real^d$, by \eqref{tg=tg periodic} and  \eqref{tS}, we have
  $$\wt G_c(f)=\wt G_c(\tilde f)\quad\text{and}\quad \wt S_c(f)=\wt S_c(\tilde f)\;\text{ on }\; \I^d.$$ 
Note that the two square functions on the right are exactly those introduced in \cite{Mei2007} by using the Poisson kernel on $\real^d$. The only difference compared with \cite{Mei2007} is that the $L_p$-norm of these square functions are now taken in $L_p(\I^d; L_p(\M))$ instead of $L_p(\real^d; L_p(\M))$. In other words, the integral is now taken on $\I^d$ instead of $\real^d$.  On the other hand, by Lemmas~\ref{BMO cube} and \ref{periodic BMO}, the map $f\mapsto\tilde f$ is an isomorphic embedding of $\mathrm{BMO}^c(\I^d; \M)$ into $\mathrm{BMO}^c(\real^d; \M)$. It is now easy to see that the proof of \cite[Theorem~2.4]{Mei2007} is valid for periodic functions and integration on $\I^d$. Hence, we get the duality result in part i) and the equivalence for $p=1$ in part ii). In the same way, we prove the periodic analogue of \cite[Theorem~4.4]{Mei2007}, which implies the remaining case of  ii). The reduction to dyadic martingales of  \cite{Mei2007} is clearly available in our present setting. The dyadic decomposition is now made in $\I^d$ (or equivalently, $\T^d$). In this way, we reduce parts iii)-v) to the martingale case as in  \cite{Mei2007}.   The  verification of all details is, however,   tedious and  lengthy, so it is more reasonable to skip it here.
\end{proof}

\begin{rk}
 It is stated in the final remark of \cite[Chapter~2]{Mei2007} that the relevant constants in part i) above for $\real^d$ are independent of $d$. This does not seem true. In fact, all constants appearing in Theorem~\ref{H1-BMO bis} depend on $d$ (and on $p$ too), except those in part iii) since the semigroup argument described in the paragraph following Theorem~\ref{H1-BMO} yields equivalence constants depending only on $p$. The same comment applies to Theorem~\ref{H1-BMO} too. However, the constants there are independent of the given skew matrix $\theta$. 
 \end{rk}

 \begin{rk}
 The $\mathrm{H}_1$-BMO duality in Theorem~\ref{H1-BMO bis} and Lemma~\ref{BMO cube} imply that $\mathrm{H}_1^c(\T^d; \M)$ admits an atomic decomposition like in the case of $\real^d$. We refer the interested reader to \cite{Mei2007} for more details.
 \end{rk}

Armed with Theorem~\ref{H1-BMO bis} and transference, it is easy to prove Theorem~\ref{H1-BMO}. To this end we still  require the following simple lemma.

\begin{lem}\label{complementation}
 The map $x\mapsto\tilde x$ in Corollary~\ref{prop:TransLp} extends to an isometric embedding from $\mathrm{H}^c_p(\mathbb{T}^d_{\theta})$ into  $\mathrm{H}^c_p(\T^d; \T^d_{\theta})$ for any $1\le p<\8$ and from  $\mathrm{BMO}_c(\mathbb{T}_{\theta}^d)$ into  $\mathrm{BMO}_c(\T^d; \T^d_{\theta})$. Moreover, the images of this embedding are $1$-complemented in their respective spaces.
 \end{lem}

\begin{proof}
 The first part  follows immediately from the identity $\mathbb{P}_r[\tilde x]=\wt{\mathbb P_r[x]}$ for any $x\in L_1(\T^d_\theta)$.  Identifying $\wt{\T^d_\theta}$ with $\T^d_\theta$, we see that  the conditional expectation $\mathbb E$ from $L_\8(\T^d)\overline\ot\T^d_\theta$ onto $\T^d_\theta$ extends to a contractive projection from $\mathrm{H}^c_p(\T^d; \T^d_{\theta})$ onto $\mathrm{H}^c_p(\T^d_{\theta})$ and from  $\mathrm{BMO}_c(\T^d; \T^d_{\theta})$ onto  $\mathrm{BMO}_c(\T^d_{\theta})$.  This yields the second part.
 \end{proof}

\noindent{\it The proof of Theorem~\ref{H1-BMO}}. It is now clear that Theorem~\ref{H1-BMO} follows immediately from Theorem~\ref{H1-BMO bis} (with $\M=\T^d_{\theta}$) and Lemma~\ref{complementation}. \hfill $\Box$

\begin{rk}
 Since $\T^d_\theta\subset \mathrm{BMO}(\mathbb{T}_{\theta}^d)$, part ii) of  Theorem~\ref{H1-BMO} implies that $\mathrm{H}_1(\mathbb{T}^d_{\theta})\subset L_1(\T^d_\theta)$ and
  $$\|x\|_1\le C\|x\|_{\mathrm{H}_1},\quad \forall\; x\in \mathrm{H}_1(\mathbb{T}^d_{\theta}).$$
\end{rk}

\medskip\n{\bf Acknowledgements.} We are grateful to Tao Mei and Eric Ricard for useful discussions.

\bigskip


\end{document}